\documentclass[10pt,letterpaper]{amsart}
\usepackage[T1]{fontenc}
\usepackage{lmodern,microtype}
\usepackage[margin=0.9in]{geometry}
\usepackage{amsmath,amssymb,amsthm,amscd,mathtools}
\usepackage{enumitem,booktabs}
\usepackage{xcolor}
\usepackage[pagebackref,colorlinks=true,linkcolor=magenta!75!black,citecolor=violet,urlcolor=blue!65!black]{hyperref}
\hypersetup{pdftitle={Projective Decomposition and Intersection Cohomology of Rigid
Spaces},pdfauthor={Jiawei Yang}}
\setlist[enumerate]{label=\textup{(\roman*)},leftmargin=2.2em}
\newtheorem{theorem}{Theorem}[section]
\newtheorem{proposition}[theorem]{Proposition}
\newtheorem{lemma}[theorem]{Lemma}
\newtheorem{corollary}[theorem]{Corollary}

\numberwithin{equation}{section}
\theoremstyle{definition}
\newtheorem{definition}[theorem]{Definition}
\newtheorem{example}[theorem]{Example}
\theoremstyle{remark}
\newtheorem{remark}[theorem]{Remark}
\newcommand{\Q}{\mathbf Q}
\newcommand{\Z}{\mathbf Z}
\newcommand{\IC}{\operatorname{IC}}
\newcommand{\Perv}{\operatorname{Perv}}
\newcommand{\Spa}{\operatorname{Spa}}
\newcommand{\Spec}{\operatorname{Spec}}
\newcommand{\supp}{\operatorname{supp}}
\newcommand{\pH}{\prescript{p}{}{\mathcal H}}
\newcommand{\Db}{D^b_{\mathrm{zc}}}
\newcommand{\et}{\mathrm{\acute et}}
\newcommand{\cL}{\mathcal L}
\newcommand{\cH}{\mathcal H}
\newcommand{\cE}{\mathcal E}
\newcommand{\cO}{\mathcal O}
\newcommand{\an}{\mathrm{an}}
\newcommand{\IH}{\operatorname{IH}}
\title[Projective decomposition of rigid spaces]{Projective Decomposition and Intersection
Cohomology of Rigid Spaces}
\author{Jiawei Yang}
\address{School of Mathematical Sciences, Xiamen University, Xiamen 361005, China}
\email{yangjw@stu.xmu.edu.cn}
\date{}
\begin{document}
\begin{abstract}
We prove a rigid analytic analogue of the Beilinson--Bernstein--Deligne decomposition theorem for
projective morphisms over \(p\)-adic fields, with \(\Q_\ell\)-coefficients for every prime \(\ell\),
including \(\ell=p\). As an application, we show that the \(p\)-adic intersection
cohomology of every proper rigid analytic space is de Rham. We also establish an independent
criterion for de Rham intersection cohomology with local coefficients in terms of the dimension of
the boundary, yielding further results for curves, surfaces, and finite covers. For finite
coefficients of characteristic different from \(p\), we obtain isotropic analytic upper bounds
for singular support and prove that singular support is analytic and Lagrangian for complexes
supported in dimension at most one.
\end{abstract}
\maketitle
\setcounter{tocdepth}{1}
\tableofcontents

\section{Introduction}

The decomposition theorem of Beilinson, Bernstein, and Deligne expresses the direct image of an
intersection complex under a projective morphism of algebraic varieties as a direct sum of shifted
intersection complexes \cite{BBD}. In rigid analytic geometry, Bhatt and Hansen developed the
formalism of Zariski-constructible sheaves and perverse \(t\)-structures, and conjectured an
analogous decomposition theorem for projective morphisms \cite[Conjecture 4.17]{BH}. Their example
of a proper smooth morphism from a nonarchimedean Hopf surface to the projective line shows that
projectivity cannot be replaced by properness \cite[\S4.5]{BH}.

We prove the rigid analytic analogue for every coefficient prime, including the residue
characteristic. Let \(K/\Q_p\) be a finite extension, \(C=\widehat{\overline K}\),
\(G_K=\operatorname{Gal}(\overline K/K)\), and \(\ell\) any prime. We use the rational constructible
derived categories and middle perverse \(t\)-structures of \cite{BH}, with \(\cH^q\) and \(\pH^a\)
denoting ordinary and perverse cohomology sheaves. Ambient spaces for
rational coefficients are always assumed to be quasi-compact and quasi-separated (qcqs).

For a reduced pure-dimensional rigid space \(W\), a dense open immersion \(j:U\hookrightarrow W\)
with \(U\) smooth, and a finite-rank \(\Q_\ell\)-local system \(\cL\) on \(U\) for which
\(j_{!*}\) is defined, the intersection complex is \(\IC_W(\cL)=j_{!*}(\cL[\dim W])\), where
\(j_{!*}\) denotes intermediate
extension as in Section~\ref{sec:formal}; put \(\IC_W=\IC_W(\Q_\ell)\). For a smooth locally closed
stratum \(U\subset Y\), with reduced closure \(\overline U\) and inclusions
\(j:U\hookrightarrow\overline U\) and \(i:\overline U\hookrightarrow Y\), we write \(\IC_{\overline
U}(\cL)=i_*j_{!*}(\cL[\dim U])\). For non-equidimensional spaces, we take the direct sum of the
intersection complexes of their irreducible components, pushed forward to the space.

\begin{theorem}\label{thm:general}
Let \(f:X\to Y\) be a projective morphism of rigid analytic spaces over \(K\), with \(Y\)
quasi-compact and quasi-separated. For every prime \(\ell\), including \(\ell=p\), there is an
isomorphism
\begin{equation}\label{eq:main}
 Rf_*\IC_X\simeq\bigoplus_a\pH^a(Rf_*\IC_X)[-a],
\end{equation}
and each perverse cohomology sheaf is a finite direct sum of intermediate extensions of finite-rank
\(\Q_\ell\)-local systems on smooth Zariski-locally closed strata of \(Y\).
\end{theorem}

This proves \cite[Conjecture 4.17]{BH} in the stated setting. The local systems occurring here need
not be semisimple. Indeed, we construct a smooth projective family of Tate elliptic curves over an
annulus whose first cohomology local system remains nonsemisimple after extension to \(C\).

The principal difficulty is to transfer the algebraic decomposition theorem to rigid analytic spaces
while preserving perversity and the property of being an intermediate extension. Our approach begins
with relative rigid GAGA and Popescu approximation, which locally realize a projective rigid
morphism as the restriction of an algebraic projective morphism along an analytic section of a
smooth family. The section can be expressed locally by successive hypersurface cuts. Although the
algebraic decomposition theorem applies before restriction, restriction to an arbitrary hypersurface
need not preserve perversity or intermediate extensions.

We simultaneously approximate the projective morphism, a prescribed closed boundary, and a
resolution diagram of its inverse image. The approximation is arranged so that each successive cut
is smooth on all the relevant smooth sources. For \(\ell\ne p\), the results of Zhou on vanishing
cycles \cite{ZhouVC}, together with abstract blow-up triangles, show that the vanishing cycles of
both the total direct image and its boundary restriction are zero. Uniform torsion estimates extend
this vanishing to rational derived summands. Natural Gysin morphisms, constructed using the theory
of cycle classes and duality developed by Li, Reinecke, and Zavyalov \cite{LRZ}, provide the
corresponding compatibility with Verdier duality.

The case \(\ell=p\) requires a separate argument. We use the same algebraic approximation and
analytic graph for all coefficient primes. Comparison between algebraic and Betti cohomology shows
that the dimensions of ordinary stalk cohomology of the relevant algebraic perverse sheaves and
their duals are independent of \(\ell\). The characterization of the perverse \(t\)-structure in
terms of dimensions of supports transfers the upper perverse bounds to \(p\), while Gysin duality
supplies the lower bounds. This yields relative hard Lefschetz and the decomposition theorem for
smooth sources. Resolution of singularities and stability under derived retracts then give the
general result.

A first consequence concerns \(p\)-adic Hodge theory. We write
\(\IH^m(X_C,\Q_p)=H^m(X_C,\IC_{X,\Q_p}|_{X_C})\).
\begin{corollary}\label{cor:intro-constant-derham}
For every proper rigid analytic space \(X/K\) and every integer \(m\), the intersection cohomology
group \(\IH^m(X_C,\Q_p)\) is a de Rham representation of \(G_K\).
\end{corollary}
This establishes the assertion for proper spaces in \cite[Conjecture 4.16(1)]{BH}. The proof
realizes the intersection complex as a derived direct summand of the direct image of the constant
sheaf under a projective resolution. The assertion then follows from \(p\)-adic Hodge comparison for
smooth proper spaces.

We also obtain a de Rham result for local coefficients that is independent of the projective
decomposition theorem. Let \(X/K\) be proper, reduced, and pure of dimension \(d\), let
\(j:U\hookrightarrow X\) be a dense Zariski-open immersion with \(U\) smooth, and let \(\mathbb L\)
be a de Rham \(\Z_p\)-local system on \(U\). Put \(L=\mathbb L[1/p]\). We normalize intersection
cohomology by \(\IH^m(X_C,L)=H^m(X_C,(j_{!*}(L[d]))|_{X_C})\).

\noindent\textbf{Theorem~\ref{thm:boundary-derham}.}
\emph{If \(Z=X\setminus U\) is nonempty of dimension \(b\), then \(\IH^m(X_C,L)\) is de Rham
whenever \(m\le-b\) or \(m\ge b\). More precisely, the natural map from intersection cohomology to
\(H^{m+d}(U_C,L)\) is an isomorphism for \(m<-b\) and injective for \(m=-b\); the map from
\(H_c^{m+d}(U_C,L)\) to intersection cohomology is an isomorphism for \(m>b\) and surjective for
\(m=b\).}

The proof combines logarithmic de Rham comparison on \(U\) with the strict support and cosupport
conditions defining intermediate extension. In particular, a finite boundary gives the de Rham
property in every degree. Normalization and finite covers yield further results for curves,
surfaces, and local systems with finite geometric or divisorial monodromy.

We also study singular support on a smooth rigid space \(M/C\), with coefficients in a finite field
\(\Lambda\) of characteristic \(\ell\ne p\). For a bounded Zariski-constructible complex \(F\) of
\(\Lambda\)-sheaves, write \(SS(F)\subset T^*M\) for its singular support in the sense of
\cite[\S8]{ZhouMicro}. If \(M\) is affinoid and pure of dimension \(n\), then \(SS(F)\) is contained
in a closed conic analytic subset whose irreducible components are isotropic on their smooth loci
and have dimension at most \(n\), by Proposition~\ref{prop:isotropic-container}. If \(M\) is
quasi-compact and pure-dimensional and \(F\) has support of dimension at most one, then \(SS(F)\)
is a finite union of closures of conormal bundles of smooth curves and cotangent fibers at classical
points. In particular, it is an analytic Lagrangian subset of \(T^*M\)
(Proposition~\ref{prop:curve-supported-ss}).

The paper is organized as follows. Section~\ref{sec:approximation} develops simultaneous algebraic
approximation, boundary resolution, perverse restriction, and Gysin duality.
Section~\ref{sec:decomposition} proves the projective decomposition theorem for all coefficient
primes. Section~\ref{sec:arithmetic} establishes the de Rham results, and Section~\ref{sec:examples}
discusses examples and proves the results on singular support.

\section{Approximation and boundary restriction}\label{sec:approximation}

\subsection{Coefficients and splitting criteria}\label{sec:formal}

For a coefficient ring \(\Lambda\), write \(\Db(M,\Lambda)\) for the bounded derived category of
Zariski-constructible complexes of \cite{BH}, and \(\Perv(M,\Lambda)\) for its perverse heart. We
write \(\mathbb D_{M,\Lambda}\) for Verdier duality, omitting the space or coefficient ring when
they are clear. Closures of supports are taken in the Zariski topology.
For a qcqs space \(M\), \cite[Definition 4.8]{BH} defines the derived
category of Zariski-constructible complexes with \(\Q_\ell\)-coefficients
by \(\Db(M,\Q_\ell):=\Db(M,\Z_\ell)\otimes_{\Z_\ell}\Q_\ell\).
Writing \(E_{\Q_\ell}=E[1/\ell]\), we have, for \(A,B\in\Db(M,\Z_\ell)\),
\[
 \operatorname{Hom}_{\Q_\ell}(A_{\Q_\ell},B_{\Q_\ell})
 \cong \operatorname{Hom}_{\Z_\ell}(A,B)\otimes_{\Z_\ell}\Q_\ell,
\]
where the Hom groups are taken in the corresponding derived categories.

As in \cite[Construction 4.3 and Proposition 4.6]{BH}, there are two
integral perverse \(t\)-structures, denoted by \(p\) and \(p^+\),
which are exchanged by Verdier duality. They induce the same rational
perverse \(t\)-structure \cite[Theorem 4.11]{BH}.
For \(n\ge1\) and a complex \(A\) with \(\Z_\ell\)-coefficients,
we write \(A/\ell^n:=A\otimes_{\Z_\ell}^{L}\Z/\ell^n\) for its derived reduction modulo \(\ell^n\).
Since \(\Z/\ell^n\) has
Tor-amplitude contained in \([-1,0]\) over \(\Z_\ell\), this reduction
sends \(D^{[a,b]}\) into \(D^{[a-1,b]}\) for ordinary cohomology.

Under the hypotheses of \cite[Theorem 3.36]{BH}, reduction modulo
\(\ell^n\) commutes with proper pushforward, ordinary pullback,
extraordinary pullback along a closed immersion, and Verdier duality;
see also \cite[Remarks 3.33 and 3.37]{BH}. These operations also
commute with extension of the base field. We use these compatibilities
to work first with finite coefficients and then pass to
\(\Z_\ell\) and \(\Q_\ell\).
In particular, Li--Reinecke--Zavyalov \cite[Theorem 7.5.18]{LRZ}
prove that proper pushforward is compatible with Verdier duality
for finite coefficients, including when \(\ell=p\).
They also show that, if \(W\) is smooth of pure dimension \(s\), then
its dualizing complex with \(\Z/\ell^n\)-coefficients is
\((\Z/\ell^n)(s)[2s]\) \cite[Corollary 7.1.2]{LRZ}.

The trace and divisor class maps commute with reduction of coefficients. Their compatible systems
modulo \(\ell^n\) therefore define the corresponding \(\ell\)-adic maps. By derived Nakayama
\cite[Tag 0G1U]{Stacks}, a morphism between
derived \(\ell\)-complete complexes is an isomorphism if its derived
reduction modulo \(\ell\) is an isomorphism; the integral isomorphism
can then be rationalized. We apply this criterion to the purity map
in Lemma~\ref{lem:purity} and to proper duality in
Lemma~\ref{lem:adic-proper-duality}.

Analytification is perverse \(t\)-exact for algebraic constructible
complexes with finite coefficients \cite[Theorem 4.2(4)]{BH}.
For \(\Q_\ell\)-coefficients, we choose an integral algebraic model,
analytify it, restrict it to the qcqs analytic open neighborhood under consideration,
and then invert \(\ell\); this also applies when the full analytification
is not quasi-compact. For an algebraic perverse object, we may
first replace the model by its zeroth perverse cohomology with integral coefficients
without changing its rationalization. Analytification with finite coefficients,
followed by restriction to this analytic open, commutes with perverse
truncation and Verdier duality. These compatibilities preserve the
defining conditions for the integral \(p\)- and \(p^+\)-structures.
After rationalization the two structures agree, so the resulting
object is perverse \cite[Construction 4.3 and proof of Theorem 4.11]{BH}.

For a Zariski-open immersion \(j:U\hookrightarrow M\) with \(M\)
qcqs, we take intermediate extensions of extendable objects in the category of perverse sheaves with
rational coefficients on \(M\), as in
\cite[Theorems 4.2(5) and 4.11]{BH}. In particular, this applies to
restrictions \(j^*P\) with \(P\in\Perv(M,\Q_\ell)\).

\begin{lemma}\label{lem:deligne}
Let \(A\) be an object in a derived category with a bounded \(t\)-structure preserved by the Tate
twist \((1)\). Let \(\eta:A\to A(1)[2]\) be a morphism whose \(r\)-th power induces
an isomorphism
\[
 \eta^r:H^{-r}(A)\xrightarrow{\sim}H^r(A)(r)
 \qquad(r\ge0).
\]
Then \(A\simeq\bigoplus_a H^a(A)[-a]\).
\end{lemma}
\begin{proof}
This is Deligne's splitting criterion \cite{Deligne68,Deligne94}.
Choose \(n\ge0\) such that \(H^a(A)=0\) for \(a<-n\) and \(a>n\).
If \(n=0\), the assertion is immediate. Suppose \(n>0\), and write
\(B:=H^{-n}(A)\) and \(E:=H^n(A)\). The bottom and top truncations
give natural maps \(u:B[n]\to A\) and \(v:A\to E[-n]\), respectively.
The hard Lefschetz isomorphism
\(\lambda:=H^{-n}(\eta^n):B\xrightarrow{\sim}E(n)\) allows us to
construct a retraction \(r:A\to B[n]\) and a section
\(s:E[-n]\to A\), with \(ru=\operatorname{id}_{B[n]}\) and
\(vs=\operatorname{id}_{E[-n]}\). Define \(r\) as the composite
\[
 A\xrightarrow{\eta^n}A(n)[2n]
 \xrightarrow{v(n)[2n]}E(n)[n]
 \xrightarrow{\lambda^{-1}[n]}B[n].
\]
To define \(s\), take the composite
\[
 E(n)[n]\xrightarrow{\lambda^{-1}[n]}B[n]
 \xrightarrow{u}A\xrightarrow{\eta^n}A(n)[2n]
\]
and apply the twist and shift \((-n)[-2n]\).

Since \(n>0\), we have \(vu=0\). Replacing \(s\) by
\(s':=s-urs\) gives \(rs'=0\) and
\(vs'=\operatorname{id}_{E[-n]}\). Thus \(e_-:=ur\) and
\(e_+:=s'v\) are idempotents with \(e_-e_+=e_+e_-=0\).
These orthogonal projectors split off the bottom and top cohomology
objects simultaneously:
\[
 A\simeq B[n]\oplus A'\oplus E[-n],
\]
where \(A'\) has cohomology only in degrees
\(-n+1,\ldots,n-1\).

Projecting \(\eta\) to \(A'\) gives a Lefschetz operator on \(A'\).
Its cohomology maps between the remaining degrees agree with those
of \(\eta\). For \(a<n\), the powers from degree \(-a\) to degree
\(a\) involve only these remaining degrees, so the hard Lefschetz
isomorphisms still hold. Applying the same argument to \(A'\) proves
the assertion by induction.
\end{proof}

\begin{lemma}\label{lem:boundary}
Let \(M\) be a qcqs rigid analytic space, let
\(P\in\Perv(M,\Q_\ell)\) be a perverse sheaf on \(M\), and let
\(i:Z\hookrightarrow M\) be a closed immersion. Write
\(j:U:=M\setminus Z\hookrightarrow M\) for the complementary
open immersion. Define
\(B:=\pH^0(i^!P)\) and \(C:=\pH^0(i^*P)\). The canonical morphisms
\(i_*B\xrightarrow{a}P\xrightarrow{b}i_*C\)
define a morphism \(c_Z(P):B\to C\) by
\(b\circ a=i_*c_Z(P)\). Then \(c_Z(P)\) is an isomorphism if
and only if
\begin{equation}\label{eq:boundary}
 P\simeq j_{!*}j^*P\oplus i_*B.
\end{equation}
When these conditions hold, the projector onto \(i_*B\) is
\(e:=a\circ i_*\bigl(c_Z(P)^{-1}\bigr)\circ b\).
\end{lemma}
\begin{proof}
Since \(P\) is perverse, the \(t\)-exactness properties of closed
immersions give \(i^!P\in{}^pD^{\ge0}(Z)\) and
\(i^*P\in{}^pD^{\le0}(Z)\).
For every \(Q\in\Perv(Z,\Q_\ell)\), adjunction and perverse
truncation yield natural isomorphisms
\[
\begin{aligned}
 \operatorname{Hom}_M(i_*Q,P)
 &\simeq\operatorname{Hom}_Z(Q,i^!P)
 \simeq\operatorname{Hom}_Z(Q,B),\\
 \operatorname{Hom}_M(P,i_*Q)
 &\simeq\operatorname{Hom}_Z(i^*P,Q)
 \simeq\operatorname{Hom}_Z(C,Q).
\end{aligned}
\]
Thus every morphism from a perverse sheaf supported on
\(Z\) to \(P\) factors uniquely through \(a\), and every morphism
from \(P\) to a perverse sheaf supported on \(Z\) factors uniquely
through \(b\). The kernel of \(a\) and the cokernel of \(b\) are
supported on \(Z\), so these universal properties force both to
vanish. Thus \(a\) is a monomorphism and \(b\) is an epimorphism.
Consequently, \(i_*B\) is the maximal subobject of \(P\) supported
on \(Z\), while \(i_*C\) is its maximal quotient supported there.

Suppose first that \(c_Z(P)\) is an isomorphism, and define
\(r:=i_*\bigl(c_Z(P)^{-1}\bigr)\circ b:P\to i_*B\).
Since \(ba=i_*c_Z(P)\), we have
\(ra=\operatorname{id}_{i_*B}\). Hence \(a\) is a split
monomorphism, and \(e=ar\) is the projector onto \(i_*B\).
Setting \(I:=\ker r=\ker b\) gives
\(P\simeq i_*B\oplus I\). The summand \(I\) has no nonzero
subobject supported on \(Z\). Indeed, such a subobject would also
be a subobject of \(P\), hence would be contained in the maximal
supported subobject \(i_*B\), whereas \(I\cap i_*B=0\).
To exclude quotients supported on \(Z\), let
\(q:I\to i_*Q\) be any morphism with \(Q\in\Perv(Z,\Q_\ell)\).
Composing with the projection \(P\to I\) gives a morphism
\(P\to i_*Q\). By the universal property of \(b\), it factors
through \(b:P\to i_*C\). Since \(b|_I=0\), we have \(q=0\).
Finally, \(j^*i_*B=0\), so the decomposition gives
\(j^*I\simeq j^*P\). Thus \(I\) extends \(j^*P\) and has neither
nonzero subobjects nor nonzero quotients supported on \(Z\).
This characterizes the intermediate extension, giving
\(I\simeq j_{!*}j^*P\) and the decomposition \eqref{eq:boundary}.

Conversely, suppose that \(P\simeq j_{!*}j^*P\oplus i_*B\).
By the defining property of intermediate extension,
\(j_{!*}j^*P\) has no nonzero subobjects or quotients supported
on \(Z\). Hence
\(\pH^0(i^!j_{!*}j^*P)=\pH^0(i^*j_{!*}j^*P)=0\).
Full faithfulness of \(i_*\) gives \(i^!i_*B\simeq B\) and
\(i^*i_*B\simeq B\).
Under these identifications, the boundary morphism associated with
\(i_*B\) is \(\operatorname{id}_B\). The boundary construction
commutes with finite direct sums, so \(c_Z(P)\) is an isomorphism.

\end{proof}

Among perverse extensions of \(j^*P\), the intermediate extension
\(I\) is characterized by the strict bounds
\begin{equation}\label{eq:strict}
 i_*i^*I\in{}^pD^{\le-1}(M),\qquad
 i_*i^*\mathbb D_MI\in{}^pD^{\le-1}(M).
\end{equation}
Indeed, the first condition excludes nonzero quotients supported
on \(Z\), while the second, by the self-duality of the perverse
\(t\)-structure, excludes nonzero subobjects supported on \(Z\).
Both conditions use only restriction, perverse truncation, and
Verdier duality on the ambient space.

\begin{corollary}\label{cor:support-decomposition}
Let \(M\) be a qcqs rigid space and let \(P\in\Perv(M,\Q_\ell)\).
Suppose that \(c_Z(P)\) is an isomorphism for every Zariski-closed
subspace \(Z\subset M\). Then there are finitely many smooth
Zariski-locally closed subspaces \(U_\alpha\subset M\) and finite-rank
\(\Q_\ell\)-local systems \(L_\alpha\) on \(U_\alpha\) such that
\[
 P\simeq\bigoplus_\alpha\IC_{\overline U_\alpha}(L_\alpha).
\]
\end{corollary}
\begin{proof}
We argue by induction on the dimension of the Zariski-closed support
of \(P\). The assertion is immediate when \(P=0\).
Let \(S=\overline{\supp(P)}\), with its reduced structure. Since
closed pushforward identifies perverse sheaves on \(S\) with
perverse sheaves on \(M\) supported on \(S\), we may work on \(S\).
By constructibility, we can choose pairwise disjoint dense smooth
open subsets of the irreducible components of \(S\) such that
their union \(W\) is open in \(S\) and the restriction of \(P\)
to each of them is a local system shifted by the dimension of
the corresponding component. Let \(T\) be the closed complement
of \(W\) in \(S\).

By our assumption and the compatibility of the boundary
construction with closed pushforward, the boundary morphism
associated with \(T\) is an isomorphism. Writing
\(j_W:W\hookrightarrow S\), Lemma~\ref{lem:boundary} gives
\(P\simeq j_{W!*}j_W^*P\oplus P'\),
where \(P'\) is a perverse sheaf supported on \(T\).
Since \(W\) is a disjoint union of smooth open subsets,
the first summand is a direct sum of intermediate extensions of
shifted local systems.

It remains to treat \(P'\). For every closed subset \(Z\subset T\),
the boundary morphism \(c_Z(P')\) is a direct summand of \(c_Z(P)\),
since the boundary construction commutes with finite direct sums.
Thus \(c_Z(P')\) is an isomorphism. Moreover, \(W\) is dense in every
irreducible component of \(S\), so
\(\dim\supp(P')\le\dim T<\dim S=\dim\supp(P)\).
The induction hypothesis therefore applies to \(P'\).
Pushing the resulting decomposition forward from \(S\) to \(M\)
gives the desired finite direct sum.
\end{proof}

\subsection{Simultaneous algebraic approximation}\label{sec:approx}

We write \(\Spa(B)=\Spa(B,B^\circ)\) for a \(K\)-affinoid algebra \(B\), and \(S^{\an}\) for
analytification, taken relative to the corresponding affinoid space when \(S\) is projective over an
affinoid algebra. For rigid spaces, \(\Omega^1_{S/K}\) denotes continuous differentials.

\begin{proposition}\label{lem:approx}
Let \(K\) be a finite extension of \(\mathbf Q_p\), and put
\[
 R=K[T_1,\ldots,T_m],\qquad A=K\langle T_1,\ldots,T_m\rangle,\qquad D=\operatorname{Spa}(A).
\]
 Consider a diagram consisting of finitely many projective \(A\)-schemes \(V_v\) and finitely
 many \(A\)-morphisms between them. Suppose that some of the \(V_v\) have smooth rigid
 analytifications over \(K\) of pure dimensions \(d_v\). Then there exist a smooth finite-type
 \(R\)-algebra \(B\), a homomorphism \(B\to A\), and an open subset
 \(S\subseteq\operatorname{Spec}B\) containing the image of \(\operatorname{Spec}A\), such that
 \(S\to\operatorname{Spec}R\) is smooth of constant relative dimension \(r\), and the following
 properties hold.

\textup{(i)} The whole diagram descends to \(S\). Every vertex \(V_v\) has a projective model
\(V_{v,S}\) over \(S\), and every morphism in the original diagram is induced by a morphism between
the corresponding models. Moreover, \(V_{v,S}\times_S\operatorname{Spec}A\simeq V_v\). Thus the
original diagram is recovered by base change from \(S\) to \(\operatorname{Spec}A\).

\textup{(ii)} The models of the smooth vertices can be chosen smooth. If the rigid analytification
of \(V_v\) is smooth over \(K\) of pure dimension \(d_v\), then \(V_{v,S}\) is smooth over \(K\) of
pure dimension \(d_v+r\).

\textup{(iii)} The original spaces can be recovered by successive smooth hypersurface cuts. Consider
the analytic base change \(\mathcal S = S^{\mathrm{an}} \times_{(\mathbb A_K^m)^{\mathrm{an}}}D\).
The homomorphism \(B\to A\) gives an analytic section \(\gamma:D\longrightarrow\mathcal S\) of the
natural projection \(\mathcal S\to D\). There is a finite collection of affinoid neighborhoods
\(N_\alpha\) of points of \(\gamma(D)\) such that the open subsets \(\gamma^{-1}(N_\alpha)\) cover
\(D\). On each \(N=N_\alpha\), there are analytic functions \(h_1,\ldots,h_r\in\mathcal O(N)\) whose
common zero locus is exactly \(N\cap\gamma(D)\), as a closed analytic subspace.

For \(0\le k\le r\), put \(M_k=V_N(h_1,\ldots,h_k),M_0=N\). Then \(M_k\) is smooth over
\(K\) of pure dimension \(\dim M_k=m+r-k\). For every \(k<r\), the function
\(h_{k+1}:M_k\longrightarrow\mathbb A_K^{1,\mathrm{an}}\) is smooth. In particular,
\(M_{k+1}\hookrightarrow M_k\) is an effective Cartier divisor.

Now let \(\mathcal V_v = V_{v,S}^{\mathrm{an}}\times_{S^{\mathrm{an}}}\mathcal S\) be the analytic
family associated with a vertex \(V_v\), and put \(V_{v,k}=\mathcal V_v\times_{\mathcal S}M_k\). At
the final stage, \(V_{v,r} \simeq V_v^{\mathrm{an}}\big|_{\gamma^{-1}(N)}\). If \(V_v\) is one of
the vertices specified in (ii), then every \(V_{v,k}\) is smooth over \(K\) of pure dimension \(\dim
V_{v,k}=d_v+r-k\), and, for every \(k<r\), the function \(h_{k+1}:V_{v,k}\longrightarrow\mathbb
A_K^{1,\mathrm{an}}\) is smooth. All vertices and morphisms of the finite diagram are restricted
along the same sequence of hypersurface cuts.

Furthermore, the construction can preserve any prescribed finite collection of projective
embeddings, line bundles, closed immersions, Cartesian squares, nilpotent immersions, and
isomorphisms over open complements. The neighborhoods \(N_\alpha\) may all be chosen quasi-compact
and quasi-separated.

\end{proposition}
\begin{proof}
We divide the proof into four steps.

\smallskip
\noindent\emph{Step 1. Descending the entire diagram}

We first construct an algebraic model of the whole diagram over a smooth finite-type \(R\)-algebra.
The homomorphism \(R=K[T_1,\ldots,T_m] \longrightarrow A=K\langle
T_1,\ldots,T_m\rangle\) is regular, by \cite[Lemma 7.1.4]{LRZ}. Popescu's theorem therefore
expresses \(A\) as a filtered direct limit \(A\simeq\varinjlim_\lambda B_\lambda\), where every
\(B_\lambda\) is a smooth finite-type \(R\)-algebra. Every
projective \(A\)-scheme in our diagram is of finite presentation. After choosing projective
embeddings, its defining equations involve only finitely many coefficients in \(A\). Similarly, the
morphisms between the schemes are described by finitely many algebraic data. Since \(A\) is the
filtered colimit of the \(B_\lambda\), these equations and morphisms are already defined over some
\(B_\lambda\), when \(\lambda\) is sufficiently large. Because the diagram is finite, we can
choose one \(B_\lambda\) that works for every vertex and every arrow.

We must also ensure that all the relations in the diagram continue to hold after descending. For
instance, suppose that the original diagram contains a Cartesian square. We descend the isomorphism
identifying one vertex with the appropriate fiber product, together with its inverse and the
identities showing that the two maps are inverse to one another. Likewise, for a closed immersion,
we retain its finitely generated
defining ideal. For a line bundle, we retain a finite collection of local trivializations and their
transition functions. For a nilpotent immersion, we retain an identity \(J^n=0\) for some fixed
integer \(n\). If a morphism is an isomorphism outside a prescribed closed subset, we descend the
defining ideal of that subset, the isomorphism on its open complement, and the inverse isomorphism.
These open complements are quasi-compact because the schemes involved are Noetherian. All these
conditions involve only finitely many algebraic data and identities. By passing to a sufficiently
late stage in the filtered system, we can make them hold simultaneously. We therefore obtain a
smooth finite-type \(R\)-algebra \(B\), together with a homomorphism \(B\to A\), over which the
entire diagram and all its prescribed relations are defined. If the original data came from
projective rigid spaces over \(D\), relative projective GAGA \cite[Example 3.2.6]{Conrad} allows us
to pass to projective \(A\)-schemes before carrying out this construction. This proves the descent
assertion, apart from the possible shrinking of the base in the next step.

\smallskip
\noindent\emph{Step 2. Making the models of the smooth vertices smooth}

We now improve the algebraic model so that every vertex whose original rigid analytification is
smooth also has a smooth algebraic model over \(K\). Fix a vertex \(V_v\)
whose rigid analytification is smooth over \(K\) of pure dimension \(d_v\). Choose a projective
embedding \(V_v\hookrightarrow\mathbf P_A^N\). We first show that \(V_v\), regarded as a scheme, is
regular. Let \(x\) be any closed point of \(V_v\). Relative analytification identifies the completed
local ring of \(V_v\) at \(x\) with the corresponding completed local ring of its rigid
analytification \cite[Example 2.2.11]{Conrad}. The latter is regular because the rigid
analytification is smooth. Hence the completed local ring of \(V_v\) is regular, and so is the local
ring \(\mathcal O_{V_v,x}\). Since \(V_v\) is an excellent Jacobson scheme, regularity at all closed
points implies regularity everywhere. Thus \(V_v\) is a regular scheme.

Next, we choose local equations that will allow us to recognize smoothness after descending. At a
closed point \(x\in V_v\), write \(C_x=\mathcal O_{\mathbf P_A^N,x}\) and let \(J_x\subset C_x\) be
the ideal defining \(V_v\). Both \(C_x\) and \(C_x/J_x\) are regular local rings. Their dimensions
are \(\dim C_x=m+N, \dim(C_x/J_x)=d_v\). Put \(c_v=m+N-d_v\). A standard result about regular
local rings \cite[Tag 00NR]{Stacks} says that \(J_x\) can be generated by \(c_v\) elements
\(f_1,\ldots,f_{c_v}\) that form part of a regular system of parameters of \(C_x\). In particular,
their differentials are linearly independent at \(x\). After shrinking around \(x\), we may regard
these \(f_i\) as regular functions defining \(V_v\) on an affine neighborhood. We choose finitely
many such neighborhoods covering \(V_v\).

On one of the affine neighborhoods, let \(q\) be a column vector of the original
defining equations, and let \(f=(f_1,\ldots,f_{c_v})^t\). Because \(q\) and \(f\) generate the same
ideal, there are matrices \(P\) and \(Q\) of regular functions such that \(f=Pq,q=Qf.\)
 We include these matrices and both identities in the finite data to be descended. After passing to
 a sufficiently late Popescu stage, the corresponding identities still hold over \(B\).
 Consequently, the descended \(f_i\) generate the ideal of the descended model on the chosen
 neighborhoods. We perform this construction simultaneously for all vertices whose rigid
 analytifications are smooth over \(K\).
 Because \(\operatorname{Spec}A\) is connected, its image lies in one connected component of
 \(\operatorname{Spec}B\). Replacing \(B\) by the corresponding open and closed factor, we may
 assume that \(B\) is smooth over \(R\) of constant relative dimension \(r\).

Now write \(V_{v,B}\) for the descended projective scheme over \(\operatorname{Spec}B\). Let \(x\in
V_v\) be a closed point, and let \(x_B\in V_{v,B}\) be its image. Since the residue field
\(\kappa(x)\) is finite over \(K\), and \(V_{v,B}\) is of finite type over \(K\), the point \(x_B\)
is also closed. The algebraic differentials of the descended equations at \(x_B\), after extension
to \(\kappa(x)\), map to the continuous differentials of the original equations at \(x\). The latter
are linearly independent, so the former must also be linearly independent. The ambient space
\(\mathbf P_B^N\) is smooth over \(K\) of dimension \(m+r+N\). Since \(V_{v,B}\) is locally defined
by \(c_v\) equations with independent differentials, the Jacobian criterion shows that \(V_{v,B}\)
is smooth at \(x_B\), of dimension \(m+r+N-c_v=d_v+r\). We have therefore proved that the descended
model is smooth at the images of all closed points of the original \(V_v\).

We now turn this local information into a statement about the entire model. Let \(G_v\subseteq
V_{v,B}\) be the union of the descended affine neighborhoods on which our chosen equations generate
the defining ideal. Inside \(G_v\), let \(U_v\) be the locus where \(V_{v,B}\) is smooth over \(K\)
of dimension \(d_v+r\). This is an open subset. Its complement \(B_v=V_{v,B}\setminus U_v\) is
closed. By the preceding argument, the inverse image of \(B_v\) in \(V_v\) contains no closed
points. Since \(V_v\) is Jacobson, this inverse image must be empty. The morphism
\(V_{v,B}\longrightarrow\operatorname{Spec}B\) is projective, hence proper. Therefore the image of
\(B_v\) in \(\operatorname{Spec}B\) is closed. Moreover, this image is disjoint from the image of
\(\operatorname{Spec}A\). We may therefore remove it from \(\operatorname{Spec}B\) without changing
the original \(A\)-scheme. There are only finitely many vertices to which this construction applies.
Removing the
corresponding finitely many closed images gives an open subset \(S\subseteq\operatorname{Spec}B\)
containing the entire image of \(\operatorname{Spec}A\). For each such vertex \(V_v\), the model
\(V_{v,S}\) is now smooth over \(K\) of pure dimension \(d_v+r\). Since the original image of
\(\operatorname{Spec}A\) is still contained in \(S\), base change to \(A\) recovers the original
diagram. This proves (i) and (ii).

\smallskip
\noindent\emph{Step 3. Recovering the original spaces by smooth hypersurface cuts}

We now prove the main geometric assertion. The idea is to describe the original analytic base \(D\)
as the common zero locus of \(r\) analytic functions in the larger smooth space \(\mathcal S\). We
will then arrange that these functions cut the analytic families of the vertices in (ii)
transversely. Recall
that \(\mathcal S = S^{\mathrm{an}} \times_{(\mathbb A_K^m)^{\mathrm{an}}}D\). Because
\(S\to\operatorname{Spec}R\) is smooth of relative dimension \(r\), the projection \(p:\mathcal
S\longrightarrow D\) is smooth of relative dimension \(r\). The homomorphism \(B\to A\) gives a
section \(\gamma:D\longrightarrow\mathcal S, p\circ\gamma=\mathrm{id}_D\). We work locally
near this section.

\smallskip
\noindent\emph{Constructing the graph equations}

A smooth morphism admits relative étale coordinates locally. Thus, near any chosen point of the
section, we can find functions \(s_1,\ldots,s_r\) giving a relative étale chart. After shrinking the
neighborhood, choose an open subset \(D_0\subseteq D\) such that \(\gamma(D_0)\) lies in this chart.
Define analytic functions on \(D_0\) by \(a_i=\gamma^*s_i\). On the corresponding neighborhood in
\(\mathcal S\), put \(h_i=s_i-p^*a_i\) for \(1\le i\le r\). By construction, \(\gamma^*h_i=0\) for
every \(i\). Hence the image of the section is contained in their common zero locus.

In fact, these equations describe the section exactly, after a further shrinking. To see this,
consider the relative étale map given by the coordinates \(s_i\). The common zero locus of the
\(h_i\) is obtained by pulling back this map along the graph of \((a_1,\ldots,a_r):D_0\to\mathbb
A_K^{r,\mathrm{an}}\). It is therefore separated and étale over \(D_0\). The section
\(\gamma|_{D_0}\) is a section of this separated étale morphism. Consequently, its image is open and closed in
the common zero locus. We may remove the other components of that zero locus by shrinking the
neighborhood around the section. Thus, on a suitable neighborhood \(N\),
\[
 V_N(h_1,\ldots,h_r)=N\cap\gamma(D).
\]
 This equality holds as closed analytic subspaces.

\smallskip
\noindent\emph{Checking transversality on the smooth sources}

Fix a vertex \(V_v\) whose rigid analytification is smooth, and consider its analytic family
\(\mathcal V_v\). Over \(N\), write \(V_{v,0}=\mathcal V_v\times_{\mathcal S}N\). By Step 2,
\(V_{v,0}\) is
smooth over \(K\) of pure dimension \(d_v+r\). Its intersection with the common zero locus of the
\(h_i\) is \(V_{v,r}=V_N(h_1,\ldots,h_r)\times_{\mathcal S}\mathcal V_v\). Because the zero locus is
the section, \(V_{v,r}\) is precisely the restriction of the original rigid analytification
\(V_v^{\mathrm{an}}\) to \(\gamma^{-1}(N)\). In particular, it is smooth over \(K\) of pure
dimension \(d_v\).

We claim that the differentials \(dh_1,\ldots,dh_r\) are linearly independent at every point of
\(V_{v,r}\). It is enough to check this at classical points. Let \(x\) be such a point, and put
\(C=\mathcal O_{V_{v,0},x}, I=(h_1,\ldots,h_r)C\). The local ring \(C\) is regular of
dimension \(d_v+r\), while \(C/I=\mathcal O_{V_{v,r},x}\) is regular of dimension \(d_v\). As in
Step 2, the criterion shows that \(I\) can be generated by \(r\) members of a
regular system of parameters of \(C\). But \(I\) is already generated by \(h_1,\ldots,h_r\). Since
the conormal space \(I/\mathfrak m_xI\) has dimension \(r\), their classes form a basis of this
space. The regularity of \(C/I\) ensures that these classes have linearly independent images in
\(\mathfrak m_x/\mathfrak m_x^2\). Equivalently, the continuous differentials \(dh_1,\ldots,dh_r\)
are linearly independent at \(x\). Since every nonempty closed analytic subspace has a classical
point locally, the locus where these differentials become dependent cannot meet \(V_{v,r}\). This
proves transversality along the final graph.

We next shrink the neighborhood so that the map \(H_v=(h_1,\ldots,h_r):
V_{v,0}\longrightarrow\mathbb A_K^{r,\mathrm{an}}\) is smooth throughout it.
Let \(C_v\subseteq V_{v,0}\) be the closed analytic locus where
the differentials \(dh_1,\ldots,dh_r\) are linearly dependent. We have just proved that \(C_v\cap
V_{v,r}=\varnothing\). The projection \(g_v:V_{v,0}\longrightarrow N\) is projective, hence proper.
Therefore \(g_v(C_v)\) is a closed analytic subset of \(N\). Furthermore, \(g_v(C_v)\) is disjoint
from the image of \(\gamma\). Indeed, a point of \(C_v\) lying above the section would satisfy all
equations \(h_i=0\), and would therefore belong to \(V_{v,r}\), contradicting the preceding
equality. We may consequently remove \(g_v(C_v)\) from \(N\) without removing any point of the
section. We do this for each of the finitely many vertices in (ii), shrinking \(N\) accordingly.
On the resulting neighborhood, all the maps \(H_v=(h_1,\ldots,h_r):
V_{v,0}\longrightarrow\mathbb A_K^{r,\mathrm{an}}\) have differentials of full rank. Since the
sources are smooth over \(K\), the analytic Jacobian criterion implies that every \(H_v\) is smooth.

\smallskip
\noindent\emph{The successive cuts}

For \(0\le k\le r\), set
\[
 V_{v,k} = V_{v,0}\times_{\mathbb A_K^{r,\mathrm{an}}} \bigl(\{0\}^k\times\mathbb
 A_K^{r-k,\mathrm{an}}\bigr).
\]
 Equivalently, \(V_{v,k}\) is defined by \(h_1=\cdots=h_k=0\). Because \(H_v\) is smooth, its base
 change defining \(V_{v,k}\) is also smooth. Hence \(V_{v,k}\) is smooth over \(K\) of pure
 dimension \(d_v+r-k\). Moreover, the remaining functions \((h_{k+1},\ldots,h_r):
 V_{v,k}\longrightarrow\mathbb A_K^{r-k,\mathrm{an}}\) define a smooth morphism. For \(k<r\),
 composing it with the projection to the first coordinate shows that \(h_{k+1}:V_{v,k}
 \longrightarrow\mathbb A_K^{1,\mathrm{an}}\) is smooth. Thus the hypersurface \(h_{k+1}=0\)
 cuts each \(V_{v,k}\) transversely.

We must also check the same statement on the base. The translated étale chart \((p,h_1,\ldots,h_r):
N\longrightarrow D_0\times\mathbb A_K^{r,\mathrm{an}}\) is étale. Therefore, after setting the first
\(k\) coordinates equal to zero, \(M_k=V_N(h_1,\ldots,h_k)\) is étale over \(D_0\times\mathbb
A_K^{r-k,\mathrm{an}}\). It follows that \(M_k\) is smooth over \(K\) of pure dimension \(m+r-k\).
For \(k<r\), the function \(h_{k+1}\) is locally a coordinate projection from this étale chart.
Hence \(h_{k+1}:M_k\longrightarrow\mathbb A_K^{1,\mathrm{an}}\) is smooth, and its zero locus is an
effective Cartier divisor. Finally, every
vertex and every morphism in the original diagram is restricted by base change along the same
sequence \(M_0\supset M_1\supset\cdots\supset M_r\). Thus all the Cartesian squares and other
prescribed base-change relations remain valid throughout the construction. This establishes (iii)
locally near the section.

\smallskip
\noindent\emph{Step 4. Choosing finitely many quasi-compact neighborhoods}

It remains to choose finitely many neighborhoods on which the preceding construction works. For each
point of \(\gamma(D)\), Step 3 gives an analytic neighborhood on which the graph equations exist and
all the required transversality conditions hold. We may choose a smaller affinoid neighborhood of
that point inside the given neighborhood. The inverse images of these affinoid neighborhoods under
\(\gamma:D\longrightarrow\mathcal S\) form an open cover of \(D\). Since \(D\) is quasi-compact, a
finite subcover exists. Denote the corresponding affinoid neighborhoods by \(N_1,\ldots,N_t\). On
each \(N_\alpha\), the successive cuts \(V_{N_\alpha}(h_1,\ldots,h_k)\) are closed analytic
subspaces of an affinoid space, so they are themselves affinoid. In particular, they are
qcqs. Restricting the finite projective diagram to these bases
preserves all its arrows, Cartesian squares, and other relations, because the restriction is
performed by base change. The resulting projective sources are also quasi-compact and
quasi-separated. This completes the proof.
\end{proof}

For a general affinoid rigid base \(Y\), one first embeds \(Y\) as a closed rigid subspace of a
polydisc and applies the proposition to the composed projective morphism. A chosen relatively ample
line bundle can be retained by descending a projective embedding defined by a sufficiently large
tensor power. Replacing \(H\) by \(H^{\otimes N}\) multiplies its first Chern class by \(N\), so the
rational Lefschetz isomorphism condition is unchanged.

\subsection{Resolution of the boundary}\label{sec:tree}

\begin{lemma}\label{lem:tree}
Let \(V\) be a projective \(A\)-scheme. For any finite coefficient
ring \(\Lambda\), the constant sheaf \(\Lambda_V\) can be
constructed from constant sheaves on finitely many regular projective
\(A\)-schemes by proper pushforwards, shifts, and distinguished
triangles. This construction is encoded by a finite diagram of
projective \(A\)-schemes. The diagram and its defining identities
descend to a common algebraic approximation and remain valid after
every analytic graph cut.
\end{lemma}

\begin{proof}
\smallskip
\noindent\emph{Step 1. Reduce to the case where \(V\) is reduced.}

A nilpotent closed immersion induces an equivalence of \'etale
topoi, identifying \(\Lambda_V\) with the pushforward of
\(\Lambda_{V_{\mathrm{red}}}\). Replace \(V\) by
\(V_{\mathrm{red}}\), retaining the reduction immersion and a
fixed nilpotence relation in the diagram so that this identification
persists under base change. The resulting scheme is reduced,
excellent, Noetherian, and finite-dimensional, of characteristic zero.

\smallskip
\noindent\emph{Step 2. Use resolution of singularities.}

If \(V\) is regular, the assertion is immediate. Otherwise,
Temkin's resolution theorem \cite[Theorem 1.2.1]{Temkin} gives a
projective resolution \(\pi:\widetilde V\to V\), obtained
by a finite sequence of blow-ups whose centers are disjoint
from the inverse image of the regular locus \(V_{\mathrm{reg}}\).
Take \(Z\) to be the reduced closed complement of
\(V_{\mathrm{reg}}\). Since \(V\) is reduced, this open locus
contains the generic point of every irreducible component.
Thus \(Z\) is nowhere dense and \(\pi\) is an isomorphism over
\(V\setminus Z\). The exceptional inverse image
\(E=\widetilde V\times_V Z\) gives a Cartesian square
\[
\begin{matrix}
E&\longrightarrow&\widetilde V\\
\downarrow&&\downarrow\pi\\
Z&\xrightarrow{i}&V.
\end{matrix}
\]
Both \(Z\) and \(E\) have smaller dimension than \(V\).

\smallskip
\noindent\emph{Step 3. Construct the abstract blow-up triangle.}

For every finite coefficient ring \(\Lambda\), this square gives
a distinguished triangle
\begin{equation}\label{eq:blowup}
\Lambda_V
\longrightarrow
R\pi_*\Lambda_{\widetilde V}\oplus i_*\Lambda_Z
\longrightarrow
i_*R(\pi|_E)_*\Lambda_E
\longrightarrow
\Lambda_V[1].
\end{equation}

The first map consists of pullback to \(\widetilde V\) and
restriction to \(Z\); the second is the difference of the two
restrictions to \(E\). Choose a bounded-below injective resolution
\(\Lambda_{\widetilde V}\to\mathcal I^\bullet\), put
\(\mathcal B^\bullet=\pi_*\mathcal I^\bullet\), and let
\(u:\Lambda_V\to\mathcal B^\bullet\) be the adjunction map.
Proper base change identifies \(i_*i^*\mathcal B^\bullet\)
with the third term of \eqref{eq:blowup}. Write \(\mathrm{res}:\Lambda_V\to i_*\Lambda_Z\) and
\(\rho:\mathcal B^\bullet\to i_*i^*\mathcal B^\bullet\) for the restriction maps. In these
representatives,
the first two arrows are the maps of complexes
\[
\begin{aligned}
\alpha=(u,\mathrm{res})&:\Lambda_V\longrightarrow
\mathcal B^\bullet\oplus i_*\Lambda_Z,\\
b=(\rho,-i_*i^*u)&:\mathcal B^\bullet\oplus i_*\Lambda_Z
\longrightarrow i_*i^*\mathcal B^\bullet.
\end{aligned}
\]
Then \(b\alpha=0\). Hence \(b\) induces a global map
\(\beta:\operatorname{Cone}(\alpha)\to i_*i^*\mathcal B^\bullet\).
Outside \(Z\), the morphism \(\pi\) is an isomorphism, so
\(\alpha\) is a quasi-isomorphism and both sides of \(\beta\)
are acyclic. At a geometric point \(\bar z\) of \(Z\), put
\(C_{\bar z}=R\Gamma(\widetilde V_{\bar z},\Lambda)\).
Proper base change identifies the stalk sequence with
\[
\Lambda
\longrightarrow
C_{\bar z}\oplus\Lambda
\longrightarrow
C_{\bar z}
\longrightarrow\Lambda[1],
\]
where the arrows are \((u,\mathrm{id})\) and
\((a,b)\mapsto a-u(b)\), with \(u:\Lambda\to C_{\bar z}\)
the natural map. This is a split distinguished triangle, so \(\beta\) is a
quasi-isomorphism at \(\bar z\). Since geometric stalks detect
quasi-isomorphisms, \(\beta\) is a global quasi-isomorphism.
Identifying its target with \(\operatorname{Cone}(\alpha)\)
gives the distinguished triangle \eqref{eq:blowup}.

\smallskip
\noindent\emph{Step 4. Induction on dimension.}

Apply the same procedure to \(Z_{\mathrm{red}}\) and
\(E_{\mathrm{red}}\). At each blow-up, the generic points of the irreducible components
correspond bijectively to those of the preceding scheme
\cite[Tag 0BFM]{Stacks}, since the center misses all of them.
The components of \(\widetilde V\) are therefore the strict
transforms of the components of \(V\), with the same dimensions.
The closed subsets \(Z\subset V\) and \(E\subset\widetilde V\)
contain none of these generic points. Hence \(\dim Z<\dim V\)
and \(\dim E<\dim\widetilde V=\dim V\).
Induction on \(\dim V\), the maximum dimension of its
irreducible components, therefore terminates after finitely many
abstract blow-up triangles. The terminal schemes are regular and
projective over \(A\); separate their finitely many open and closed
pure-dimensional components. Relative analytification identifies
completed local rings at classical points
\cite[Example 2.2.11]{Conrad}, so their rigid analytifications
are regular, hence smooth over the characteristic-zero field \(K\).
The triangles express \(\Lambda_V\) in terms of their constant
sheaves by proper pushforwards and shifts.

\smallskip
\noindent\emph{Step 5. Preserve the construction after restriction to the graph.}

The construction involves finitely many projective schemes,
Cartesian squares, closed immersions, nilpotent immersions, and
isomorphisms outside prescribed closed subsets.
Proposition~\ref{lem:approx} descends this diagram and its
relations to one smooth algebraic approximation. After each
analytic graph cut, the squares remain Cartesian, the proper maps
remain proper, and the maps remain isomorphisms outside the inverse images of the centers. The
abstract blow-up triangles therefore
remain valid.
\end{proof}

For a function \(h:M\to\mathbb A^{1,\an}\) over a complete algebraically closed extension of \(K\),
put \(M_0=h^{-1}(0)\) and \(M^\times=M\setminus M_0\), with inclusions \(k:M_0\hookrightarrow M\)
and \(j:M^\times\hookrightarrow M\). Let \(q_e:M_e^\times\to M^\times\) be the pullback of
\(z\mapsto z^e\) on the punctured affine line, for \(e\ge1\). For a complex \(F\) with finite
coefficients on \(M^\times\), its nearby cycles are
\[
 \psi_h(F)=\varinjlim_e k^*Rj_*Rq_{e*}q_e^*F,
\]
where the limit is indexed by divisibility. For a complex \(E\) on \(M\) with the same finite
coefficients, its vanishing cycles \(\phi_h(E)\) are defined by the distinguished triangle
\[
 k^*E\xrightarrow{\mathrm{sp}}\psi_h(j^*E)
 \longrightarrow\phi_h(E)\longrightarrow k^*E[1],
\]
where \(\mathrm{sp}\) is the specialization map induced by adjunction. Both functors carry the
natural monodromy action of \(\widehat{\Z}(1)=\varprojlim_e\mu_e\), where \(\mu_e\) is the group of
\(e\)-th roots of unity \cite[Definition 2.1 and Remark 2.3]{ZhouVC}.

\begin{proposition}\label{prop:boundaryvanish}
Let \(f:X\to D\) be projective, with \(X\) smooth over
\(K\) and pure-dimensional, and let \(i:Z\hookrightarrow D\)
be a closed immersion. There is a simultaneous algebraic approximation
as in Proposition~\ref{lem:approx} for which the following holds
after extension to a completed algebraic closure of \(K\), at
every successive graph cut. Write \(g:W\to M\) for the
projective morphism at one stage, \(i_M:Z_M\hookrightarrow M\)
for its boundary, and \(h\) for the next cut function.
For every prime \(\ell\neq p\) and \(n\geq1\), put
\(\Lambda_n=\mathbb Z/\ell^n\), \(s=\dim W\), and
\(K_n=Rg_*\Lambda_{n,W}[s]\). Then
\begin{equation}\label{eq:twovanishing}
\phi_h(K_n)=0,
\qquad
\phi_h(i_{M*}i_M^*K_n)=0.
\end{equation}

The same approximation works for all primes \(\ell\neq p\) and all \(n\).
\end{proposition}

\begin{proof}
\smallskip
\noindent\emph{Step 1. Construct a single approximation containing the boundary.}

Let \(I\subset A\) be the finitely generated ideal defining \(Z\),
and put \(Z_A=\Spec(A/I)\). By relative projective GAGA
\cite[Example 3.2.6]{Conrad}, the morphism \(f\) is the rigid
analytification of a projective morphism
\(f_A:X_A\to\Spec A\). Set
\((X_Z)_A=X_A\times_{\Spec A}Z_A\).
Relative analytification commutes with fibre products
\cite[Example 2.2.11]{Conrad}, so the analytification of
\((X_Z)_A\) is \(X_Z=X\times_D Z\).

Apply Lemma~\ref{lem:tree} to \((X_Z)_A\), and include its
finite resolution diagram, \(f_A\), the immersion
\(Z_A\hookrightarrow\Spec A\), and the Cartesian square
defining \((X_Z)_A\) in one finite diagram.
The vertex \(X_A\) and all regular terminal schemes have smooth
rigid analytifications over \(K\). Proposition~\ref{lem:approx}
gives algebraic models on which each successive cut function is
smooth on the model of \(X\) and every regular terminal source,
after the preceding cuts. It also preserves the Cartesian squares and the relations defining the
abstract blow-up squares.

\smallskip
\noindent\emph{Step 2. Vanishing cycles of the total direct image.}

At a stage \(g:W\to M\) of the successive cuts defining the graph, \(W\)
is smooth over \(K\), and the composite of \(g\) with the
next cut function \(h:M\to\mathbb A_K^{1,\mathrm{an}}\)
is smooth on the chosen neighborhood. After passage to a completed
algebraic closure, smooth local acyclicity
\cite[Example 2.14(1)]{ZhouVC} gives
\(\phi_{h\circ g}(\Lambda_{n,W})=0\) for \(\ell\neq p\).
Compatibility with proper pushforward \cite[Lemma 2.12(2)]{ZhouVC}
then gives
\[
\phi_h(Rg_*\Lambda_{n,W})
\simeq
Rg_{0*}\phi_{h\circ g}(\Lambda_{n,W})
=0,
\]
where \(g_0\) is the induced proper map between the zero fibers.
Shifting by \([s]\) gives \(\phi_h(K_n)=0\).

\smallskip
\noindent\emph{Step 3. Vanishing cycles of the boundary complex.}

Put \(W_Z=W\times_M Z_M\). Proper base change identifies
\[
i_{M*}i_M^*Rg_*\Lambda_{n,W}
\simeq
R(g|_{W_Z})_*\Lambda_{n,W_Z},
\]
where \(W_Z\) maps to \(M\) by its natural proper map.
Use the finite resolution diagram of Lemma~\ref{lem:tree}, whose
abstract blow-up triangles remain valid after the preceding cuts. At a
terminal vertex, the source and the next cut function on it are
smooth by Proposition~\ref{lem:approx}, so its constant sheaf
has zero vanishing cycles. At an internal vertex, the abstract
blow-up triangle expresses the constant sheaf in terms of those
on a resolution, a center, and an exceptional inverse image,
followed by proper pushforwards. Induction, compatibility with proper
pushforward, and exactness of vanishing cycles give vanishing at
that vertex. Proceeding back to the boundary gives
\(\phi_{(h\circ g)|_{W_Z}}(\Lambda_{n,W_Z})=0\).
Compatibility with proper pushforward, followed by the shift \([s]\), gives
\(\phi_h(i_{M*}i_M^*K_n)=0\), proving the second equality
in \eqref{eq:twovanishing}.

\smallskip
\noindent\emph{Step 4. Repeat the argument after each graph cut.}

The geometric relations included in the finite diagram remain valid
after every successive graph cut, and the next cut function is smooth
on every regular terminal source in the resolution diagram. Thus
Steps 2 and 3 apply at every stage, for every \(\ell\neq p\)
and \(n\geq1\), with the same approximation.
\end{proof}

In the proof of Theorem~\ref{thm:support}, the two equalities in
\eqref{eq:twovanishing} preserve, respectively, perverse bounds
after restriction to the graph and the strict boundary conditions
characterizing intermediate extension.

\subsection{Perverse restriction and uniform torsion}\label{sec:bounds}

We work over a complete algebraically closed nonarchimedean field of characteristic zero and residue
characteristic \(p\), and fix a prime \(\ell\neq p\). Let \(h:M\to\mathbb A^{1,\mathrm{an}}\) be a
morphism with \(M\) qcqs, let \(k:M_0\hookrightarrow M\) be its zero fiber, and put \(T=k^*[-1]\).
For finite coefficients prime to \(p\), the shifted functors of nearby and vanishing cycles
\(\psi_h[-1]\) and \(\phi_h[-1]\) are perverse \(t\)-exact on Zariski-constructible complexes of
finite Tor-amplitude \cite[Corollary 5.7 and Theorem 5.8(1)]{ZhouVC}. For nearby cycles, the complex
on \(M\setminus M_0\) must admit a constructible extension to \(M\).

A complex \(F\in D^b_{\mathrm{zc}}(M,\mathbb Q_\ell)\) has uniformly annihilated vanishing cycles
along \(h\) if there exist an integral model \(E\in D^b_{\mathrm{zc}}(M,\mathbb Z_\ell)\) of \(F\)
and an integer \(N\ge0\) such that
\begin{equation}\label{eq:uniform}
\ell^N\operatorname{id}_{\phi_h(E\otimes_{\mathbb Z_\ell}^{\mathbf L}\mathbb Z/\ell^n)}
=0
\qquad\text{for every }n\ge1.
\end{equation}
We retain the monodromy action on nearby and vanishing cycles, and impose \eqref{eq:uniform} in the
corresponding derived category with finite coefficients \cite[Remark 2.3(2)]{ZhouVC}. Perverse
bounds refer to the underlying complexes on \(M_0\) after forgetting monodromy.

\begin{lemma}\label{lem:thick}
Condition \eqref{eq:uniform} is independent of the integral model of \(F\). The class of rational
complexes satisfying this condition is stable under finite direct sums, shifts, Tate twists, cones,
and derived direct summands.
\end{lemma}
\begin{proof}
Suppose that \(F'\) is a rational direct summand of \(F\), and choose integral models \(E'\) and
\(E\), respectively. The inclusion and retraction have composite \(\operatorname{id}_{F'}\). Since
rational Hom groups are obtained by inverting \(\ell\) in integral Hom groups, clearing denominators
and multiplying by a further power of \(\ell\) gives integral maps \(u:E'\to E\) and \(v:E\to E'\)
with \(vu=\ell^c\operatorname{id}_{E'}\) for a fixed integer \(c\ge0\). This is an equality in the
integral derived category. If \(E\) satisfies \eqref{eq:uniform} with exponent \(N\), reduction
modulo \(\ell^n\) and application of vanishing cycles give
\(\ell^{N+c}\operatorname{id}_{\phi_h(E'/\ell^n)}=0\) for every \(n\ge1\).
Thus \(F'\) satisfies \eqref{eq:uniform}. Applying this argument to an isomorphism between two
rationalizations proves independence of the integral model. Finite direct sums, shifts, and Tate
twists preserve the condition.

For cones, let \(F_1\xrightarrow{f}F_2\to F_3\to F_1[1]\) be a distinguished triangle of rational
complexes, with \(F_1\) and \(F_2\) satisfying \eqref{eq:uniform}. Choose integral models \(E_1\)
and \(E_2\) with uniform annihilation exponents \(u\) and \(v\), respectively. Since rational Hom
groups are obtained by inverting \(\ell\), there are an integer \(c\ge0\) and an integral morphism
\(\widetilde f:E_1\to E_2\) such that \(\widetilde f[1/\ell]=\ell^cf\). Put
\(E_3=\operatorname{Cone}(\widetilde f)\). The automorphism
\(\ell^c\operatorname{id}_{F_2}\) identifies the arrow \(f\) with \(\ell^cf\), so their cones are
isomorphic. Hence \(E_3[1/\ell]\simeq\operatorname{Cone}(\ell^cf)\simeq F_3\), and \(E_3\) is an
integral model of \(F_3\).

For each \(n\ge1\), derived reduction modulo \(\ell^n\), followed by vanishing cycles, gives a
distinguished triangle
\[
 A_n\longrightarrow B_n\xrightarrow{\beta_n}C_n
 \xrightarrow{\delta_n}A_n[1],
\]
where \(A_n=\phi_h(E_1/\ell^n)\), \(B_n=\phi_h(E_2/\ell^n)\), and
\(C_n=\phi_h(E_3/\ell^n)\). By assumption,
\(\ell^u\operatorname{id}_{A_n}=0\) and \(\ell^v\operatorname{id}_{B_n}=0\). Since \(\ell^u\)
also annihilates \(A_n[1]\), we have
\(\delta_n\circ(\ell^u\operatorname{id}_{C_n})
=(\ell^u\operatorname{id}_{A_n[1]})\circ\delta_n=0\).
Applying \(\operatorname{Hom}(C_n,-)\) to the triangle gives the exact sequence
\[
 \operatorname{Hom}(C_n,B_n)
 \xrightarrow{\beta_n\circ-}\operatorname{Hom}(C_n,C_n)
 \xrightarrow{\delta_n\circ-}\operatorname{Hom}(C_n,A_n[1]).
\]
Thus there is a morphism \(t_n:C_n\to B_n\) such that
\(\ell^u\operatorname{id}_{C_n}=\beta_n\circ t_n\). It follows that
\[
 \ell^{u+v}\operatorname{id}_{C_n}
 =\beta_n\circ(\ell^v\operatorname{id}_{B_n})\circ t_n=0.
\]
The exponent \(u+v\) is independent of \(n\), so \(E_3\) satisfies \eqref{eq:uniform}.
This proves stability under cones.
\end{proof}

\begin{lemma}\label{lem:bounds}
Suppose that \(F\in D^b_{\mathrm{zc}}(M,\mathbb Q_\ell)\) satisfies \eqref{eq:uniform}. Then, for
every integer \(a\),
\[
\begin{aligned}
F\in{}^pD^{\le a}&\quad\Longrightarrow\quad TF\in{}^pD^{\le a},\\
F\in{}^pD^{\ge a}&\quad\Longrightarrow\quad TF\in{}^pD^{\ge a}.
\end{aligned}
\]
In particular, \(k^*F[-1]\) is perverse whenever \(F\) is perverse and satisfies \eqref{eq:uniform}.
This applies to rational perverse direct summands of shifts of integral complexes whose vanishing
cycles are zero after reduction modulo every \(\ell^n\).
\end{lemma}
\begin{proof}
\smallskip
\noindent\emph{Step 1. Choose an integral model with the correct upper bound.}
Assume \(F\in{}^pD^{\le a}(M,\mathbb Q_\ell)\). Choose an integral model \(E_0\) and put
\(E={}^p\tau^{\le a}E_0\). Rationalization is compatible with the perverse t-structure, so \(E\)
still represents \(F\), and Lemma~\ref{lem:thick} shows that it satisfies \eqref{eq:uniform} for
some exponent \(N\). We have \(E\in{}^pD^{\le a}(M,\mathbb Z_\ell)\). By \cite[Proposition
4.6(2)]{BH}, each derived reduction \(E_n=E\otimes_{\mathbb Z_\ell}^{\mathbf L}\mathbb Z/\ell^n\)
belongs to \({}^pD^{\le a}(M,\mathbb Z/\ell^n)\).

The reductions have finite Tor-amplitude over \(\mathbb Z/\ell^n\), by \cite[Remark 3.33]{BH}.
Indeed,
\[
E_n\otimes_{\mathbb Z/\ell^n}^{\mathbf L}\mathbb F_\ell
\simeq E\otimes_{\mathbb Z_\ell}^{\mathbf L}\mathbb F_\ell
\]
is bounded and Zariski-constructible, so \(E_n\) has finite Tor-amplitude, as required in
\cite[\S5]{ZhouVC}. This also applies to integral models obtained by perverse truncation or Verdier
duality, since both operations preserve the integral constructible category \cite[Proposition 4.6(1)
and Theorem 3.36(2)]{BH}.

\smallskip
\noindent\emph{Step 2. Apply nearby and vanishing cycles.}
For a complex \(E\) on \(M\), write \(\Psi E=\psi_h(E|_{M\setminus M_0})[-1]\) and \(\Phi
E=\phi_h(E)[-1]\). By \cite[Corollary 5.7 and Theorem 5.8(1)]{ZhouVC}, both \(\Psi E_n\) and \(\Phi
E_n\) belong to \({}^pD^{\le a}(M_0,\mathbb Z/\ell^n)\); the assertion about nearby cycles applies
because the complex on the open complement extends to \(E_n\) on \(M\). The distinguished triangle
\[
TE_n\longrightarrow\Psi E_n\longrightarrow\Phi E_n\longrightarrow TE_n[1]
\]
gives \(TE_n\in{}^pD^{\le a+1}\), and its perverse long exact sequence identifies
\(\pH^{a+1}(TE_n)\) as a quotient of \(\pH^a(\Phi E_n)\). Condition \eqref{eq:uniform} annihilates
\(\Phi E_n\) by one power \(\ell^N\) for all \(n\), so \(\ell^N\pH^{a+1}(TE_n)=0\) for every
\(n\ge1\).

\smallskip
\noindent\emph{Step 3. Remove the possible extra perverse degree.}
Set \(A=TE\). Ordinary pullback commutes with derived coefficient reduction, giving \(A/\ell^n\simeq
TE_n\). The bound for \(A/\ell\) implies \(A\in{}^pD^{\le a+1}(M_0,\mathbb Z_\ell)\) by
\cite[Construction 4.3]{BH}, and Step 2 gives \(\ell^N\pH^{a+1}(A/\ell^n)=0\) for all \(n\). Put
\(C=A[a]\) and \(B=\mathbb D C\). Biduality identifies \(\mathbb D B\) with \(C\), so
\[
\mathbb D B\in{}^pD^{\le1},\qquad
\ell^N\,\pH^1((\mathbb D B)/\ell^n)=0
\quad(n\ge1).
\]
These are the two conditions defining \(B\in{}^pD^{\ge0}\) in \cite[Construction 4.3]{BH}. By
duality, \(C\in{}^{p^+}D^{\le0}\), hence \(A\in{}^{p^+}D^{\le a}\). Equivalently, \(\pH^{a+1}(A)\)
is locally bounded torsion \cite[Proposition 4.6(5)]{BH}. The integral \(p\)- and \(p^+\)-structures
agree after inverting \(\ell\) \cite[Theorem 4.11]{BH}, giving \(TF=A[1/\ell]\in{}^pD^{\le
a}(M_0,\mathbb Q_\ell)\).

\smallskip
\noindent\emph{Step 4. Prove the lower perverse bound.}
Suppose \(F\in{}^pD^{\ge a}(M,\mathbb Q_\ell)\). Self-duality of the perverse \(t\)-structure with
rational coefficients gives \(\mathbb DF\in{}^pD^{\le-a}\). Choose an integral model
\(H\in{}^pD^{\le-a}(M,\mathbb Z_\ell)\) of \(\mathbb DF\) and put \(E=\mathbb DH\), an integral
model of \(F\). By \cite[Proposition 4.6(2)]{BH}, \(H_n=H/\ell^n\) belongs to \({}^pD^{\le-a}\).
Compatibility of Verdier duality with reduction of coefficients, together with perverse duality with
finite coefficients, gives
\[
E_n=E/\ell^n\simeq\mathbb D_{\mathbb Z/\ell^n}(H_n)
\in{}^pD^{\ge a}.
\]
Perverse \(t\)-exactness of \(\Psi\) and \(\Phi\) with finite coefficients gives \(\Psi E_n,\Phi
E_n\in{}^pD^{\ge a}\). Since this subcategory is closed under fibers, the triangle in Step 2 yields
\(TE_n\in{}^pD^{\ge a}\).

Applying Verdier duality with finite coefficients and its compatibility with reduction, we obtain
\[
(\mathbb D(TE))/\ell^n
\simeq\mathbb D_{\mathbb Z/\ell^n}(TE_n)
\in{}^pD^{\le-a}.
\]
Thus \(\mathbb D(TE)\in{}^pD^{\le-a}(M_0,\mathbb Z_\ell)\), so \(TE\in{}^{p^+}D^{\ge a}\) by
definition of the dual integral structure. Rationalization gives \(TF\in{}^pD^{\ge a}(M_0,\mathbb
Q_\ell)\).
\end{proof}

\noindent\textit{Remark.} The lower bound holds for every \(F\in{}^pD^{\ge a}(M,\mathbb Q_\ell)\),
by perverse exactness with finite coefficients and Verdier duality. Condition \eqref{eq:uniform}
supplies the upper bound.

\subsection{Gysin maps and duality}\label{sec:purity}

Let \(\ell\) be any prime, including \(\ell=p\), and let \(k:M_0\hookrightarrow M\) be an effective
Cartier divisor. We work in the Zariski-constructible derived categories and assume that the ambient
spaces are qcqs when using rational coefficients. For \(n\ge1\), write \(\Lambda_n=\mathbb
Z/\ell^n\). The localized divisor class of \(M_0\) \cite[Variant 3.1.5]{LRZ} gives a canonical
morphism
\[
\operatorname{cl}_{k,n}:k_*\Lambda_{n,M_0}\longrightarrow\Lambda_{n,M}(1)[2].
\]
For a constructible \(\Lambda_n\)-complex \(F\) on \(M\), tensoring with \(F(-1)[-2]\) and applying
the projection formula give a morphism \(k_*(k^*F(-1)[-2])\to F\). Its adjoint under \((k_*,k^!)\)
defines the natural Gysin transformation
\begin{equation}\label{eq:gysin}
\tau_{k,F}:k^*F(-1)[-2]\longrightarrow k^!F.
\end{equation}
We use the same notation for the integral and rational Gysin transformations constructed in the
proof of Lemma~\ref{lem:purity}.

\begin{lemma}\label{lem:purity}
Let \(g:W\to M\) be projective, with \(W\) smooth over \(K\) of pure dimension \(s\), and suppose
that \(W_0=W\times_M M_0\) is a smooth effective Cartier divisor in \(W\). Put \(K_g=Rg_*\mathbb
Q_\ell[s]\). Then \(\tau_{k,K_g}\) is an isomorphism. The same holds for every derived direct
summand of a shift or Tate twist of \(K_g\), for every prime \(\ell\), including \(\ell=p\).
\end{lemma}
\begin{proof}
\smallskip
\noindent\emph{Step 1. Compare the Gysin maps under proper pushforward.}
Consider the Cartesian square
\[
\begin{matrix}
W_0&\xrightarrow{k'}&W\\
\downarrow g_0&&\downarrow g\\
M_0&\xrightarrow{k}&M.
\end{matrix}
\]
For finite coefficients \(\Lambda_n\), ordinary and extraordinary proper base change give
\[
\begin{aligned}
k^*Rg_*\Lambda_{n,W}&\simeq Rg_{0*}k'^*\Lambda_{n,W},\\
k^!Rg_*\Lambda_{n,W}&\simeq Rg_{0*}k'^!\Lambda_{n,W}.
\end{aligned}
\]
For a constructible \(\Lambda_n\)-complex \(F\) on \(W\), \cite[Lemma 7.3.4]{LRZ} gives the
canonical equivalence \(Rg_{0*}k'^!F\xrightarrow{\sim}k^!Rg_*F\). The second displayed isomorphism
is its inverse for \(F=\Lambda_{n,W}\). Indeed, closed base change gives \(g^*k_*\simeq k'_*g_0^*\),
and taking right adjoints gives this equivalence. Its adjoint under \((k_*,k^!)\) is the composite
\[
k_*Rg_{0*}k'^!F\simeq Rg_*k'_*k'^!F
\longrightarrow Rg_*F,
\]
where the last arrow comes from the counit for \(k'\). Under these identifications, the Gysin
transformations fit into the commutative square
\begin{equation}\label{eq:gysin-basechange}
\begin{CD}
k^*Rg_*\Lambda_{n,W}(-1)[-2]
@>{\tau_{k,Rg_*\Lambda_{n,W}}}>>
k^!Rg_*\Lambda_{n,W}\\
@V{\sim}VV @V{\sim}VV\\
Rg_{0*}k'^*\Lambda_{n,W}(-1)[-2]
@>{Rg_{0*}\tau_{k',\Lambda_{n,W}}}>>
Rg_{0*}k'^!\Lambda_{n,W}.
\end{CD}
\end{equation}
Indeed, both \(k\) and \(k'\) are effective Cartier divisors, so
\(g^*\operatorname{cl}_{k,n}=\operatorname{cl}_{k',n}\) by \cite[Lemma 3.1.6]{LRZ}. Taking the
\((k_*,k^!)\)-adjoints of the two composites in \eqref{eq:gysin-basechange}, and then their
\((g^*,Rg_*)\)-adjoints, identifies them through this equality and the projection formula.

\smallskip
\noindent\emph{Step 2. Prove purity on the smooth source.}
The spaces \(W\) and \(W_0\) are smooth of dimensions \(s\) and \(s-1\), respectively. Their
dualizing complexes with \(\Lambda_n\)-coefficients are
\[
\omega_W\simeq\Lambda_{n,W}(s)[2s],\qquad
\omega_{W_0}\simeq\Lambda_{n,W_0}(s-1)[2s-2]
\]
by \cite[Corollary 7.1.2]{LRZ}. The trace map for \(k'\), applied to dualizing complexes, has
adjoint \(\omega_{W_0}\xrightarrow{\sim}k'^!\omega_W\). For \(k'\), the divisor class used above
equals the cycle class for a local complete intersection of codimension one by \cite[Lemma
3.3.6]{LRZ}, applied to the single divisor \(W_0\). After the twists and shifts in
\cite[Construction 7.2.5]{LRZ}, the map defined by this cycle class agrees with the trace for \(k'\)
on dualizing complexes by \cite[Lemma 7.2.6]{LRZ}. Substituting these formulas and removing the
common twist and shift gives
\[
\tau_{k',\Lambda_{n,W}}:
\Lambda_{n,W_0}(-1)[-2]\xrightarrow{\sim}k'^!\Lambda_{n,W}.
\]
Applying \(Rg_{0*}\) and \eqref{eq:gysin-basechange}, we obtain that \(\tau_{k,Rg_*\Lambda_{n,W}}\)
is an isomorphism for every \(n\ge1\). The duality and trace results used here hold in mixed
characteristic with \(\mathbb Z/\ell^n\)-coefficients for every prime \(\ell\), including
\(\ell=p\).

\smallskip
\noindent\emph{Step 3. Pass from finite coefficients to integral coefficients.}
To construct the integral divisor class, we first work in the stable \(\infty\)-category of
complexes of abelian \'etale sheaves on \(M\). Write \(U=M\setminus M_0\) and \(j:U\hookrightarrow
M\), and let \(\mathbb G_m\) denote the \'etale sheaf of invertible functions. The line bundle
\(\mathcal O_M(M_0)\), with its canonical trivialization on \(U\), defines a Picard class with
support in \(M_0\) \cite[Lemma 3.1.1]{LRZ}. At the level of mapping spaces, the trivialization is a
nullhomotopy of the restriction of the line bundle to \(U\); the cofiber sequence \(j_!\mathbb
Z_U\to\mathbb Z_M\to k_*\mathbb Z_{M_0}\) therefore gives a morphism \(d_k:k_*\mathbb
Z_{M_0}\to\mathbb G_{m,M}[1]\). Let \(\mathcal M_\bullet=(\mu_{\ell^n})_n\) and \(\mathcal
G_\bullet=(\mathbb G_m)_n\), both with transition maps given by the \(\ell\)-th power. For a sheaf
\(F\), let \(\operatorname{const}(F)\) be the tower with value \(F\) and identity transition maps.
The Kummer sequences form a short exact sequence of towers, in which \([\ell^n]\) is the
\(\ell^n\)-th power map at level \(n\):
\[
0\longrightarrow\mathcal M_\bullet
\longrightarrow\mathcal G_\bullet
\xrightarrow{[\ell^n]}\operatorname{const}(\mathbb G_m)
\longrightarrow0.
\]
The squares commute as maps of sheaves. Taking the connecting morphism in the derived
\(\infty\)-category of towers and composing its shift with \(\operatorname{const}(d_k)\) gives a
morphism of towers \(\operatorname{const}(k_*\mathbb Z_{M_0})\to\mathcal M_\bullet[2]\).

Let \(\mathcal R_\bullet=(\Lambda_n)_n\) be the tower of rings whose transition maps are the
quotient maps. The tower \(\mathcal M_\bullet\) is an \(\mathcal R_\bullet\)-module. The free-module
adjunction in the stable \(\infty\)-category of towers turns the preceding morphism into an
\(\mathcal R_\bullet\)-linear morphism
\[
\operatorname{cl}_{k,\bullet}:
(k_*\Lambda_{n,M_0})_n\longrightarrow
(\mu_{\ell^n,M}[2])_n,
\]
using \(\Lambda_n\otimes_{\mathbb Z}^{\mathbf L}k_*\mathbb Z_{M_0}\simeq k_*\Lambda_{n,M_0}\). Both
towers are derived-cartesian: derived extension of scalars from \(\Lambda_{n+1}\) to \(\Lambda_n\)
identifies the term at level \(n+1\) with the term at level \(n\). For the target, this follows
because \(\mu_{\ell^{n+1}}\) is
locally free of rank one over \(\Lambda_{n+1}\), and the power map induces its reduction to
\(\mu_{\ell^n}\). The terms in both towers are bounded constructible complexes of finite
Tor-dimension. The tower construction supplies the coherence required for a morphism in the inverse
limit of categories in \cite[Remark 3.33]{BH}, and hence defines an integral class
\(\operatorname{cl}_{k,\mathbb Z_\ell}:k_*\mathbb Z_{\ell,M_0}\to\mathbb Z_{\ell,M}(1)[2]\). Its
reduction modulo \(\ell^n\) is the Kummer boundary of the line bundle with its trivialization on
\(U\), hence is \(\operatorname{cl}_{k,n}\) by \cite[Definition 3.1.4 and Variant 3.1.5]{LRZ}.

For an integral constructible complex \(E\), tensor this fixed class with \(E(-1)[-2]\), apply the
integral projection formula, and take the \((k_*,k^!)\)-adjoint. This defines
\(\tau_{k,E}:k^*E(-1)[-2]\to k^!E\), naturally in \(E\). Put
\(E_n=E\otimes_{\mathbb Z_\ell}^{\mathbf L}\Lambda_n\). Compatibility with reduction of coefficients
\cite[Theorem 3.36(1), (3), (5)]{BH} identifies the reduction of \(\tau_{k,E}\) with
\(\tau_{k,E_n}\); part~(5) applies to \(k^!\) at every prime since \(k\) is finite. In particular,
these reductions give the commutative square
\begin{equation}\label{eq:gysin-reduction}
\begin{CD}
\bigl(k^*E_{n+1}(-1)[-2]\bigr)
\otimes_{\Lambda_{n+1}}^{\mathbf L}\Lambda_n
@>{\tau_{k,E_{n+1}}\otimes^{\mathbf L}1}>>
(k^!E_{n+1})\otimes_{\Lambda_{n+1}}^{\mathbf L}\Lambda_n\\
@V{\sim}VV @V{\sim}VV\\
k^*E_n(-1)[-2]
@>{\tau_{k,E_n}}>>
k^!E_n.
\end{CD}
\end{equation}
Take \(E=Rg_*\mathbb Z_\ell[s]\). Proper pushforward commutes with coefficient reduction, so the
reduction of \(\tau_{k,E}\) modulo every \(\ell^n\) is an isomorphism by Steps 1 and 2. Its cone is
derived \(\ell\)-complete and has zero reduction modulo every \(\ell^n\), so derived Nakayama makes
\(\tau_{k,E}\) an isomorphism. Inverting \(\ell\) gives
\(\tau_{k,K_g}:k^*K_g(-1)[-2]\xrightarrow{\sim}k^!K_g\). The integral natural transformation also
rationalizes: on the qcqs ambient space, rational morphisms are obtained from integral morphisms by
inverting \(\ell\), so the resulting transformation is natural for all rational morphisms.

\smallskip
\noindent\emph{Step 4. Pass to derived direct summands.}
Let \(F\) be a derived direct summand of \(K_g[a](t)\), with inclusion \(u:F\to K_g[a](t)\) and
retraction \(v:K_g[a](t)\to F\) satisfying \(vu=\operatorname{id}_F\). Naturality makes
\(\tau_{k,F}\) a retract of the isomorphism \(\tau_{k,K_g[a](t)}\), and therefore an isomorphism.
\end{proof}

\begin{lemma}\label{lem:adic-proper-duality}
Let \(f:X\to Y\) be a proper morphism of rigid analytic spaces over \(K\), with \(Y\) qcqs. For
every prime \(\ell\) and every \(E\in\Db(X,\Z_\ell)\), the proper duality maps with finite
coefficients induce a natural isomorphism
\[
 Rf_*\mathbb D_{X,\Z_\ell}E
 \xrightarrow{\sim}\mathbb D_{Y,\Z_\ell}(Rf_*E).
\]
Its reduction modulo \(\ell^n\) is the proper duality isomorphism with \(\Z/\ell^n\)-coefficients.
Rationalization gives the corresponding natural isomorphism with \(\Q_\ell\)-coefficients.
\end{lemma}
\begin{proof}
Put \(\Lambda_n=\Z/\ell^n\), and write \(\omega_{X,n}\) and \(\omega_{Y,n}\) for the dualizing
complexes with \(\Lambda_n\)-coefficients, and \(\omega_{X,\Z_\ell}\) and \(\omega_{Y,\Z_\ell}\) for
the integral dualizing complexes. The proof of \cite[Theorem 3.36]{BH} gives identifications
\(\omega_{X,m}\otimes_{\Lambda_m}^{\mathbf L}\Lambda_n\simeq\omega_{X,n}\) for \(m\ge n\),
compatible with successive reductions, and likewise on \(Y\). Proper pushforward commutes with these
reductions.

We first check that the proper traces \(\operatorname{Tr}_{f,n}:Rf_*\omega_{X,n}\to\omega_{Y,n}\)
are compatible with reduction. The trace of a closed immersion \(i\) is the adjunction counit under
the fixed identification \(\omega_{X,n}\simeq i^!\omega_{Y,n}\), so it commutes with coefficient
reduction. For smooth trace, reduction preserves the identities for composition and base change, the
\'etale counit, and the algebraic trace on the projective line. The uniqueness in \cite[Theorem
6.1.1]{LRZ} therefore gives compatibility for smooth trace. For a proper morphism with smooth source
and separated taut target, the trace is the composite of the trace for the closed graph immersion
and the trace for the smooth projection \cite[Construction 7.3.1]{LRZ}. This is the smooth-source
normalization in \cite[Theorem 7.4.1(2)]{LRZ}. Consequently, reduction preserves all the
normalizations, composition identities, and \'etale locality characterizing the proper traces in
\cite[Theorem 7.4.1]{LRZ}. Its uniqueness gives
\(\operatorname{Tr}_{f,m}\otimes_{\Lambda_m}^{\mathbf L}\Lambda_n=\operatorname{Tr}_{f,n}\).

By \cite[Theorem 7.3.14]{LRZ}, the internal Hom complex \(R\mathcal
Hom(Rf_*\omega_{X,n},\omega_{Y,n})\) lies in \(D^{\ge0}\). Applying derived global sections, which
preserve this lower bound, gives \(\operatorname{Ext}^{-q}(Rf_*\omega_{X,n},\omega_{Y,n})=0\) for
\(q>0\). Thus the mapping spaces containing these traces are discrete. The identifications under
successive reductions therefore give unique homotopies realizing the reduction identities, and
uniqueness makes these homotopies coherent. They define a morphism
\(\operatorname{Tr}_{f,\Z_\ell}:Rf_*\omega_{X,\Z_\ell}\to\omega_{Y,\Z_\ell}\) in the inverse limit
of categories in \cite[Remark 3.33]{BH}. Compose the natural evaluation map
\[
 Rf_*R\mathcal Hom(E,\omega_{X,\Z_\ell})
 \longrightarrow
 R\mathcal Hom(Rf_*E,Rf_*\omega_{X,\Z_\ell})
\]
with this trace. Compatibility of internal Hom and proper pushforward with coefficient reduction
\cite[Theorem 3.36(1), (3)]{BH} identifies its reduction with the duality map of \cite[Definition
7.5.4]{LRZ}, which is an isomorphism by \cite[Theorem 7.5.18]{LRZ}. Its cone is derived
\(\ell\)-complete and has zero reduction modulo \(\ell\), so derived Nakayama proves the integral
assertion. Localizing the integral functors and natural transformation at \(\ell\) gives the
rational assertion.
\end{proof}

Return to the setting and notation of Lemma~\ref{lem:purity}. By Lemma~\ref{lem:adic-proper-duality}
and smooth duality on \(W\), we have \(\mathbb D_MK_g\simeq K_g(s)\). Thus the class of shifted and
Tate-twisted derived summands of \(K_g\) is preserved by Verdier duality. For such an object \(F\),
put \(T=k^*[-1]\). Closed duality and Lemma~\ref{lem:purity} give
\begin{equation}\label{eq:dualcut}
\begin{aligned}
\mathbb D_{M_0}(TF)
&\simeq k^!\mathbb D_MF[1]\\
&\simeq k^*\mathbb D_MF(-1)[-1]\\
&=T(\mathbb D_MF)(-1).
\end{aligned}
\end{equation}
For \(r\) successive transverse cuts defining the graph as in Proposition~\ref{lem:approx}, let
\(M_r\) be the final closed subspace and \(\gamma:M_r\hookrightarrow M\) the composite immersion.
Put \(T^{(r)}=\gamma^*[-r]\). At each stage the source and its hypersurface cut are smooth, so
Lemma~\ref{lem:purity} applies to the proper direct images and their derived summands. Iterating
\eqref{eq:dualcut} gives
\begin{equation}\label{eq:dualgraph}
\mathbb D(T^{(r)}F)\simeq T^{(r)}(\mathbb DF)(-r).
\end{equation}
We use these comparisons to check the strict cosupport conditions for intermediate extension after
restriction to the analytic graph.

\section{Projective decomposition}\label{sec:decomposition}

\subsection{Relative hard Lefschetz}\label{sec:lefschetz}

For a projective morphism with smooth source, we obtain relative hard Lefschetz by restricting the
algebraic Lefschetz isomorphisms to the analytic graph supplied by Proposition~\ref{lem:approx}.
Lemma~\ref{lem:bounds} shows that the perverse summands remain perverse after each successive
hypersurface cut.

\begin{theorem}\label{thm:hl}
Let \(f:X\to Y\) be a proper morphism with \(Y\) qcqs, equipped with a global relatively ample line
bundle \(H\) on \(X\), and assume that \(X\) is smooth over \(K\) of pure dimension \(d\). Fix a
prime \(\ell\ne p\), put \(K_f=Rf_*\Q_\ell[d]\), and let \(\eta:K_f\to K_f(1)[2]\) be cup product
with the first Chern class \(c_1(H)\). Then the maps
\begin{equation}\label{eq:hl}
\eta^a:\pH^{-a}(K_f)\xrightarrow{\sim}\pH^a(K_f)(a)
\qquad(a\ge0)
\end{equation}
are isomorphisms.
Consequently, there is a derived decomposition \(K_f\simeq\bigoplus_a\pH^a(K_f)[-a]\) over \(K\).
\end{theorem}
\begin{proof}
\smallskip
\noindent\emph{Step 1. Algebraic approximation.}
The assertion that the canonical Lefschetz maps are isomorphisms is local on \(Y\). Restrict to an
affinoid open subset, over which a sufficiently high power of \(H\) gives a projective embedding,
and embed the affinoid as a closed subspace of a polydisc \(D\). Closed pushforward is fully
faithful and perverse t-exact, so we may prove the assertion for the composite map to \(D\). Apply
Proposition~\ref{lem:approx}, retaining the projective embedding and a positive tensor power
\(H^{\otimes N}\). We obtain a smooth algebraic \(K\)-variety \(S\), a projective morphism
\(g_0:W_0\to S\), and local cuts defining an analytic graph that recovers the original morphism. If
the graph has codimension \(r\), then \(W_0\) is smooth of pure dimension \(s=d+r\). On each qcqs
neighborhood supplied by the proposition, every intermediate source is smooth, and the next cut
function is smooth on that source.

\smallskip
\noindent\emph{Step 2. Algebraic hard Lefschetz.}
Put \(K_0=Rg_{0*}\Q_\ell[s]\) and \(P_{a,0}=\pH^a(K_0)\). Let \(\eta_0\) be cup product with the
first Chern class of the retained relatively ample line bundle on \(W_0\). The algebraic relative
hard Lefschetz theorem \cite[Theorem 6.2.10]{BBD} gives
\(\eta_0^a:P_{-a,0}\xrightarrow{\sim}P_{a,0}(a)\). These morphisms are defined over \(K\), and their
geometric base changes are isomorphisms, so they are isomorphisms over \(K\).
Lemma~\ref{lem:deligne} gives a derived decomposition
\[
K_0\simeq\bigoplus_a P_{a,0}[-a].
\]
In particular, each \(P_{a,0}[-a]\) is a direct summand of \(K_0\) in the rational derived category.
On a chosen qcqs analytic neighborhood of the graph, form the restricted analytifications from
integral models as in Section~\ref{sec:formal}, and then invert \(\ell\). For each algebraic
perverse object, replace an integral model by its integral perverse cohomology in degree zero; this
has the same rationalization. Analytification with finite coefficients is perverse \(t\)-exact by
\cite[Theorem 4.2(4)]{BH}. Construction~4.3 and the proof of Theorem~4.11 of \cite{BH} then show
that the object obtained by restriction and rationalization is perverse. Compatibility with proper
pushforward \cite[Proposition 3.16]{BH} identifies the ambient complex with the analytic direct
image, and the functor preserves the displayed decomposition. We use the same notation after
restriction.

\smallskip
\noindent\emph{Step 3. Restriction to the graph.}
Pass to a completed algebraic closure to apply the results on vanishing cycles with finite
coefficients. Write the bases obtained after successive cuts as \(M_0\supset M_1\supset\cdots\supset
M_r\), with \(M_0\) the chosen analytic neighborhood. For \(0\le j<r\), let \(h_j\) define the next
cut \(k_j:M_{j+1}\hookrightarrow M_j\), and put \(T_j=k_j^*[-1]\). For \(0\le j\le r\), let
\(W_j=W_0\times_{M_0}M_j\) and \(g_j:W_j\to M_j\) be the corresponding base changes, and put
\(K_j=Rg_{j*}\Q_\ell[s-j]\) and \(E_j=Rg_{j*}\Z_\ell[s-j]\). Proposition~\ref{lem:approx} makes each
\(W_j\) smooth of pure dimension \(s-j\), and makes \(h_j\circ g_j\) smooth for \(j<r\). For
\(j<r\), proper base change identifies \(T_jK_j\simeq K_{j+1}\).

We prove by induction that
\[
K_j\simeq\bigoplus_a P_{a,j}[-a],
\qquad
P_{a,j}=T_{j-1}\cdots T_0P_{a,0}\in\Perv(M_j,\Q_\ell),
\]
with the empty composite understood as the identity. The case \(j=0\) follows from Step 2. Suppose
the assertion holds at stage \(j\). For \(\Lambda_n=\Z/\ell^n\), coefficient reduction gives
\(K_{j,n}:=E_j\otimes_{\Z_\ell}^{\mathbf L}\Lambda_n\simeq Rg_{j*}\Lambda_{n,W_j}[s-j]\). Smooth
local acyclicity and compatibility with proper pushforward give \(\phi_{h_j}(K_{j,n})=0\) for every
\(n\), as in Proposition~\ref{prop:boundaryvanish}. Each \(P_{a,j}\) is a shifted rational derived
summand of \(K_j\), so Lemma~\ref{lem:thick} gives condition~\eqref{eq:uniform} for \(P_{a,j}\).
Lemma~\ref{lem:bounds} then shows that \(P_{a,j+1}:=T_jP_{a,j}\) is perverse. Applying \(T_j\) to
the decomposition of \(K_j\) gives \(K_{j+1}\simeq\bigoplus_aP_{a,j+1}[-a]\), completing the
induction. At the final graph, write \(\gamma:M_r\hookrightarrow M_0\) for the composite of the cuts
and put \(T^{(r)}=T_{r-1}\cdots T_0=\gamma^*[-r]\). On the part of the original base covered by this
neighborhood, we obtain
\begin{equation}\label{eq:transport}
\begin{aligned}
K_f&\simeq T^{(r)}K_0
     \simeq\bigoplus_a T^{(r)}P_{a,0}[-a],\\
\pH^a(K_f)&\simeq T^{(r)}P_{a,0}.
\end{aligned}
\end{equation}

\smallskip
\noindent\emph{Step 4. The Lefschetz maps after restriction.}
The retained algebraic line bundle restricts to \(H^{\otimes N}\). Since first Chern classes commute
with pullback, its cup-product operator restricts to \(N\eta\). To identify the induced maps on
perverse cohomology, choose the decomposition of \(K_0\) in Step 2 to induce the identity on its
perverse cohomology objects, and express \(\eta_0^a:K_0\to K_0(a)[2a]\) as a matrix relative to this
decomposition. Its component from perverse degree \(-a\) to degree \(a\) is the algebraic Lefschetz
isomorphism \(P_{-a,0}\to P_{a,0}(a)\). Every transported summand remains perverse, so applying
\(T^{(r)}\) to this component gives the induced map on \(\pH^{-a}(K_f)\) and \(\pH^a(K_f)(a)\). This
map is \(N^a\eta^a\) under \eqref{eq:transport} and is an isomorphism. Since \(N^a\) is invertible
in \(\Q_\ell\), equation~\eqref{eq:hl} follows.

\smallskip
\noindent\emph{Step 5. Descent and global splitting.}
Write \(q_Y:Y_C\to Y\) and \(q_X:X_C\to X\) for extension to the completed algebraic closure. Proper
pushforward and Verdier duality commute with this base change \cite[Proposition 3.24 and Theorem
3.36]{BH}. Base extension preserves the dimensions of constructible supports and detects vanishing
on geometric stalks. On the qcqs ambient spaces, finite constructible stratifications show that
this detection of vanishing remains valid after rationalization. The support criterion and duality
therefore make \(q_Y^*\)
conservative and perverse \(t\)-exact \cite[Definition 4.1, Construction 4.3, and Theorem 4.11]{BH}.
Naturality of the Kummer connecting morphisms gives \(q_X^*c_1(H)=c_1(q_X^*H)\), and proper base
change identifies the associated cup-product operators. Thus the Lefschetz maps over \(C\) are the
pullbacks of the canonical maps over \(K\), whose invertibility follows by conservativity. They are
defined globally by \(c_1(H)\), so local invertibility proves \eqref{eq:hl} on \(Y\). Applying
Lemma~\ref{lem:deligne} in the global rational derived category with its perverse \(t\)-structure
gives \(Rf_*\Q_\ell[d]\simeq\bigoplus_a\pH^a(Rf_*\Q_\ell[d])[-a]\).
\end{proof}

\subsection{Splitting along the boundary}\label{sec:support}

The algebraic decomposition by supports passes through the cuts defining the graph once the summands
given by intermediate extension retain their strict bounds on the boundary, both before and after
Verdier duality.

\begin{theorem}\label{thm:support}
Under the assumptions of Theorem~\ref{thm:hl}, let \(P_a=\pH^a(Rf_*\Q_\ell[d])\), with \(\ell\ne
p\). For every closed immersion \(i:Z\hookrightarrow Y\), the canonical boundary map
\[
c_Z(P_a):\pH^0(i^!P_a)\longrightarrow\pH^0(i^*P_a)
\]
is an isomorphism. Consequently, each \(P_a\) is a finite direct sum of intermediate extensions of
finite-rank local systems on smooth Zariski-locally closed subspaces of \(Y\).
\end{theorem}
\begin{proof}
\smallskip
\noindent\emph{Step 1. Approximation with a prescribed boundary.}
Fix \(i:Z\hookrightarrow Y\). The assertion is local on \(Y\), so restrict to an affinoid open
subset, embed it in a polydisc, and apply Proposition~\ref{prop:boundaryvanish}. This gives a
projective algebraic morphism \(g_0:W_0\to S\), a closed boundary \(i_0:Z_0\hookrightarrow S\), and
a finite resolution diagram for \(g_0^{-1}(Z_0)\), all in one approximation. The
proposition supplies both vanishing statements in \eqref{eq:twovanishing} at every successive cut of
this diagram. Let \(r\) be the codimension of the analytic graph and put \(s=\dim W_0=d+r\). Write
\(K_0=Rg_{0*}\Q_\ell[s]\), \(P_{a,0}=\pH^a(K_0)\), and let \(j_0\) be the open complement of
\(i_0\).

\smallskip
\noindent\emph{Step 2. Algebraic boundary splitting.}
The algebraic decomposition theorem \cite[Theorem 6.2.5]{BBD} makes \(P_{a,0}\) geometrically
semisimple. Over an algebraic closure, each simple constituent is either supported on \(Z_0\) or is
the intermediate extension of its restriction to \(S\setminus Z_0\). Hence the canonical boundary
map \(c_{Z_0}(P_{a,0})\) becomes an isomorphism after geometric base change. It is defined over
\(K\), so it is an isomorphism over \(K\). Lemma~\ref{lem:boundary} gives the canonical
decomposition
\begin{equation}\label{eq:algsupport}
P_{a,0}\simeq I_{a,0}\oplus i_{0*}Q_{a,0},
\qquad
I_{a,0}=j_{0!*}j_0^*P_{a,0},
\end{equation}
where \(Q_{a,0}\) is perverse on \(Z_0\). The summand \(I_{a,0}\) satisfies
\[
i_{0*}i_0^*I_{a,0}\in{}^pD^{\le-1},
\qquad
i_{0*}i_0^*\mathbb D I_{a,0}\in{}^pD^{\le-1}.
\]
By \eqref{eq:strict}, these conditions express the absence of nonzero quotients and subobjects
supported on \(Z_0\). The restricted analytification of Section~\ref{sec:formal} preserves the
decomposition and these bounds.

\smallskip
\noindent\emph{Step 3. Induction through the cuts.}
On a chosen qcqs analytic neighborhood, pass to a completed algebraic closure and use
\(M_j,W_j,g_j,h_j,k_j,K_j,K_{j,n},T_j\), and \(P_{a,j}\) as in the proof of Theorem~\ref{thm:hl}.
Let \(i_j:Z_j\hookrightarrow M_j\) be the boundary at stage \(j\), with complementary open immersion
\(j_j:M_j\setminus Z_j\hookrightarrow M_j\). We prove inductively that
\[
\begin{gathered}
P_{a,j}\simeq I_{a,j}\oplus i_{j*}Q_{a,j},
\qquad I_{a,j}\in\Perv(M_j,\Q_\ell),\quad
Q_{a,j}\in\Perv(Z_j,\Q_\ell),\\
i_{j*}i_j^*I_{a,j}\in{}^pD^{\le-1},
\qquad
i_{j*}i_j^*\mathbb D I_{a,j}\in{}^pD^{\le-1}.
\end{gathered}
\]
We also require \(I_{a,j}[-a]\) and \(i_{j*}Q_{a,j}[-a]\) to remain derived retracts of \(K_j\). These
assertions hold at \(j=0\) by Step 2 and the algebraic derived decomposition.

\smallskip
\noindent\emph{Step 4. Perversity and the first boundary bound.}
Assume the assertions at stage \(j\), and consider \(T_j=k_j^*[-1]\).
Proposition~\ref{prop:boundaryvanish} gives \(\phi_{h_j}(K_{j,n})=0\) for every \(n\). The objects
\(I_{a,j}\) and \(i_{j*}Q_{a,j}\) are perverse and are shifted rational retracts of \(K_j\), so
Lemmas~\ref{lem:thick} and~\ref{lem:bounds} show that \(I_{a,j+1}:=T_jI_{a,j}\) and
\(T_j(i_{j*}Q_{a,j})\) are perverse. Closed base change shows that the latter is supported on
\(Z_{j+1}\).

Put \(B_j=i_{j*}i_j^*K_j\). The second vanishing statement in Proposition~\ref{prop:boundaryvanish}
gives \(\phi_{h_j}(i_{j*}i_j^*K_{j,n})=0\) for every \(n\). The complex \(i_{j*}i_j^*I_{a,j}\) is a
shifted rational retract of \(B_j\), and the induction hypothesis places it in \({}^pD^{\le-1}\).
Lemmas~\ref{lem:thick} and~\ref{lem:bounds} preserve this upper bound under \(T_j\). Closed base
change therefore gives
\begin{equation}\label{eq:boundtransfer}
i_{j+1*}i_{j+1}^*I_{a,j+1}
\simeq T_j(i_{j*}i_j^*I_{a,j})
\in{}^pD^{\le-1}.
\end{equation}

\smallskip
\noindent\emph{Step 5. The dual boundary bound.}
Smoothness of \(W_j\) and compatibility of proper pushforward with Verdier duality give \(\mathbb
DK_j\simeq K_j(s-j)\). Thus \(\mathbb DI_{a,j}\) is a shifted and Tate-twisted rational retract of
\(K_j\), and its boundary restriction is a corresponding retract of \(B_j\). The induction
hypothesis gives \(i_{j*}i_j^*\mathbb DI_{a,j}\in{}^pD^{\le-1}\). Applying the boundary vanishing
argument of Step 4 yields \(T_j(i_{j*}i_j^*\mathbb DI_{a,j})\in{}^pD^{\le-1}\). The next graph cut
is a Cartier divisor whose pullback to \(W_j\) is a smooth Cartier divisor, so
Lemma~\ref{lem:purity} and \eqref{eq:dualcut} apply to these retracts. They give \(\mathbb
D(T_jI_{a,j})\simeq T_j(\mathbb DI_{a,j})(-1)\). Combining this with closed base change, we obtain
\[
i_{j+1*}i_{j+1}^*\mathbb DI_{a,j+1}
\simeq T_j(i_{j*}i_j^*\mathbb DI_{a,j})(-1)
\in{}^pD^{\le-1},
\]
since Tate twists preserve the perverse t-structure.

\smallskip
\noindent\emph{Step 6. Intermediate extension after restriction.}
The perverse sheaf \(I_{a,j+1}\) satisfies both strict boundary conditions, hence \(I_{a,j+1}\simeq
j_{j+1!*}j_{j+1}^*I_{a,j+1}\). Applying \(T_j\) to the decomposition at stage \(j\) gives
\[
P_{a,j+1}\simeq I_{a,j+1}\oplus T_j(i_{j*}Q_{a,j}).
\]
Both summands are perverse, and the second is supported on \(Z_{j+1}\). Closed pushforward therefore
identifies it with \(i_{j+1*}Q_{a,j+1}\) for a perverse sheaf \(Q_{a,j+1}\) on \(Z_{j+1}\).
Lemma~\ref{lem:boundary} makes \(c_{Z_{j+1}}(P_{a,j+1})\) an isomorphism. The shifted summands
remain retracts of \(T_jK_j\simeq K_{j+1}\), completing the induction.

\smallskip
\noindent\emph{Step 7. Descent to the rigid base.}
At the final graph cut, \eqref{eq:transport} identifies \(P_{a,r}\) with \(\pH^a(K_f)\). The
induction shows that its canonical boundary map for \(Z\) is an isomorphism. Since closed
extraordinary pullback and duality commute with extension to the completed algebraic closure, this
assertion descends to \(K\). Invertibility can be checked locally, so \(c_Z(P_a)\) is an isomorphism
on \(Y\). Applying the construction for each closed \(Z\) and then
Corollary~\ref{cor:support-decomposition} expresses \(P_a\) as a finite direct sum of intermediate
extensions of finite-rank local systems on smooth Zariski-locally closed strata. Together with the
global derived splitting of Theorem~\ref{thm:hl}, this gives the projective decomposition for smooth
sources and \(\ell\ne p\).
\end{proof}

\subsection{Coefficients in the residue characteristic}\label{sec:allprimes}

To extend the decomposition for smooth sources proved in Sections~\ref{sec:lefschetz}
and~\ref{sec:support} to \(\Q_p\)-coefficients, we keep the algebraic approximation and analytic
graph fixed for all coefficient primes. The algebraic complexes arise from a single Betti complex
with rational coefficients, so the dimensions of their ordinary stalk cohomology are independent of
the prime. The
support criterion transfers to \(p\) the upper perverse bounds obtained after restriction to the
graph; applying the Gysin comparison of Section~\ref{sec:purity} to the algebraic duals gives the
lower bounds.

\begin{lemma}\label{lem:betti}
Let \(g_0:W_0\to S\) be a projective morphism of \(K\)-schemes of finite type, with \(W_0\) smooth
of pure dimension \(s\). Fix a closed immersion \(i_0:Z_0\hookrightarrow S\), with complementary
open immersion \(j_0:S\setminus Z_0\hookrightarrow S\). For every prime \(\ell\), put
\(K_{0,\ell}=Rg_{0*}\Q_\ell[s]\), \(P_{a,\ell}=\pH^a(K_{0,\ell})\), and
\(I_{a,\ell}=j_{0!*}j_0^*P_{a,\ell}\). At each geometric closed point of \(S\), the
dimensions of ordinary stalk cohomology are independent of \(\ell\) for
\begin{equation}\label{eq:bettilist}
P_{a,\ell},\quad \mathbb DP_{a,\ell},\quad
I_{a,\ell},\quad \mathbb DI_{a,\ell},\quad
i_{0*}i_0^*I_{a,\ell},\quad i_{0*}i_0^*\mathbb DI_{a,\ell}.
\end{equation}
\end{lemma}
\begin{proof}
Choose an abstract embedding \(\overline K\hookrightarrow\mathbf C\), where \(\mathbf C\) is the
field of complex numbers. Write \(g_{0,\mathbf C}:W_{0,\mathbf C}\to S_{\mathbf
C}\) for the base change along this embedding, and consider the Betti complex \(K_B=R(g_{0,\mathbf
C})_*\Q[s]\) on the underlying complex analytic space of \(S_{\mathbf C}\). Take its perverse
cohomology, Verdier duals, intermediate extensions, and restrictions to the prescribed boundary in
the constructible derived category on the complex variety \(S_{\mathbf C}\). These constructions are
defined over \(\Q\) and commute with extension of the coefficient field. In particular, intermediate
extension is the image of the natural morphism between the perverse extensions, and formation of
this image commutes with coefficient extension.

The comparison between algebraic and classical cohomology with finite coefficients extends to
constructible \(\Z_\ell\)- and \(\Q_\ell\)-complexes, compatibly with the cohomological operations,
by \cite[\S6.1.2, (B'), (B''), and (C')]{BBD}. The support and
cosupport conditions make this comparison perverse \(t\)-exact, so it commutes with perverse
cohomology and intermediate extension. Moreover, the boundary functors \(i_0^*\) and \(i_{0*}\)
commute with flat extension of the Betti coefficient field from \(\Q\) to \(\Q_\ell\). These
compatibilities identify all the \(\Q_\ell\)-realizations in \eqref{eq:bettilist} with the
corresponding geometric algebraic complexes. They apply to every prime, including \(p\), since the
varieties have characteristic zero. Each ordinary stalk cohomology group is therefore obtained from
a finite-dimensional rational Betti group by extension of scalars, proving the assertion.
\end{proof}

\begin{lemma}\label{lem:rational-support}
Let \(V=\Spa(A_V)\) be an affinoid rigid space over \(K\), and put \(\mathcal V=\Spec A_V\).
We use the affinoid comparison of \cite[Lemma 3.34]{BH} to identify integral Zariski-constructible
complexes on \(V\) with the corresponding complexes on \(\mathcal V\). In particular, supports and
their dimensions may be studied on \(\mathcal V\). Choose finitely many coefficient primes and,
for each prime \(\ell\), finitely many integral constructible complexes on \(\mathcal V\). Denote
these complexes generically by \(G_\ell\), and write \(F_\ell=G_\ell[1/\ell]\), using the same
symbols for the corresponding objects on \(V\). There is a common finite stratification
\(\mathcal V=\coprod_b S_b\) by connected
regular affine locally closed subschemes such that every \(G_\ell|_{S_b}\) is pro-\'etale locally
constant with a fixed perfect \(\Z_\ell\)-complex \(C_{\ell,b}\) as its value. Put \(e_b=\dim
S_b=\dim\overline{S_b}\) and \(M_{\ell,b}^q=H^q(C_{\ell,b})\). For every integer \(a\), the
following conditions are equivalent:
\begin{align}
&F_\ell\in{}^pD^{\le a}(V,\Q_\ell);\label{eq:stratified-upper}\\
&M_{\ell,b}^q[1/\ell]=0
  \quad\text{whenever }q>a-e_b;\label{eq:stratified-modules}\\
&\dim\overline{\supp\cH^q(F_\ell)}\le a-q
  \quad\text{for every }q.\label{eq:supporttest}
\end{align}
\end{lemma}
\begin{proof}
Integral affinoid comparison \cite[Lemma 3.34]{BH} allows us to work on \(\mathcal V\). By
\cite[Proposition 6.6.11]{BS}, the finitely many complexes admit finite constructible
stratifications on which they are pro-\'etale locally constant and perfect. Take a common
refinement, and use excellence and finite dimension of \(A_V\) to refine it into connected regular
affine strata compatible with a finite filtration by closed subschemes. On each stratum, \cite[Lemma
6.6.8]{BS} gives a fixed perfect value \(C_{\ell,b}\). A connected regular stratum is irreducible,
and in this Jacobson setting its dimension equals the dimension of its closure.

For finite coefficients, \cite[Definition 4.1]{BH}, together with Construction~4.3 and Example~4.4
for integral coefficients, identifies the upper perverse bound for a locally constant complex on a
regular stratum of dimension \(e_b\) with the ordinary bound \(H^q=0\) for \(q>a-e_b\). These bounds
glue along the chosen finite filtration by closed subschemes as in \cite[proof of Proposition
4.6(1)]{BH}. By \cite[Theorem 4.11]{BH}, affinoid comparison identifies the resulting localized
structure with the rational perverse structure on the whole affinoid \(V\).

\(\eqref{eq:stratified-modules}\Rightarrow\eqref{eq:stratified-upper}\).
Fix \(\ell\), and put \(G_b=G_\ell|_{S_b}\) and
\(c_b=a-e_b\). Consider the ordinary truncation triangle
\[
 \tau^{\le c_b}G_b\longrightarrow G_b
 \longrightarrow\tau^{>c_b}G_b
 \longrightarrow(\tau^{\le c_b}G_b)[1].
\]
The complex \(\tau^{>c_b}G_b\) has perfect value \(\tau^{>c_b}C_{\ell,b}\). Its cohomology modules
are finitely generated and become zero after rationalization, so one power \(\ell^{N_\ell}\)
annihilates all of them, for every stratum and degree. The same power annihilates the corresponding
locally constant cohomology sheaves. Choose \(d\) bounding the number of their possibly nonzero
cohomological degrees. Put \(E_b=\tau^{>c_b}G_b\). Its Postnikov filtration consists of the ordinary
truncations, with triangles
\[
 \tau^{\le q-1}E_b\longrightarrow\tau^{\le q}E_b
 \longrightarrow\cH^q(E_b)[-q]
 \longrightarrow(\tau^{\le q-1}E_b)[1].
\]
Since the third term is annihilated by \(\ell^{N_\ell}\), the exact Hom sequence shows that
\(\ell^{N_\ell}\mathrm{id}_{\tau^{\le q}E_b}\) factors through \(\tau^{\le q-1}E_b\).
Thus, if \(\ell^u\) annihilates \(\tau^{\le q-1}E_b\), then \(\ell^{u+N_\ell}\) annihilates
\(\tau^{\le q}E_b\). The truncation changes only in degrees with nonzero cohomology, so induction
over at most \(d\) such degrees shows that \(\ell^{dN_\ell}\) annihilates \(E_b\). Thus the
truncation maps become isomorphisms after rationalization \cite[Remark 4.9]{BH}.

At each step of the chosen finite filtration of \(\mathcal V=\Spec A_V\) by closed subschemes, write
\(j\) for the open stratum and \(i\) for its closed complement. The localization triangle is
\(j_!j^*G_\ell\to G_\ell\to i_*i^*G_\ell\to j_!j^*G_\ell[1]\). Both \(j_!\) and \(i_*\) preserve
scalar annihilators and upper perverse bounds. Induction along the finite filtration therefore gives
\eqref{eq:stratified-upper} after rationalization.

\(\eqref{eq:stratified-upper}\Rightarrow\eqref{eq:stratified-modules}\).
Successive \(*\)-restrictions preserve upper perverse bounds, and the shifted ordinary bound on each
locally constant restriction gives \eqref{eq:stratified-modules}.

\(\eqref{eq:stratified-modules}\Leftrightarrow\eqref{eq:supporttest}\).
On \(S_b\), the rational cohomology sheaf \(\cH^q(F_\ell)\) is locally constant with fiber
\(M_{\ell,b}^q[1/\ell]\), so it is either zero everywhere or nonzero everywhere. Consequently,
\[
\dim\overline{\supp\cH^q(F_\ell)}
=\max_{M_{\ell,b}^q[1/\ell]\ne0}e_b,
\]
with value \(-\infty\) for the empty support. This identifies \eqref{eq:stratified-modules} with
\eqref{eq:supporttest} and proves the lemma.
\end{proof}

\begin{lemma}\label{lem:ranktransfer}
Use the notation of Lemma~\ref{lem:betti}. Let \(\gamma:V\to S^{\an}\) be an analytic graph obtained
by \(r\) successive cuts on a qcqs neighborhood over \(K\). For corresponding objects \(E_\ell\) in
\eqref{eq:bettilist}, form their restricted analytifications on this neighborhood as in
Section~\ref{sec:formal}, and denote them by the same symbols. The upper perverse bounds of
\(\gamma^*E_\ell[-r]\) are independent of \(\ell\). This includes the strict upper bounds after
restriction to the corresponding closed boundaries.
\end{lemma}
\begin{proof}
Work on an affinoid chart \(V=\Spa(A_V)\) whose graph maps into an affine algebraic open of \(S\). A
classical point \(x\in V\) has residue field finite over \(K\). Its image \(x_0\in S\) is an
algebraic closed point: evaluation of a \(K\)-algebra of finite type in \(\kappa(x)\) has
finite-dimensional image, which is a field. Ordinary pullback preserves geometric stalks, so the
stalk cohomology of \(\gamma^*E_\ell[-r]\) at \(x\) is the algebraic stalk cohomology at \(x_0\),
shifted by \(-r\). Lemma~\ref{lem:betti} makes these dimensions independent of \(\ell\).

Compare two coefficient primes. Choose integral algebraic models for the finitely many corresponding
objects, pull them back to \(\Spec A_V\), and include the shift \([-r]\).
Lemma~\ref{lem:rational-support} gives a common finite regular stratification with perfect values.
Since \(A_V\) is Jacobson, every nonempty stratum contains closed points of \(\Spec A_V\),
corresponding to classical points of \(V\). Rational cohomology ranks are constant on each stratum.
Equality at classical points therefore shows that the same strata support nonzero cohomology in each
degree for both primes. The closures of their supports agree and therefore have the same dimension,
so \eqref{eq:supporttest} gives identical upper perverse bounds. The same argument applies to the
two boundary complexes in \eqref{eq:bettilist}, using closed base change, and hence also to their
strict upper bounds.
\end{proof}

\begin{theorem}\label{thm:smoothsource}
The conclusions of Theorems~\ref{thm:hl} and~\ref{thm:support} remain valid when \(\ell=p\). In
particular, every proper morphism with smooth source and a global relatively ample line bundle has
the decomposition in \eqref{eq:main} for every prime.
\end{theorem}
\begin{proof}
We work locally on the base as in Theorems~\ref{thm:hl} and~\ref{thm:support}.

\smallskip
\noindent\emph{Step 1. Fix one geometric approximation.}
Choose an auxiliary prime \(\ell_0\ne p\), and apply Proposition~\ref{prop:boundaryvanish} to the
projective morphism and a prescribed closed boundary. Keep this algebraic model, its finite diagram
resolving the inverse image of the boundary, and its analytic graph fixed for all coefficient
primes. Write \(g_0:W_0\to S\) and \(i_0:Z_0\hookrightarrow S\) for the algebraic morphism and
boundary, and use \(K_{0,\ell},P_{a,\ell},I_{a,\ell}\) as in Lemma~\ref{lem:betti}. Put
\(Q_{a,p}=\pH^0(i_0^!P_{a,p})\). Let \(M_0\subset S^{\an}\) be the chosen qcqs neighborhood and
\(\gamma:V\hookrightarrow M_0\) the graph immersion. We use the same notation for the restricted
analytifications. Let \(r\) be the codimension of \(\gamma\), put \(T^{(r)}=\gamma^*[-r]\), and
write \(i:Z\hookrightarrow V\) for the prescribed boundary.

\smallskip
\noindent\emph{Step 2. Transfer the upper perverse bounds.}
For \(\ell_0\ne p\), the proofs of Theorems~\ref{thm:hl} and~\ref{thm:support} show that
\(T^{(r)}P_{a,\ell_0}\) and \(T^{(r)}I_{a,\ell_0}\) are perverse. They also give the strict boundary
bounds for \(T^{(r)}I_{a,\ell_0}\) and its Verdier dual. Equation~\eqref{eq:dualgraph} identifies
the Verdier duals of these perverse objects, up to Tate twist, with the shifted restrictions of
their algebraic duals. Lemma~\ref{lem:ranktransfer} therefore transfers the upper bounds for all six
types of complexes in \eqref{eq:bettilist} to \(p\). In particular, \(T^{(r)}P_{a,p}\),
\(T^{(r)}I_{a,p}\), and the shifted restrictions of their algebraic duals belong to
\({}^pD^{\le0}\), while the two transported boundary
complexes belong to \({}^pD^{\le-1}\).

\smallskip
\noindent\emph{Step 3. Obtain the lower bounds by Gysin duality.}
Algebraic Lefschetz splitting and the algebraic decomposition by supports make \(P_{a,p}\),
\(I_{a,p}\), and their Verdier duals shifted and Tate-twisted derived retracts of \(K_{0,p}\).
Pullback preserves these retractions, and proper base change identifies the ambient object after
each cut with the normalized direct image from the new smooth source. Thus Lemma~\ref{lem:purity}
applies at every stage. For \(E=P_{a,p}\) or \(I_{a,p}\), equation~\eqref{eq:dualgraph} gives
\[
\mathbb D(T^{(r)}E)\simeq T^{(r)}(\mathbb DE)(-r).
\]
The right-hand side belongs to \({}^pD^{\le0}\) by Step 2. Verdier duality with rational
coefficients gives \(T^{(r)}E\in{}^pD^{\ge0}\), so both \(T^{(r)}P_{a,p}\) and \(T^{(r)}I_{a,p}\)
are perverse.

\smallskip
\noindent\emph{Step 4. Preserve the strict boundary conditions.}
Let \(\gamma_Z:Z\to Z_0\) be the map induced by \(\gamma\), with \(Z_0\) restricted to \(M_0\). For
each coefficient prime, closed base change gives the natural identities \(\gamma^*i_{0*}\simeq
i_*\gamma_Z^*\) and \(\gamma_Z^*i_0^*=i^*\gamma^*\). These hold for integral complexes and commute
with rationalization \cite[Theorems 3.15 and 3.36]{BH}. Applying them to \(I_{a,p}\) and \(\mathbb
DI_{a,p}\), the transferred bounds give
\[
\begin{aligned}
i_*i^*T^{(r)}I_{a,p}
&\simeq T^{(r)}(i_{0*}i_0^*I_{a,p})\in{}^pD^{\le-1},\\
i_*i^*\mathbb D(T^{(r)}I_{a,p})
&\simeq T^{(r)}(i_{0*}i_0^*\mathbb DI_{a,p})(-r)
  \in{}^pD^{\le-1},
\end{aligned}
\]
where the second identification uses \eqref{eq:dualgraph}. Hence \(T^{(r)}I_{a,p}\) is the
intermediate extension of its open restriction. The algebraic decomposition \(P_{a,p}\simeq
I_{a,p}\oplus i_{0*}Q_{a,p}\) remains a direct-sum decomposition after restriction to the graph. Its
complementary summand is perverse, as a derived summand of the perverse object \(T^{(r)}P_{a,p}\),
and is supported on the final boundary. Lemma~\ref{lem:boundary} therefore makes the canonical
boundary map for \(T^{(r)}P_{a,p}\) an isomorphism.

\smallskip
\noindent\emph{Step 5. Recover hard Lefschetz and the global decomposition.}
Algebraic relative hard Lefschetz applies to \(K_{0,p}\), since the algebraic model has
characteristic zero. The shifted restrictions \(T^{(r)}P_{a,p}\) of its perverse summands are now
known to be perverse. Applying \(T^{(r)}\) to the algebraic splitting identifies them with the
perverse cohomology of the rigid direct image, as in \eqref{eq:transport}. The matrix argument of
Theorem~\ref{thm:hl}
identifies the transported Lefschetz maps with those induced by the original relatively ample line
bundle, up to the invertible scalar arising from a tensor power. These canonical maps are
isomorphisms locally over \(K\), hence globally. Lemma~\ref{lem:deligne} gives the global derived
splitting. The boundary argument applies to every prescribed closed subset, and
Corollary~\ref{cor:support-decomposition} decomposes each perverse cohomology sheaf by supports.
\end{proof}

\subsection{Singular sources and derived retracts}\label{sec:singular}

We resolve the source and realize its intersection complex as a derived direct summand of the direct
image of the constant sheaf. The next lemma shows that this summand inherits the decomposition.

\begin{lemma}\label{lem:retract}
Let \(M\) be a qcqs rigid space, and suppose \(E\simeq\bigoplus_a P_a[-a]\) is a finite direct sum
in \(\Db(M,\Q_\ell)\), with each \(P_a\) a finite direct sum of intermediate extensions of
finite-rank local systems on smooth Zariski-locally closed strata. Every derived direct summand of
\(E\) has a decomposition of the same form.
\end{lemma}
\begin{proof}
For \(a<b\), the perverse t-structure gives
\[
\operatorname{Hom}(P_a[-a],P_b[-b])
=\operatorname{Hom}(P_a,P_b[a-b])=0.
\]
Thus \(\operatorname{End}(E)\) is triangular, and its strictly triangular ideal \(J\) is nilpotent
because only finitely many degrees occur. Let \(e\in\operatorname{End}(E)\) satisfy \(e^2=e\) and
define the chosen derived retract. Write \(e_{aa}:P_a\to P_a\) for its diagonal components and put
\(d=\bigoplus_a e_{aa}[-a]\). Triangularity gives \(e_{aa}^2=e_{aa}\), so \(d^2=d\). Set
\(u=ed+(1-e)(1-d)\). Then \(u\equiv1\pmod J\), hence \(u\) is invertible, and \(eu=ud\). It follows
that \(e=udu^{-1}\), so its image is isomorphic to \(\bigoplus_a\operatorname{im}(e_{aa})[-a]\).
These diagonal images exist in the abelian perverse heart.

To describe each diagonal image, group the summands of \(P_a\) given by intermediate extension
according to their irreducible supports. There are no nonzero morphisms between summands with
distinct strict supports: the image would be both a quotient of one and a subobject of the other,
supported on a proper closed subset of at least one support. Thus \(e_{aa}\) preserves each group.
On a fixed support \(Z\), choose a common smooth dense open \(U\) on which all the summands are
shifted local systems. Full faithfulness of intermediate extension identifies \(e_{aa}\) on this
group with an idempotent of the direct sum of these local systems on \(U\). Its image is a
finite-rank local system, and intermediate extension preserves the resulting direct-sum
decomposition. Hence each \(\operatorname{im}(e_{aa})\) is a finite direct sum of intermediate
extensions, proving the lemma.
\end{proof}

\begin{proof}[Proof of Theorem~\ref{thm:general}]
Let \(f:X\to Y\) be projective, with \(Y\) qcqs. Then \(X\) is qcqs. Nilpotent thickenings do not
change the \'etale objects under consideration, and the componentwise definition of \(\IC_X\) allows
us to treat the finitely many reduced irreducible components separately, shifting by the dimension
of each component and applying its closed pushforward. We may therefore assume that \(X\) is reduced
and pure-dimensional of dimension \(d\).

\smallskip
\noindent\emph{Step 1. Resolve the source.}
Temkin's analytic desingularization theorem \cite[Section 5.2.1 and Theorem 5.2.2]{Temkin} gives a
resolution \(\pi:\widetilde X\to X\) by analytic blow-ups, with \(\widetilde X\) smooth of pure
dimension \(d\) and \(\pi\) an isomorphism over a smooth dense open \(j:U\hookrightarrow X\).
Quasi-compactness makes the sequence finite. Each blow-up is projective and carries a global
relatively ample tautological line bundle. For proper maps \(h:V\to W\) and \(g:W\to S\) with
relatively ample line bundles \(A\) and \(B\), respectively, \(A\otimes h^*B^{\otimes n}\) is
\(g\circ h\)-ample for sufficiently large \(n\) when \(S\) is qcqs. Locally on \(S\), relative GAGA
and the relative Proj construction \cite[Example 3.2.6 and Theorem 3.2.7]{Conrad} reduce this to the
algebraic composition theorem \cite[Tag 0C4K]{Stacks}; a finite affinoid covering gives one exponent
on \(S\). Applying this successively gives global relatively ample line bundles for \(\pi\) and
\(f\circ\pi\). Theorem~\ref{thm:smoothsource} therefore applies to both maps.

\smallskip
\noindent\emph{Step 2. Realize the intersection complex as a retract.}
Put \(Z=X\setminus U\), \(A=R\pi_*\Q_\ell[d]\), and \(P=\pH^0(A)\). Since \(\pi\) is an isomorphism
over \(U\), we have \(P|_U\simeq\Q_\ell[d]\). The boundary splitting supplied by
Theorem~\ref{thm:smoothsource} gives \(P\simeq j_{!*}(\Q_\ell[d])\oplus P_Z\), where \(P_Z\) is
supported on \(Z\). By definition, the first summand is \(\IC_X\). The global derived
splitting for \(\pi\) makes \(P\) a derived summand of \(A\). Composing the perverse and derived
retractions gives
\[
\IC_X\longrightarrow R\pi_*\Q_\ell[d]\longrightarrow\IC_X
\]
with composite \(\operatorname{id}_{\IC_X}\).

\smallskip
\noindent\emph{Step 3. Apply the projective direct image.}
Applying \(Rf_*\) shows that \(Rf_*\IC_X\) is a derived retract of \(R(f\circ\pi)_*\Q_\ell[d]\). The
source \(\widetilde X\) is smooth, so Theorem~\ref{thm:smoothsource} decomposes this latter object
into finitely many shifts of intermediate extensions of local systems. Lemma~\ref{lem:retract} gives
a decomposition of the same form for \(Rf_*\IC_X\), proving \eqref{eq:main}.
\end{proof}

\begin{corollary}\label{cor:perverse-leray}
Let \(f:X\to Y\) be as in Theorem~\ref{thm:general}, and put \(C=\widehat{\overline K}\). For every
prime \(\ell\), including \(p\), the perverse Leray spectral sequence
\[
E_2^{a,b}
=H^a\bigl(Y_C,(\pH^b(Rf_*\IC_X))_C\bigr)
\Longrightarrow H^{a+b}\bigl(X_C,(\IC_X)_C\bigr)
\]
degenerates at \(E_2\).
\end{corollary}
\begin{proof}
Base change the derived decomposition in \eqref{eq:main} to \(C\) and apply \(R\Gamma(Y_C,-)\). The
perverse truncation filtration is split, so every differential from \(E_2\) onwards vanishes.
\end{proof}

\section{De Rham intersection cohomology}\label{sec:arithmetic}

Theorem~\ref{thm:boundary-derham} and its consequences through Corollary~\ref{cor:rank-one-derham}
use comparison theorems, normalization, and finite covers, independently of projective
decomposition. In Section~\ref{sec:decomposition-coefficients}, we use the retraction obtained from
a resolution and the decomposition for smooth sources, both from Section~\ref{sec:decomposition}, to
prove \cite[Conjecture 4.16(1)]{BH} for proper spaces. The same arguments give de Rham intersection
cohomology for local systems that extend across the boundary or admit a projective realization.

Throughout this section, \(K/\Q_p\) is finite,
\(C=\widehat{\overline K}\), and \(G_K=\operatorname{Gal}(\overline K/K)\).
Let \(X/K\) be proper, reduced, and pure-dimensional of dimension
\(d\). Let \(j:U\hookrightarrow X\) be the inclusion of a smooth dense
Zariski-open subspace, and let \(\mathbb L\) be a de Rham
\(\Z_p\)-local system on \(U\).
Write
\[
 L=\mathbb L\otimes_{\Z_p}\Q_p,\qquad
 P(L)=\IC_X(L)=j_{!*}(L[d]),\qquad
 \IH^m(X_C,L)=H^m(X_C,P(L)_C).
\]
This is the perverse normalization of \cite[Construction 4.12]{BH}.
In particular, when \(X\) is smooth and \(L\) is defined on all of \(X\),
\(\IH^m(X_C,L)=H^{m+d}(X_C,L)\).
The de Rham assertions concern the rational Galois representations
associated with the given integral lattices.
For a nonproper smooth open we take integral cohomology before
inverting \(p\). More explicitly, with
\(\mathbb L_n=\mathbb L/p^n\), our conventions are
\begin{align*}
 R\Gamma(U_C,L)
   &=\bigl(R\varprojlim_n R\Gamma(U_C,\mathbb L_n)\bigr)[1/p],\\
 R\Gamma_c(U_C,L)
   &=\bigl(R\varprojlim_n R\Gamma(X_C,j_{C,!}\mathbb L_n)\bigr)[1/p].
\end{align*}
These are the conventions used by the logarithmic comparison
theorems; see also \cite[Lemma 2.4.3 and Remark 2.4.4]{LLZ}.
Direct images along open immersions are likewise formed from the
integral system on the fixed proper ambient space and then rationalized.

\subsection{The boundary range}

The dimension of the boundary bounds the degrees in which intersection cohomology can differ from
ordinary or compactly supported cohomology of the smooth open subset.

\begin{lemma}\label{lem:open-derham}
For every integer \(q\), the representations \(H^q(U_C,L)\) and \(H_c^q(U_C,L)\) are
finite-dimensional de Rham representations of \(G_K\).
\end{lemma}
\begin{proof}
Apply Temkin's desingularization and weak principalization theorems to the pair \((X,X\setminus U)\)
\cite[Theorems 1.1.11 and 1.1.13(i)]{TemkinEmbedded}. Since \(U\) is smooth, the resulting proper
modification \(\pi:\widetilde X\to X\) is an isomorphism over \(U\), with \(\widetilde X\) smooth
and \(D=(\widetilde X\setminus U)_{\mathrm{red}}\) a strict normal crossings divisor. The space
\(\widetilde X\) is proper because \(X\) is proper. Equip \(\widetilde X\) with the log structure
defined by \(D\). By \cite[Corollary 6.3.4]{DLLZFound}, the local system \(\mathbb L\) extends to a
Kummer \'etale \(\Z_p\)-local system on this log space. The logarithmic de Rham comparison theorem
\cite[Corollary 3.2.10]{DLLZLog} then gives a finite-dimensional, filtered, \(G_K\)-equivariant
comparison over Fontaine's de Rham period field \(B_{\mathrm{dR}}\) for \(H^q(U_C,L)\), so this
representation is de Rham.

For compactly supported cohomology, apply \cite[Theorem 3.1.10]{LLZ} with compact supports imposed
along every component of \(D\). This gives the corresponding finite-dimensional de Rham comparison.
Writing \(\widetilde j:U\hookrightarrow\widetilde X\) and \(\mathbb L_n=\mathbb L/p^n\), proper base
change gives \(R\pi_*\widetilde j_!\mathbb L_n\simeq j_!\mathbb L_n\): over \(U\) the map \(\pi\) is
an isomorphism, and over \(X\setminus U\) both sides have zero stalks. These identifications are
compatible with reduction in \(n\). Taking derived inverse limits and then inverting \(p\) gives the
cohomology groups with compact support defined above.
\end{proof}

\begin{lemma}\label{lem:boundary-upper-bound}
Let \(S/K\) be proper of dimension at most \(b\), and let \(F\in{}^pD^{\le-1}(S,\Q_p)\). Then
\(R\Gamma(S_C,F_C)\in D^{\le b-1}(\Q_p)\).
\end{lemma}
\begin{proof}
Choose an integral model \(A_0\) of \(F\), and put \(A={}^p\tau_{\le-1}A_0\). Rationalization
commutes with perverse truncation, so \(A[1/p]\simeq F\). For every \(n\ge1\), the reduction
\(A_n=A\otimes_{\Z_p}^{\mathbf L}\Z/p^n\) lies in \({}^pD^{\le-1}(S,\Z/p^n)\) by \cite[Proposition
4.6(2) and Theorem 4.11]{BH}. For each nonzero \(\mathcal H^q(A_n)_C\), let \(Z_q\) be the Zariski
closure of its support with the reduced structure, let \(a_q:Z_q\hookrightarrow S_C\) be the
inclusion, and put \(s_q=\dim Z_q\). The perverse support condition gives \(q+s_q\le-1\), and
\(s_q\le b\). Since \(\mathcal H^q(A_n)_C\simeq a_{q*}a_q^*\mathcal H^q(A_n)_C\) and closed
pushforward is exact, its cohomology can be computed on \(Z_q\). The bound on \'etale cohomological
dimension for torsion sheaves on this qcqs space over \(C\), including for \(p\)-torsion, gives
\(H^r(S_C,\mathcal H^q(A_n)_C)=0\) for \(r>2s_q\) \cite[Corollary 4.2.6]{BerkovichEtale}; see also
\cite[Introduction, p.~2]{HansenArtin} for its formulation. Thus every nonzero
term in the ordinary hypercohomology spectral sequence has total degree bounded by
\[
 q+r\le q+2s_q\le s_q-1\le b-1.
\]
It follows that \(R\Gamma(S_C,(A_n)_C)\in D^{\le b-1}\) for every \(n\). Proper finiteness
\cite[Theorem 3.10]{BH} makes the cohomology groups for each \(n\) finite, so the inverse systems in
each degree satisfy the Mittag--Leffler condition. In particular, \(R^1\!\varprojlim_n
H^{b-1}(S_C,(A_n)_C)=0\), so the derived inverse limit has no cohomology above degree \(b-1\).
Inverting \(p\) proves the assertion.
\end{proof}

\begin{theorem}\label{thm:boundary-derham}
Suppose \(Z=X\setminus U\) is nonempty, and put \(b=\dim Z\). Then \(\IH^m(X_C,L)\) is de Rham
whenever
\begin{equation}\label{eq:derham-range}
 m\le-b\qquad\text{or}\qquad m\ge b.
\end{equation}
More precisely, the canonical restriction map
\begin{equation}\label{eq:ih-open-injection}
 \IH^m(X_C,L)\longrightarrow H^{m+d}(U_C,L)
\end{equation}
is an isomorphism for \(m<-b\) and injective for \(m=-b\). The canonical map from compactly
supported cohomology
\begin{equation}\label{eq:compact-ih-surjection}
 H_c^{m+d}(U_C,L)\longrightarrow\IH^m(X_C,L)
\end{equation}
is an isomorphism for \(m>b\) and surjective for \(m=b\). If \(Z\) is empty, all intersection
cohomology groups are de Rham.
\end{theorem}
\begin{proof}
Let \(i:Z\hookrightarrow X\), and write \(P=P(L)\). The strict cosupport condition for intermediate
extension gives \(i^!P\in{}^pD^{\ge1}(Z,\Q_p)\). On a space of dimension at most \(b\), the support
definition of the perverse \(t\)-structure gives \(D^{\le-b}\subset{}^pD^{\le0}\): every possibly
nonzero ordinary cohomology sheaf on the left has degree \(q\le-b\) and support of dimension at most
\(b\le-q\). Taking right orthogonals yields
\begin{equation}\label{eq:perverse-ordinary-lower}
 {}^pD^{\ge1}\subset D^{\ge1-b}.
\end{equation}
Here we use the rational perverse structure of \cite[Definition 4.1, Proposition 4.6, and Theorem
4.11]{BH}. After base change to \(C\), left \(t\)-exactness of derived global sections for the
ordinary \(t\)-structure gives \(R\Gamma(Z_C,i_C^!P_C)\in D^{\ge1-b}(\Q_p)\). Consider the
localization triangle
\begin{equation}\label{eq:arithmetic-localization}
 i_*i^!P\longrightarrow P\longrightarrow Rj_*L[d]
 \longrightarrow i_*i^!P[1].
\end{equation}
Its long exact sequence gives the isomorphism and injection in \eqref{eq:ih-open-injection}. We form
\(Rj_*L[d]\) from the given integral lattice on the proper ambient space \(X\), using \cite[Theorem
3.36(4)]{BH}, and then invert \(p\); its geometric cohomology is therefore the cohomology of \(U_C\)
defined above. All maps are defined over \(K\), so the resulting maps of representations are
\(G_K\)-equivariant. Lemma~\ref{lem:open-derham} and stability of de Rham representations under
subrepresentations prove the assertion for \(m\le-b\).

For the upper range, the strict support condition gives \(i^*P\in{}^pD^{\le-1}(Z,\Q_p)\), and
Lemma~\ref{lem:boundary-upper-bound} gives \(R\Gamma(Z_C,i_C^*P_C)\in D^{\le b-1}(\Q_p)\). The
localization triangle
\[
 j_!L[d]\longrightarrow P\longrightarrow i_*i^*P
 \longrightarrow j_!L[d+1]
\]
then gives the isomorphism and surjection in \eqref{eq:compact-ih-surjection}. The representation
\(H_c^{m+d}(U_C,L)\) is de Rham by Lemma~\ref{lem:open-derham}, and de Rham representations are
closed under quotients. This proves
the assertion for \(m\ge b\). If \(Z\) is empty, \(X=U\) is smooth and proper, so smooth proper de
Rham comparison \cite[Theorem 1.6]{ScholzeHT} applies.
\end{proof}

\begin{remark}
Smooth duality \cite[Corollary 7.1.2]{LRZ}, compatibility of intermediate extension with Verdier
duality, and Lemma~\ref{lem:adic-proper-duality} give
\begin{align}
 \mathbb D_XP(L)&\simeq P(L^\vee)(d),\label{eq:ic-coefficient-duality}\\
 \IH^m(X_C,L)&\simeq
 \IH^{-m}(X_C,L^\vee)^\vee(-d).\label{eq:ih-coefficient-duality}
\end{align}
Thus duality exchanges the lower and upper ranges in Theorem~\ref{thm:boundary-derham}.
\end{remark}

\begin{corollary}\label{cor:finite-boundary-derham}
If \(X\setminus U\) is finite, every \(\IH^m(X_C,L)\) is de Rham. In particular, this holds for a
proper reduced curve with de Rham local coefficients on a smooth dense open subset, and for constant
coefficients on a proper pure-dimensional rigid space with finite singular locus.
\end{corollary}
\begin{proof}
For nonempty finite boundary, \(b=0\), and the two ranges in \eqref{eq:derham-range} cover every
integer. The case of an empty boundary is included in Theorem~\ref{thm:boundary-derham}. The
complement of a dense Zariski open in a proper curve is finite; for constant coefficients, take
\(U=X^{\mathrm{sm}}\).
\end{proof}

\begin{proposition}\label{prop:clean-boundary-derham}
Let \(i:Z=X\setminus U\hookrightarrow X\), and let \(T\subset Z\) be closed. Suppose that the
canonical morphism \(j_!L\to Rj_*L\), formed on \(X\) using the given lattice, is an isomorphism on
\(X\setminus T\). If \(T\) is nonempty of dimension \(c\), then \(\IH^m(X_C,L)\) is de Rham for
\(m\le-c\) or \(m\ge c\). If \(T\) is finite or empty, it is de Rham in every degree.
\end{proposition}
\begin{proof}
On \(X\setminus T\), the complexes \(j_!L[d]\) and \(Rj_*L[d]\) are isomorphic. The first has the
upper perverse bound and the second the lower perverse bound, so their common restriction is
perverse. Its \(*\)-restriction to the boundary vanishes because it is an extension by zero, and its
\(!\)-restriction to the boundary vanishes because it is a direct image under the open immersion. It
is therefore the intermediate extension there, and both \(i^*P(L)\) and \(i^!P(L)\) are supported on
\(T\). These assertions may be checked on qcqs affinoid neighborhoods in \(X\setminus T\), with the
complexes formed on \(X\).

Closed pushforward is perverse \(t\)-exact and reflects perverse bounds for objects supported on its
image. We may therefore apply the argument of Theorem~\ref{thm:boundary-derham} to the boundary
complexes on \(T\), using \eqref{eq:perverse-ordinary-lower} and
Lemma~\ref{lem:boundary-upper-bound} with \(c\) in place of \(b\). If \(T\) is finite, the resulting
ranges cover every degree. If \(T\) is empty, \(P(L)=Rj_*L[d]\), and Lemma~\ref{lem:open-derham}
gives the assertion in every degree.
\end{proof}

Under the hypotheses of Proposition~\ref{prop:clean-boundary-derham}, assume that \(T\) is finite.
The two localization maps then identify the middle intersection cohomology with
\begin{equation}\label{eq:clean-middle-image}
 \IH^0(X_C,L)\simeq
 \operatorname{im}\bigl(H_c^d(U_C,L)\longrightarrow H^d(U_C,L)\bigr).
\end{equation}
Indeed, the boundary complexes are supported at finitely many points, where their strict perverse
bounds become ordinary cohomological bounds. Geometric global sections on these supports are exact,
so the localization maps give a surjection from \(H_c^d(U_C,L)\) onto \(\IH^0(X_C,L)\) and an
injection from \(\IH^0(X_C,L)\) into \(H^d(U_C,L)\). Their composite is the usual map from compactly
supported to ordinary cohomology.

\subsection{Normalization and finite covers}

Normalization and finite covers allow us to apply Theorem~\ref{thm:boundary-derham} with a smaller
boundary. We use the compatibility of finite pushforward with the perverse \(t\)-structure and
intermediate extension.

\begin{proposition}\label{prop:normalized-coeff-derham}
Let \(\nu:X^\nu\to X\) be the disjoint union of the normalizations of the reduced irreducible
components, and put \(W=(X^\nu)^{\mathrm{sm}}\). Suppose that the pullback of \(\mathbb L\) from
\(U\simeq\nu^{-1}U\) extends to a \(\Z_p\)-local system \(\mathbb M\) on \(W\). If \(X^\nu\setminus
W\) is nonempty of dimension \(b_\nu\), then \(\IH^m(X_C,L)\) is de Rham for \(m\le-b_\nu\) or
\(m\ge b_\nu\). If \(X^\nu\) is smooth, \(\IH^m(X_C,L)\) is de Rham for every \(m\). Under
the stated extension hypothesis, it is also de Rham in every degree when \(\dim X\le2\).
\end{proposition}
\begin{proof}
Normalization is finite because affinoid algebras are excellent. The local system \(M=\mathbb
M[1/p]\) agrees with the pullback of \(L\) on the dense open \(\nu^{-1}U\). Every connected smooth
affinoid open in \(W\) meets this subset, so the rigidity theorem of Liu--Zhu \cite[Theorem 3.9(iv)]{LiuZhu}
implies that \(M\) is de Rham on \(W\). Let \(a:W\hookrightarrow X^\nu\) be the inclusion.
Intermediate extension is unchanged when the smooth dense open is shrunk, and finite pushforward
is perverse \(t\)-exact and
preserves the strict boundary conditions. Consequently,
\[
 P(L)\simeq\nu_*a_{!*}(M[d]).
\]
The de Rham property of \(M\) is checked on smooth affinoid opens of \(W\), and the two
intermediate extensions are formed on the proper qcqs spaces \(X\) and \(X^\nu\).

Taking geometric cohomology of the preceding isomorphism gives
\(\IH^m(X_C,L)\simeq\IH^m(X^\nu_C,M)\), since derived global sections commute with finite
pushforward. Applying
Theorem~\ref{thm:boundary-derham} to \((X^\nu,W,\mathbb M)\) gives the stated ranges, including
every degree when \(X^\nu\) is smooth. Since \(K\) has characteristic zero, a normal rigid space
over \(K\) is smooth in codimension one. Thus normal curves and reduced zero-dimensional spaces are
smooth, while the singular locus of a normal surface has dimension at most zero. This proves the
conclusion in dimensions at most two.
\end{proof}

\begin{corollary}\label{cor:low-dimension-derham}
If \(X\) has pure dimension \(d\ge2\), then \(\IH^m(X_C,\Q_p)\) is de Rham for \(m\le2-d\) or \(m\ge
d-2\). In particular, every intersection cohomology group of a proper rigid space of dimension at
most two is de Rham.
\end{corollary}
\begin{proof}
Apply Proposition~\ref{prop:normalized-coeff-derham} to the constant local system on
\((X^\nu)^{\mathrm{sm}}\). Its complement has dimension at most \(d-2\), giving the stated ranges,
which cover every integer for \(d=2\). The same proposition handles the lower dimensions. Finite
sums and invariance under nilpotent thickenings give the last assertion.
\end{proof}

\begin{proposition}\label{prop:finite-geometric-derham}
Suppose that \(\mathbb L\) has finite geometric monodromy on each geometrically connected component
of \(U\). If \(d\ge2\), then \(\IH^m(X_C,L)\) is de Rham for \(m\le2-d\) or \(m\ge d-2\). If
\(d\le2\), it is de Rham in every degree.
\end{proposition}
\begin{proof}
Compatibility of intermediate extension with finite pushforward reduces us to normal \(X\).
We may also work componentwise after a finite extension of
\(K\), since intermediate extension commutes with this base change and the de Rham property is
detected after a finite field extension. Put \(r=\operatorname{rank}\mathbb L\). Choose a geometric
base point, let \(\rho:\pi_1(U)\to\operatorname{GL}_r(\Z_p)\) be the monodromy representation, and
put \(H=\rho(\pi_1(U_C))\), which is finite by hypothesis. Choose \(n\) such that reduction modulo
\(p^n\) is injective on \(H\). The finite \'etale frame cover \(g:U'\to U\) associated with the
kernel of \(\rho\bmod p^n\) kills the geometric monodromy: a geometric element acting trivially
modulo \(p^n\) already acts trivially on \(\mathbb L\). Hansen's extension theorem \cite[Theorem
1.6]{HansenCover} extends \(g\) to a finite cover \(f:Y\to X\) with \(Y\) normal. The space \(Y\) is
proper and pure of dimension \(d\).

After a further finite extension of \(K\), the components \(Y_\alpha\) and their dense smooth opens
\(U'_\alpha\) are geometrically connected. The pullback \(L'_\alpha=g^*L|_{U'_\alpha}\) is
geometrically constant, so evaluation identifies it with the arithmetic constant local system
associated with \(V_\alpha=H^0(U'_{\alpha,C},L'_\alpha)\). This identification is
\(G_K\)-equivariant. Since \(L'_\alpha\) is de Rham, Lemma~\ref{lem:open-derham}, applied to
\(U'_\alpha\subset Y_\alpha\), shows that \(V_\alpha\) is de Rham.

The unit and the trace divided by the degree exhibit \(L\) as a rational direct summand of
\(g_*g^*L\). Intermediate extension preserves this split pair, finite pushforward commutes with
intermediate extension, and tensoring by \(V_\alpha\) also commutes with intermediate extension.
Hence \(P(L)\) is a direct summand of
\(f_*\bigoplus_\alpha(\IC_{Y_\alpha,\Q_p}\otimes_{\Q_p}V_\alpha)\). Taking geometric cohomology
realizes \(\IH^m(X_C,L)\) as a direct summand of
\[
 \bigoplus_\alpha
 \IH^m(Y_{\alpha,C},\Q_p)\otimes_{\Q_p}V_\alpha.
\]
Corollary~\ref{cor:low-dimension-derham} gives the required ranges for each factor with constant
coefficients. Stability of de Rham representations under tensor products, finite direct sums, and
direct summands, followed by descent along the finite field extension, proves the result.
\end{proof}

\begin{definition}\label{def:finite-divisorial}
Identify \(U\) with its inverse image in the componentwise normalization \(X^\nu\). Choose a closed
analytic subset \(T\subset X^\nu\setminus U\) of codimension at least two such that
\(X^\circ=X^\nu\setminus T\) is smooth and \(D=X^\circ\setminus U\) is a disjoint union of smooth
divisors. Such a choice is obtained by removing the singular locus of \(X^\nu\), the singular loci
and pairwise intersections of its divisorial boundary components, and the boundary components of
higher codimension. Equip \(X^\circ\) with the log structure defined by \(D\). For a geometric point
\(\xi\) of \(D\) and a log-geometric point \(\widetilde\xi\) above it, let \(X^\circ(\xi)\) be the
strict localization at \(\xi\), with the log structure pulled back from \(X^\circ\). Writing
\(\pi_1^{\mathrm{k\acute et}}\) for the Kummer \'etale fundamental group, the geometric Kummer
inertia group is
\[
 I_\xi=\pi_1^{\mathrm{k\acute et}}
       (X^\circ(\xi),\widetilde\xi)\simeq\widehat\Z(1)
\]
\cite[Proposition 4.4.9, Definition 6.3.7, and Example 6.3.8]{DLLZFound}. We say that \(\mathbb L\)
has finite divisorial inertia if, for one such choice of \(T\), every \(I_\xi\) has finite image on
the corresponding stalk of its Kummer \'etale extension.
\end{definition}

For a smooth logarithmic pair \((Y,D)\), write \(Y_{\mathrm{k\acute et}}\) and \(Y_{\et}\) for its
Kummer \'etale and ordinary \'etale sites. We use the following local descent criterion: a finite
locally constant sheaf on \(Y_{\mathrm{k\acute et}}\) comes from \(Y_{\et}\) if and only if every
geometric Kummer inertia group acts trivially. Indeed, represent it by a finite Kummer \'etale
cover. Over a strict localization, trivial inertia identifies this cover with a disjoint union of
copies of the base \cite[Proposition 4.4.9 and Corollary 4.4.22]{DLLZFound}, so its underlying
finite morphism is \'etale. Full faithfulness \cite[Remark 4.4.23]{DLLZFound} gives uniqueness and
preserves morphisms, module structures, and transition maps.

\begin{proposition}\label{prop:finite-divisorial-derham}
If \(\mathbb L\) has finite divisorial inertia, then \(\IH^m(X_C,L)\) is de Rham for \(m\le2-d\) or
\(m\ge d-2\) when \(d\ge2\), and in every degree when \(d\le2\).
\end{proposition}
\begin{proof}
Finite normalization reduces us to normal \(X\). Take the finite \'etale frame cover \(g:U'\to U\)
that trivializes \(\mathbb L/p^2\), and extend it to a finite normal cover \(f:Y\to X\) by
\cite[Theorem 1.6]{HansenCover}. On each connected component of \(U'\), the monodromy of \(\mathbb
L'=g^*\mathbb L\) takes values in \(1+p^2M_r(\Z_p)\), where \(r\) is the rank of \(\mathbb L'\) on
that component. This group is torsion-free for every prime \(p\), including \(p=2\): the matrix
logarithm converges and is injective there, and an element of finite order has zero logarithm.

Choose \(X^\circ,D,T\) as in Definition~\ref{def:finite-divisorial}. Rigid Abhyankar's lemma
\cite[Proposition 4.2.1]{DLLZFound} identifies \(Y^\circ=f^{-1}(X^\circ)\) as a finite Kummer
\'etale cover of \((X^\circ,D)\). Since \(D\) is a disjoint union of smooth divisors, its log rank
is at most one. Up to \'etale maps, the local Kummer charts have the form \(t=s^e\), so \(Y^\circ\)
is smooth and its boundary is again a disjoint union of smooth divisors. The geometric inertia
groups on \(Y^\circ\) map into those on \(X^\circ\). Their images under \(\mathbb L'\) are finite by
hypothesis and lie in the torsion-free congruence subgroup, so they are trivial. The local descent
criterion therefore extends each \(\mathbb L'/p^n\) uniquely as an ordinary \'etale local system on
\(Y^\circ\), compatibly with its module structure and the transition maps.

Put \(W=Y^{\mathrm{sm}}\). The closed subset \(W\setminus Y^\circ\) is contained in \(f^{-1}T\) and
has codimension at least two. Rigid analytic purity \cite[Corollary 2.15]{HansenCover} extends every
\(\mathbb L'/p^n\) uniquely to \(W\). Full faithfulness extends the module structures and transition
maps, so these extensions for varying \(n\) form a \(\Z_p\)-local system \(\mathbb M\) on \(W\).
Liu--Zhu rigidity \cite[Theorem 3.9(iv)]{LiuZhu} makes \(\mathbb M[1/p]\) de Rham on \(W\), since it
is de Rham on the dense open \(U'\). Apply Proposition~\ref{prop:normalized-coeff-derham} to each
normal proper component of \(Y\). For \(d\ge2\), its singular locus has dimension at most \(d-2\);
in dimensions at most two, the proposition gives every degree. Finally, the unit and normalized
trace over \(U\), followed by intermediate extension and finite pushforward, make \(P(L)\) a direct
summand of \(f_*\IC_Y(g^*L)\). Taking geometric cohomology gives the asserted ranges.
\end{proof}

\begin{corollary}\label{cor:rank-one-derham}
Every rank-one de Rham \(\Z_p\)-local system on \(U\) satisfies the conclusion of
Proposition~\ref{prop:finite-divisorial-derham}. In particular, all its intersection cohomology
groups are de Rham when \(\dim X\le2\), including when the local system is ramified along a
divisorial boundary.
\end{corollary}
\begin{proof}
We show that \(\mathbb L\) has finite divisorial inertia. Choose \(X^\circ,D,T\) as in
Definition~\ref{def:finite-divisorial}, and let \(x\) be a classical point of \(D\) with residue
field \(K'/K\) finite. After passing to an \'etale neighborhood over \(K'\), choose a parameter
defining \(D\). The resulting split log point has Kummer fundamental group \(I_x\rtimes G_{K'}\),
where \(I_x\simeq\widehat\Z(1)\) and \(G_{K'}\) acts cyclotomically; see \cite[Examples
4.4.24--4.4.25]{DLLZFound}. The rank-one local system gives a character \(\rho:I_x\rtimes
G_{K'}\to\Z_p^\times\). Let \(\mu(\Z_p)\) be the group of roots of unity in \(\Z_p^\times\), and put
\(w=|\mu(\Z_p)|\). The character \(\rho^w\) takes values in \(1+p\Z_p\), and in \(1+4\Z_2\) when
\(p=2\), where the \(p\)-adic logarithm is injective. Let \(\chi_p:G_{K'}\to\Z_p^\times\) be the
cyclotomic character. Since the target of \(\rho\) is abelian, the relation in the semidirect
product gives \((\chi_p(\sigma)-1)\log(\rho(t)^w)=0\) for every \(\sigma\in G_{K'}\) and every \(t\)
in the pro-\(p\) subgroup of \(I_x\). The cyclotomic character has infinite image, so
\(\rho(t)^w=1\). The prime-to-\(p\) inertia has finite image in \(\Z_p^\times\), hence also has
image in \(\mu(\Z_p)\). Thus \(\rho(I_x)\subseteq\mu(\Z_p)\), and the inertia action on \(\mathbb
L^{\otimes w}\) is trivial at every classical point of the boundary.

For each \(n\), consider the finite \'etale frame cover of \(\mathbb L^{\otimes w}/p^n\) on \(U\).
Rigid Abhyankar's lemma \cite[Proposition 4.2.1]{DLLZFound} extends it to a finite Kummer \'etale
cover of \((X^\circ,D)\). The inertia is trivial at every classical point of the boundary, so the
local descent criterion makes the extension ordinary \'etale there. A nonempty non-\'etale locus is
a closed analytic subset and contains a classical point. Since the cover is already \'etale on
\(U\), this locus is empty. The extension is therefore finite \'etale everywhere on \(X^\circ\), and
\(\mathbb L^{\otimes w}/p^n\) has trivial geometric Kummer inertia at every log-geometric point of
the boundary. This holds for all \(n\), so \(\mathbb L^{\otimes w}\) has trivial geometric Kummer
inertia. The image of inertia acting on \(\mathbb L\) consequently has exponent dividing \(w\) and
is contained in the finite group \(\mu(\Z_p)\). Proposition~\ref{prop:finite-divisorial-derham} now
applies.
\end{proof}

\subsection{Consequences of projective decomposition}
\label{sec:decomposition-coefficients}

With coefficients in \(\Q_p\), the decomposition theorem of Section~\ref{sec:decomposition} realizes
the intersection complex as a derived direct summand of the direct image from a smooth proper
resolution. The retraction is defined over \(K\), so geometric cohomology gives \(G_K\)-equivariant
direct summands to which smooth proper de Rham comparison applies.

\begin{corollary}\label{cor:constant-derham}
For every proper rigid space \(X/K\) and every integer \(m\), the group \(\IH^m(X_C,\Q_p)\) is a de
Rham \(G_K\)-representation. Thus Theorem~\ref{thm:general} implies the assertion for proper spaces
in \cite[Conjecture 4.16(1)]{BH}.
\end{corollary}
\begin{proof}
First suppose that \(X\) is reduced and pure-dimensional of dimension \(d\). Choose a projective
resolution \(\pi:\widetilde X\to X\) over \(K\) as in Section~\ref{sec:singular}. Its source is
smooth and is proper over \(K\) because \(\pi\) and \(X\to\Spa(K)\) are proper. The morphism \(\pi\)
is an isomorphism over a smooth dense open subset. The decomposition with coefficients in \(\Q_p\)
and the splitting along the boundary used there give a derived retraction
\[
 \IC_{X,\Q_p}\xrightarrow{u}R\pi_*\Q_p[d]
 \xrightarrow{v}\IC_{X,\Q_p},\qquad vu=\operatorname{id},
\]
over \(K\). After base change to \(C\), proper base change identifies
\(R\Gamma(X_C,R\pi_{C*}\Q_p[d])\) with \(R\Gamma(\widetilde X_C,\Q_p)[d]\). Taking cohomology in
degree \(m\) gives a \(G_K\)-equivariant retraction
\begin{equation}\label{eq:derham-retract}
 \IH^m(X_C,\Q_p)\longrightarrow
 H^{m+d}(\widetilde X_C,\Q_p)\longrightarrow
 \IH^m(X_C,\Q_p).
\end{equation}
Scholze's smooth proper comparison theorem \cite[Corollary 1.8]{ScholzeHT} makes the middle
representation de Rham. Its direct summands are de Rham, proving the assertion in the
pure-dimensional case.

For general \(X\), nilpotent thickenings do not change the relevant \'etale objects. Apply the
preceding argument to each reduced irreducible component, with its own dimension shift, and use the
componentwise definition of the intersection complex. This also proves
Corollary~\ref{cor:intro-constant-derham}.
\end{proof}

If \(\mathbb L\) has finite geometric monodromy on each geometrically connected component of \(U\),
then \(\IH^m(X_C,L)\) is de Rham for every \(m\), in every dimension. The proof of
Proposition~\ref{prop:finite-geometric-derham} realizes this intersection cohomology, after a finite
extension \(K'/K\), as a \(G_{K'}\)-equivariant direct summand of finite sums of
\(\IH^m(Y_{\alpha,C},\Q_p)\otimes_{\Q_p}V_\alpha\), where each \(Y_\alpha\) is proper and
\(V_\alpha\) is de Rham. Corollary~\ref{cor:constant-derham} makes the first factor de Rham in every
degree. Tensor products and direct summands preserve this property, which then descends to the
original field.

\begin{proposition}\label{prop:extendible-derham}
Suppose that the de Rham \(\Z_p\)-local system \(\mathbb L\) on \(U\) extends to a
\(\Z_p\)-local system \(\mathbb M\) on \(X\). Then \(\IH^m(X_C,L)\) is de Rham for every integer
\(m\), where \(L=\mathbb L[1/p]\).
\end{proposition}
\begin{proof}
Put \(M=\mathbb M[1/p]\), and choose the resolution \(\pi:\widetilde X\to X\) used in the proof of
Corollary~\ref{cor:constant-derham}. Tensoring with the finite-rank local system \(M\) preserves
ordinary and perverse bounds, as well as the strict support and cosupport conditions for
intermediate extension. Consequently, \(M\otimes_{\Q_p}\IC_{X,\Q_p}\simeq P(L)\). Tensor the derived
retraction in that corollary with \(M\). The projection formula identifies
\(R\pi_*\Q_p[d]\otimes_{\Q_p}M\) with \(R\pi_*(\pi^*M)[d]\), so \(P(L)\) is a derived direct summand
of \(R\pi_*(\pi^*M)[d]\) over \(K\).

Let \(V\subseteq X\) be a dense open over which \(\pi\) is an isomorphism. The local system
\(\pi^*M\) is de Rham on \(\pi^{-1}(U\cap V)\). Every connected component of \(\widetilde X\) meets
this open, so Liu--Zhu rigidity \cite[Theorem 3.9(iv)]{LiuZhu} makes \(\pi^*M\) de Rham throughout
\(\widetilde X\). Since \(\widetilde X\) is smooth and proper, comparison with de Rham local
coefficients \cite[Theorem 1.6]{ScholzeHT} shows that \(H^{m+d}(\widetilde X_C,\pi^*M)\) is de Rham.
Applying geometric cohomology to the derived retraction exhibits \(\IH^m(X_C,L)\) as a
\(G_K\)-equivariant direct summand of this representation.
\end{proof}

\begin{proposition}\label{prop:realized-derham}
Let \(g:W\to X\) be proper with a global relatively ample line bundle, and suppose that \(W\) is
smooth of pure dimension \(N\). Assume that \(h:W_U=g^{-1}(U)\to U\) is smooth of pure relative
dimension \(e=N-d\). If \(L\) is a direct summand over \(K\) of \(R^qh_*\Q_p(t)\) for some integers
\(q,t\), then \(\IH^m(X_C,L)\) is de Rham for every integer \(m\). More precisely, it is a
\(G_K\)-equivariant direct summand of
\begin{equation}\label{eq:realized-derham}
 H^{m+q+d}(W_C,\Q_p)(t).
\end{equation}
\end{proposition}
\begin{proof}
Put \(A=Rg_*\Q_p[N]\), \(a=q-e\), and \(P_a=\pH^a(A)\). Smooth proper base change makes the higher
direct images of \(h\) local systems on \(U\). Since \(U\) is smooth of dimension \(d\), a local system
shifted by \([d]\) is perverse. Perverse exactness of open restriction therefore gives
\(j^*P_a\simeq(R^qh_*\Q_p)[d]\). The given inclusion and retraction exhibit \(L[d]\) as a direct
summand of \(j^*P_a(t)\), and intermediate extension preserves this splitting. The boundary
splitting of Theorem~\ref{thm:smoothsource} makes \(j_{!*}j^*P_a(t)\) a perverse direct summand of
\(P_a(t)\), so \(P(L)\) is a perverse direct summand of \(P_a(t)\).

The global derived splitting in the same theorem makes \(P_a[-a]\) a derived direct summand of
\(A\). Combining these retractions gives
\[
 P(L)[-a]\longrightarrow A(t)\longrightarrow P(L)[-a]
\]
with composite the identity, all over \(K\). Apply geometric global sections and take cohomology in
degree \(m+a\). The outer terms give \(H^{m+a}(X_C,P(L)[-a])=\IH^m(X_C,L)\), while the middle term
gives \(H^{m+a+N}(W_C,\Q_p)(t)\). Since \(a=q-e\) and \(e=N-d\), we have \(m+a+N=m+q+d\), proving
that \(\IH^m(X_C,L)\) is a direct summand as stated in \eqref{eq:realized-derham}. The space \(W\)
is proper over \(K\), because both \(g\) and \(X\to\Spa(K)\) are proper. Smooth proper comparison
\cite[Corollary 1.8]{ScholzeHT} makes the middle representation de Rham, and hence so is its
\(G_K\)-equivariant direct summand.
\end{proof}

\subsection{Further questions}

The projective decomposition theorem gives the de Rham property in every degree for constant
coefficients. The results of Section~\ref{sec:decomposition-coefficients} also give this property
for de Rham local systems that extend across the boundary, have finite geometric monodromy, or occur
as direct summands of the cohomology of a smooth proper family admitting the stated projective
realization. For arbitrary ramified de Rham coefficients with nonempty boundary,
Theorem~\ref{thm:boundary-derham} leaves the middle range \(-b<m<b\), where \(b=\dim(X\setminus
U)\).

For coefficients with \(\ell\ne p\), Hansen--Zavyalov \cite[Theorems 1.2.1(1) and 3.2.8]{HZ}
establish the assertion about nearby cycles in \cite[Conjecture 4.15(1)]{BH} in its corrected
formulation, using the nearby-cycle functor with values in Deligne's topos. Their Remark~1.2.2
compares this functor with derived pushforward along specialization to the special fiber over a
non-algebraically closed field. Lu--Zheng \cite[Corollary 1.3 and Remark 1.4]{LuZheng} establish the
nonnegativity of Frobenius weights predicted by \cite[Conjecture 4.15(2)]{BH}.

The conjecture that singular support is an analytic Lagrangian subset is formulated in
\cite[Conjecture 9.6]{ZhouMicro}. Section~\ref{sec:microlocal} gives isotropic upper bounds for
singular support with finite coefficients of characteristic different from \(p\) and proves the
conjecture for complexes supported in dimension at most one.

\section{Examples and further remarks}\label{sec:examples}

\subsection{Smooth families and restriction}\label{sec:tests}

\begin{corollary}\label{thm:smoothfamily}
Let \(f:X\to Y\) be smooth projective of pure relative dimension \(d\), where \(Y\) is qcqs,
reduced, and pure-dimensional of dimension \(m\). For every coefficient prime \(\ell\), the sheaves
\(L_q=R^qf_*\Q_\ell\) are local systems, and there are derived decompositions
\[
\begin{aligned}
 Rf_*\Q_\ell&\simeq\bigoplus_{q=0}^{2d}L_q[-q],\\
 Rf_*\IC_X&\simeq
 \bigoplus_{q=0}^{2d}\IC_Y(L_q|_{Y^{\mathrm{sm}}})[d-q].
\end{aligned}
\]
\end{corollary}
\begin{proof}
Smooth proper base change makes each \(L_q\) a local system, including for \(\ell=p\) by \cite[Corollary
6.4.8]{LRZ} and passage from finite to rational coefficients. Let \(H\) be a relatively ample line
bundle. Hard Lefschetz on the smooth projective geometric fibers, together with rigid GAGA and
proper base change, gives isomorphisms \(c_1(H)^a:L_{d-a}\xrightarrow{\sim}L_{d+a}(a)\) for \(0\le
a\le d\). These are morphisms of local systems induced by the global Chern class, and we can check
that they are isomorphisms on the geometric fibers. Applying Lemma~\ref{lem:deligne} to \(Rf_*\Q_\ell[d]\)
gives the first decomposition.

For the second decomposition, put \(U=Y^{\mathrm{sm}}\) and \(U_X=f^{-1}(U)\), and write
\(j:U\hookrightarrow Y\) and \(j_X:U_X\hookrightarrow X\) for the inclusions. Since \(f\) is smooth
and open, \(U_X\) is smooth and dense in \(X\), of pure dimension \(m+d\). Thus
\(\IC_Y=j_{!*}(\Q_{\ell,U}[m])\) and
\(\IC_X=(j_X)_{!*}(\Q_{\ell,U_X}[m+d])\). Shifted smooth pullback \(f^*[d]\) is perverse
\(t\)-exact. For \(A\in\Db(Y,\Q_\ell)\), it also satisfies
\(\mathbb D_X(f^*A[d])\simeq f^*(\mathbb D_YA)[d](d)\), by smooth duality
\cite[Corollary 7.1.2]{LRZ}, regular base change for internal Hom, and
\cite[Propositions 3.7 and 3.16]{BH}. Consequently, \(P=f^*\IC_Y[d]\) is perverse and
\(j_X^*P\simeq\Q_{\ell,U_X}[m+d]\).

Let \(i:Z=Y\setminus U\hookrightarrow Y\) and \(i_X:Z_X=f^{-1}(Z)\hookrightarrow X\) be the
boundary inclusions, and let \(f_Z:Z_X\to Z\) be the restriction of \(f\). The intermediate
extension \(\IC_Y\) satisfies
\(i^*\IC_Y,i^*\mathbb D_Y\IC_Y\in{}^pD^{\le-1}(Z,\Q_\ell)\). Ordinary base change and the
duality identity above give
\[
\begin{aligned}
 i_X^*P
 &\simeq f_Z^*(i^*\IC_Y)[d],\\
 i_X^*\mathbb D_XP
 &\simeq f_Z^*(i^*\mathbb D_Y\IC_Y)[d](d).
\end{aligned}
\]
The morphism \(f_Z\) is smooth of pure relative dimension \(d\), so \(f_Z^*[d]\) preserves
perverse bounds; Tate twists preserve them as well. Therefore both complexes belong to
\({}^pD^{\le-1}(Z_X,\Q_\ell)\). These two conditions exclude nonzero quotients and subobjects of
\(P\) supported on \(Z_X\). So \(\IC_X\simeq f^*\IC_Y[d]\).

Applying the projection formula to this identification and substituting the first decomposition,
we obtain
\[
\begin{aligned}
 Rf_*\IC_X
 &\simeq Rf_*(f^*\IC_Y)[d]\\
 &\simeq
   \bigl(\IC_Y\otimes_{\Q_\ell}^{\mathbf L}Rf_*\Q_\ell\bigr)[d]\\
 &\simeq\bigoplus_{q=0}^{2d}(\IC_Y\otimes_{\Q_\ell}L_q)[d-q].
\end{aligned}
\]
For each \(q\), put \(B_q=\IC_Y\otimes_{\Q_\ell}L_q\). Since \(L_q\) is a finite-rank local
system on all of \(Y\), tensoring with it is perverse \(t\)-exact and commutes with restriction.
Hence \(B_q\) is perverse, \(j^*B_q\simeq(L_q|_U)[m]\), and
\(i^*B_q\simeq i^*\IC_Y\otimes_{\Q_\ell}i^*L_q\) belongs to \({}^pD^{\le-1}(Z,\Q_\ell)\).
Writing \(L_q^\vee=\mathcal Hom_{\Q_\ell}(L_q,\Q_\ell)\), Verdier duality gives
\(\mathbb D_YB_q\simeq\mathbb D_Y\IC_Y\otimes_{\Q_\ell}L_q^\vee\), so
\(i^*\mathbb D_YB_q\) satisfies the same strict upper bound. Thus
\(B_q\simeq j_{!*}((L_q|_U)[m])=\IC_Y(L_q|_{Y^{\mathrm{sm}}})\).
Substituting these identifications proves the second decomposition.
\end{proof}

Finite pushforward is perverse \(t\)-exact and commutes with intermediate extension in the
corresponding Cartesian diagrams. Let \(f:X\to Y\) be projective, with \(Y\) qcqs and \(X\) smooth
of pure dimension \(n\), and put \(Y_r=\{y\in Y:\dim f^{-1}(y)\ge r\}\). The morphism \(f\) is
semismall if \(\dim Y_r\le n-2r\) for every \(r\ge0\), and small if these inequalities are strict
for \(r>0\). In the semismall case, proper base change and the bound on the cohomological dimension
of the fibers give the perverse support inequalities for \(Rf_*\Q_\ell[n]\), and proper duality
gives the cosupport inequalities. Thus \(Rf_*\Q_\ell[n]=Rf_*\IC_X\) is perverse, and its support
decomposition follows from Theorem~\ref{thm:smoothsource}. For a small resolution \(f\), the strict
support and cosupport inequalities give \(Rf_*\Q_\ell[n]\simeq\IC_Y\).

\begin{example}\label{ex:nodecut}
Consider the algebraic threefold ordinary double point \(Z=\{xy-zw=0\}\subset\mathbb A^4_K\),
whose only singular point is the vertex \(0\). Let
\(T=\operatorname{Tot}(\cO_{\mathbf P^1}(-1)^{\oplus2})\). In homogeneous coordinates \([s:t]\)
on \(\mathbf P^1\) and fiber coordinates \(u,v\), define \(\pi:T\to Z\) by \(x=su\), \(z=sv\),
\(w=tu\), and \(y=tv\). These formulas identify \(T\) with the closed subspace of
\(Z\times\mathbf P^1\) defined by \(tx=sw\) and \(tz=sy\), so \(\pi\) is projective. It is an
isomorphism away from \(0\), and its exceptional fiber is \(E=\pi^{-1}(0)\simeq\mathbf P^1\).
Since \(T\) is smooth of dimension three and \(2\dim E=2<3=\operatorname{codim}_Z\{0\}\),
\(\pi\) is a small resolution.

Take \(h=x+y\), so \(h\circ\pi=su+tv\). On the chart \(s\ne0\), write \(\lambda=t/s\) and
trivialize the bundle to obtain \(h\circ\pi=u+\lambda v\); its derivative with respect to \(u\)
is \(1\). On the chart \(t\ne0\), the analogous expression is \(\mu u+v\), with \(\mu=s/t\).
Thus \(h\circ\pi\) is smooth. Put \(Z_0=h^{-1}(0)\) and \(T_0=T\times_Z Z_0\), and let
\(\pi_0:T_0\to Z_0\) be the induced map. Substituting \(y=-x\) gives
\(Z_0=\{x^2+zw=0\}\), an \(A_1\) surface singularity. The equation of \(T_0\) is \(su+tv=0\),
so it is the total space of the kernel in
\[
 0\longrightarrow\cO_{\mathbf P^1}(-2)
 \longrightarrow\cO_{\mathbf P^1}(-1)^{\oplus2}
 \xrightarrow{(s,t)}\cO_{\mathbf P^1}\longrightarrow0.
\]
The map \((s,t)\) is surjective, and its kernel is a line bundle of degree \(-2\). Hence
\(T_0\simeq\operatorname{Tot}(\cO_{\mathbf P^1}(-2))\). The exceptional curve \(E\) is its zero
section, with normal bundle \(N_{E/T_0}\simeq\cO_{\mathbf P^1}(-2)\) and self-intersection \(E^2=-2\).
The equality \(2\dim E=2=\operatorname{codim}_{Z_0}\{0\}\) shows that \(\pi_0\) is semismall
and is no longer small.

We now analytify and restrict to a qcqs neighborhood of the vertex and its inverse images, keeping
the same notation. Smallness gives \(\IC_Z\simeq R\pi_*\Q_\ell[3]\), while semismallness makes
\(A_0=R\pi_{0*}\Q_\ell[2]\) perverse. Proper base change and purity for the divisor
\(E\hookrightarrow T_0\) identify the boundary morphism \(c_{\{0\}}(A_0)\) on the geometric
vertex with
\[
 H^0(E_C,\Q_\ell)(-1)
 \xrightarrow{\ \cup c_1(N_{E/T_0})\ }
 H^2(E_C,\Q_\ell).
\]
Both sides are \(\Q_\ell(-1)\), and the map is multiplication by
\(\deg N_{E/T_0}=-2\). It is therefore invertible for every prime \(\ell\), including
\(\ell=2\). Away from the vertex, \(A_0\) is the constant sheaf shifted by \([2]\).
Lemma~\ref{lem:boundary} consequently splits \(A_0\) into \(\IC_{Z_0}\) and the summand
\(\Q_{\ell,\{0\}}(-1)\), where \(\Q_{\ell,\{0\}}\) denotes the constant sheaf on the vertex
extended by zero to \(Z_0\). For \(i:Z_0\hookrightarrow Z\), proper base change gives
\[
\begin{aligned}
 i^*\IC_Z[-1]
 &\simeq i^*R\pi_*\Q_\ell[3][-1]\\
 &\simeq R\pi_{0*}\Q_\ell[2]\\
 &\simeq\IC_{Z_0}\oplus\Q_{\ell,\{0\}}(-1).
\end{aligned}
\]
Thus \(i^*\IC_Z[-1]\) is perverse, but its nonzero summand supported at the vertex prevents it
from being the intermediate extension of its restriction to \(Z_0\setminus\{0\}\).

To compare the two vanishing conditions in Proposition~\ref{prop:boundaryvanish}, pass to \(C\)
and again work on qcqs analytic neighborhoods, using the same notation. Fix \(\ell\ne p\) and
\(n\ge1\), and put \(\Lambda_n=\Z/\ell^n\) and \(K_n=R\pi_*\Lambda_{n,T}[3]\).
Smooth local acyclicity for \(h\circ\pi\), followed by compatibility with proper pushforward, gives
\[
 \phi_h(K_n)
 \simeq R\pi_{0*}\phi_{h\circ\pi}(\Lambda_{n,T}[3])
 =0.
\]
For the vertex inclusion \(v:\{0\}\hookrightarrow Z\), put \(B_n=v_*v^*K_n\). Proper base change
identifies \(B_n\) with \(v_*R\Gamma(E_C,\Lambda_n)[3]\), which is nonzero since
\(H^0(E_C,\Lambda_n)=\Lambda_n\). Let \(j:Z\setminus Z_0\hookrightarrow Z\) be the open
complement. The complex \(B_n\) is supported at the vertex, so \(j^*B_n=0\) and
\(\psi_h(j^*B_n)=0\). The defining triangle
\[
 i^*B_n\longrightarrow\psi_h(j^*B_n)
 \longrightarrow\phi_h(B_n)\longrightarrow i^*B_n[1]
\]
therefore gives \(\phi_h(B_n)\simeq i^*B_n[1]\ne0\). The first vanishing condition in
\eqref{eq:twovanishing} holds for \(K_n\), while the second fails for the boundary \(\{0\}\).
This calculation shows why control of the total direct image alone does not ensure preservation
of intermediate extension under shifted restriction.
\end{example}

\subsection{A projective family with nonsemisimple monodromy}\label{sec:tate}

\begin{theorem}\label{thm:tate}
Let \(\pi\) be a uniformizer of \(K\), and consider the closed annulus \(A=\Spa(K\langle
T/\pi,\pi^2/T\rangle)\), where \(|\pi|^2\le|T|\le|\pi|\). There exists a smooth projective
family of elliptic curves \(f:\cE\to A\) such that \(R^1f_{C*}\Q_\ell\) is nonsemisimple on
\(A_C\) for every prime \(\ell\). More precisely, its rational Tate module \(V_\ell(\cE_C)\)
fits into a nonsplit exact
sequence
\begin{equation}\label{eq:tate}
 0\longrightarrow\Q_\ell(1)\longrightarrow V_\ell(\cE_C)
 \longrightarrow\Q_\ell\longrightarrow0,
\end{equation}
whose extension class is the Kummer class of \(T\). The complex \(Rf_*\Q_\ell\) has the
decomposition of Corollary~\ref{thm:smoothfamily}.
\end{theorem}
\begin{proof}
\smallskip
\noindent\emph{Step 1. Construct the elliptic family.}
Put \(R=K\langle T/\pi,\pi^2/T\rangle\). The inequalities
\(|\pi|^2\le|T|\le|\pi|<1\) ensure that \(T\) is a unit of \(R\) and is topologically nilpotent.
The universal Tate Weierstrass equation is
\(y^2+xy=x^3+a_4(q)x+a_6(q)\), where \(a_4(q),a_6(q)\in q\Z[[q]]\), and its discriminant is
\(\Delta(q)=q\prod_{n\ge1}(1-q^n)^{24}\). Substituting \(q=T\) gives convergent series in \(R\).
Each factor \(1-T^n\) is a unit because \(\|T^n\|<1\), and the infinite product converges to a
unit since \(T^n\to0\) uniformly. Hence \(\Delta(T)\in R^\times\). Homogenizing the equation
defines a smooth cubic \(\cE_R\subset\mathbf P^2_R\), with the point at infinity as a section.
Its relative analytification over \(A=\Spa(R)\) is a smooth projective family of elliptic curves
\(f:\cE\to A\), whose fibers are Tate curves with parameter given by \(T\).

\smallskip
\noindent\emph{Step 2. Identify the torsion extension.}
Fix an integer \(N\ge1\). In the multiplicative uniformization of a Tate curve, a point represented
by \(z\) is killed by \(N\) precisely when \(z^N=q^a\) for some integer \(a\).
Replacing \(z\) by \(zq^b\), with \(b\in\Z\), changes \(a\) by \(Nb\), so one can take
\(0\le a<N\). With \(q=T\), this gives the presentation
\cite[\S2]{Huan}, following \cite[\S8.7]{KM},
\begin{equation}\label{eq:torsion}
\begin{gathered}
 \cE[N]=\coprod_{a=0}^{N-1}
 \bigl(\Spec R[z]/(z^N-T^a)\bigr)^{\an},\\
 0\longrightarrow\mu_N\longrightarrow\cE[N]
 \longrightarrow\Z/N\longrightarrow0.
\end{gathered}
\end{equation}
Here analytification is relative over \(A\). The group law comes from
multiplication of representatives modulo powers of \(T\); the index \(a\bmod N\) is therefore
additive and defines the quotient map. Its kernel is the component \(z^N=1\), namely \(\mu_N\).
Each component is finite \'{e}tale over \(A\): both \(N\) and \(z\) are units, so the derivative
\(Nz^{N-1}\) is invertible. This applies also when \(N\) is a power of \(p\).

The inverse image of \(1\in\Z/N\) is the \(\mu_N\)-torsor \(z^N=T\). Since
\eqref{eq:torsion} is an extension of \(\Z/N\)-module sheaves, a local lift of \(1\) determines
a local splitting. On overlaps, two such lifts differ by an \(N\)-th root of unity. The resulting
cocycle is exactly the Kummer cocycle of \(T\), so the extension class is its image under the
Kummer map \(\cO(A)^\times\to H^1(A,\mu_N)\). Now take \(N=\ell^n\). The transition maps are
multiplication by \(\ell\) on the torsion groups, the \(\ell\)-th power map on roots of unity,
and reduction on the quotients. Since the maps on roots of unity are surjective on geometric
fibers, taking inverse limits gives
\[
 0\longrightarrow\Z_\ell(1)\longrightarrow T_\ell(\cE)
 \xrightarrow{\beta}\Z_\ell\longrightarrow0,
\]
where \(T_\ell(\cE)=\varprojlim_n\cE[\ell^n]\). Its extension class is the compatible integral
Kummer class of \(T\). Tensoring with \(\Q_\ell\) defines
\(V_\ell(\cE)=T_\ell(\cE)\otimes_{\Z_\ell}\Q_\ell\); after base change to \(C\), the sequence
becomes \eqref{eq:tate}.

\smallskip
\noindent\emph{Step 3. Prove nonsplitting over \(C\).}
Every analytic unit on \(A_C\) can be written as \(u=aT^k(1+h)\), where \(a\in C^\times\),
\(k\in\Z\), and \(h\in\cO(A_C)\) satisfies \(\|h\|<1\); see
\cite[p.~213, (2.4.6)]{Ramero}. The integer \(k\) is the dominant Laurent exponent of \(u\)
throughout the annulus. It is uniquely determined and is called its winding number \(w(u)\).
Multiplication adds these exponents, giving a homomorphism
\(w:\cO(A_C)^\times\to\Z\) with \(w(uv)=w(u)+w(v)\) and \(w(T)=1\).

Let \(c_n\in H^1(A_C,\mu_{\ell^n})\) be the Kummer class of \(T\). These classes are compatible
under the transition maps. Since \(A_C\) is connected and \(C\) contains all roots of unity,
\(H^0(A_C,\mu_{\ell^n})=\mu_{\ell^n}(C)\), and the transition maps on these groups are
surjective. Thus the term \(\varprojlim^1_nH^0(A_C,\mu_{\ell^n})\) vanishes, and the compatible
classes define \(c\in H^1(A_C,\Z_\ell(1))\). This is the class of the integral Tate module
sequence after base change to \(A_C\). Suppose that \(\ell^m c=0\) for some \(m\ge0\).
Reduction modulo \(\ell^n\) gives \(\ell^m c_n=0\), which is the Kummer class of
\(T^{\ell^m}\). Exactness of the Kummer sequence therefore gives a unit
\(u_n\in\cO(A_C)^\times\) such that \(T^{\ell^m}=u_n^{\ell^n}\). Applying \(w\) gives
\(\ell^m=\ell^n w(u_n)\), contradicting \(w(u_n)\in\Z\) when \(n>m\). Hence no power of
\(\ell\) annihilates \(c\).

On the qcqs base \(A_C\), rationalization localizes morphisms by inverting \(\ell\).
If \eqref{eq:tate} split, clearing denominators and multiplying by a further power of \(\ell\)
would give an integral map \(s:\Z_\ell\to T_\ell(\cE_C)\) with
\(\beta\circ s=\ell^m\operatorname{id}_{\Z_\ell}\) for some \(m\ge0\). The exact Hom sequence
of the integral extension would then give \(\ell^m c=0\), a contradiction.
Thus \eqref{eq:tate} is nonsplit. This argument applies to every prime \(\ell\), including \(p\).

Both rank-one terms of \eqref{eq:tate} are constant on \(A_C\). If its middle term were
semisimple, the subobject \(\Q_\ell(1)\) would have an invariant complement, splitting the
sequence. Thus \(V_\ell(\cE_C)\) is nonsemisimple. The identification
\(R^1f_{C*}\Q_\ell\simeq V_\ell(\cE_C)^\vee\) gives the nonsplit dual sequence
\[
 0\longrightarrow\Q_\ell\longrightarrow R^1f_{C*}\Q_\ell
 \longrightarrow\Q_\ell(-1)\longrightarrow0,
\]
so \(R^1f_{C*}\Q_\ell\) is also nonsemisimple. Finally, the geometric fibers of \(f\) are
connected curves, giving \(R^0f_*\Q_\ell\simeq\Q_\ell\), while the trace map gives
\(R^2f_*\Q_\ell\simeq\Q_\ell(-1)\). Corollary~\ref{thm:smoothfamily} therefore yields
\[
 Rf_*\Q_\ell\simeq
 \Q_\ell\oplus R^1f_*\Q_\ell[-1]\oplus\Q_\ell(-1)[-2].
\]
\end{proof}

\begin{remark}
For a single Tate elliptic curve \(E_q/K\) with parameter \(q\in K\), \(0<|q|<1\), the same
construction gives a nonsplit \(G_K\)-extension \(0\to\Q_\ell(1)\to V_\ell(E_q)\to\Q_\ell\to0\). The
nonzero valuation of \(q\) detects its rational Kummer class. In the annular family,
nonsemisimplicity persists after base change to \(C\).
\end{remark}

\subsection{An isotropic upper bound for singular support}
\label{sec:microlocal}

Let \(k\) be a complete algebraically closed, non-trivially valued field of characteristic zero and
residue characteristic \(p\), and let \(\Lambda=\mathbf F_{\ell^r}\), where \(\ell\ne p\) is prime
and \(r\ge1\). For a smooth rigid space \(M/k\) and \(F\in\Db(M,\Lambda)\), write \(SS(F)\subset
T^*M\) for its singular support, defined by local test functions in \cite[\S8]{ZhouMicro}. We first
construct an analytic isotropic upper bound, then determine the singular support of complexes
supported in dimension at most one.

\begin{proposition}\label{prop:isotropic-container}
Let \(M/k\) be smooth affinoid of pure dimension \(n\), and let \(F\in\Db(M,\Lambda)\). There is a
closed conic analytic subset \(C_F\subset T^*M\) containing \(SS(F)\) such that every irreducible
component of \(C_F\) is isotropic on its smooth locus. In particular, every component has dimension
at most \(n\).
\end{proposition}
\begin{proof}
\smallskip
\noindent\emph{Step 1. Reduce to proper direct images from smooth spaces.}
By \cite[Proposition 3.6]{BH}, the constructible category on \(M\) is thickly generated by finite
pushforwards of constant sheaves with finite coefficients. Let \(f:Y\to M\) be such a finite
morphism. Passing to the reduction and using the Mayer--Vietoris triangles for closed covers, we
reduce to an irreducible \(Y\). Choose a resolution \(\pi:\widetilde Y\to Y\) that is an isomorphism
outside a closed subset \(Z\), and put \(E=\pi^{-1}Z\). For the inclusion \(i:Z\hookrightarrow Y\),
proper base change gives the abstract blow-up triangle
\[
 \Lambda_Y\longrightarrow
 R\pi_*\Lambda_{\widetilde Y}\oplus i_*\Lambda_Z
 \longrightarrow R(\pi|_E)_*\Lambda_E\longrightarrow\Lambda_Y[1].
\]
The dimensions of \(Z\) and \(E\) are smaller than that of \(Y\). Repeating the construction
terminates, so \(F\) lies in the thick subcategory generated by finitely many
\(Rg_{a*}\Lambda_{V_a}\), with \(g_a:V_a\to M\) proper and \(V_a\) smooth.

\smallskip
\noindent\emph{Step 2. Construct the critical incidence set.}
For a proper map \(g:V\to M\) with \(V\) smooth, let \(dg^*:g^*T^*M\to T^*V\) be the cotangent map
and set
\[
 Z_g=\{(v,\xi)\in V\times_M T^*M:dg_v^*(\xi)=0\},
 \qquad C_g=\operatorname{im}(Z_g\to T^*M).
\]
The equations define the zero locus of a section of the pullback of \(T^*V\), so \(Z_g\) is a closed
analytic subspace. The projection \(V\times_M T^*M\to T^*M\) is proper, and its restriction to
\(Z_g\) is proper. The proper mapping theorem gives \(C_g\) a reduced closed analytic structure. It
is conic because the equations are homogeneous in \(\xi\). A constant sheaf on a smooth space has
singular support equal to the zero section. The inclusion under proper pushforward \cite[Lemma
8.8(2)]{ZhouMicro} therefore gives \(SS(Rg_*\Lambda_V)\subseteq C_g\).

\smallskip
\noindent\emph{Step 3. Prove isotropy.}
Let \(\lambda_M\) be the tautological one-form on \(T^*M\). Choose local coordinates \(x_i\) on
\(M\), write \(g_i=x_i\circ g\), and let \(\xi_i\) be the corresponding cotangent coordinates. The
pullback of \(\lambda_M\) to \(Z_g\) is \(\sum_i\xi_i\,dg_i\), which vanishes by the defining
equations of \(Z_g\). Work over an affinoid open in \(T^*M\), and let \(C\) be an irreducible
component of \(C_g\) there. Choose an irreducible component of \((Z_g)_{\mathrm{red}}\) whose image
is \(C\). Such a component exists because there are finitely many irreducible components of the
source over the chosen affinoid, their proper images are closed, and their union covers \(C\).
Resolving this component gives a proper surjection \(h:\widetilde Z\to C\), with \(\widetilde Z\)
smooth and \(h^*\lambda_M=0\). Generic smoothness for proper morphisms \cite[Theorem 2.29]{BH} gives
a dense smooth open \(C^\circ\subset C^{\mathrm{sm}}\) over which \(h\) is smooth and surjective.
The restriction \(h^\circ:h^{-1}(C^\circ)\to C^\circ\) is faithfully flat, and the cotangent
sequence gives an injection
\((h^\circ)^*\Omega^1_{C^\circ/k}\hookrightarrow\Omega^1_{h^{-1}(C^\circ)/k}\). Thus
\(h^*\lambda_M=0\) implies \((h^\circ)^*(\lambda_M|_{C^\circ})=0\), and faithful flatness gives
\(\lambda_M|_{C^\circ}=0\). By density, this form and its exterior derivative vanish on
\(C^{\mathrm{sm}}\). Hence \(C\) is isotropic and has dimension at most \(n\). The argument is local
on \(T^*M\).

\smallskip
\noindent\emph{Step 4. Assemble the upper bound.}
For the finitely many maps chosen in Step 1, put \(C_F=\bigcup_a C_{g_a}\). This is a closed conic
analytic subset whose irreducible components are isotropic. The class of complexes whose singular
support is contained in \(C_F\) is closed under shifts, cones, and direct summands. Therefore
\(SS(F)\subseteq C_F\).
\end{proof}

\begin{lemma}\label{lem:disc-vanishing}
Let \(D\) be a smooth disc over \(k\) with coordinate \(t\), let
\(j:D^\times=D\setminus\{0\}\hookrightarrow D\), and let \(L\) be a finite-rank \(\Lambda\)-local
system on \(D^\times\), trivialized by a finite Kummer cover \(t=s^N\) for some \(N\ge1\). Let \(V\)
be a geometric fiber of \(L\), and let \(T\in\operatorname{GL}_\Lambda(V)\) be the action of a
generator of its Kummer monodromy group. Put \(V^T=\ker(T-\operatorname{id}_V)\) and
\(Q=j_{!*}(L[1])\). If \(h\) is an analytic function on \(D\) with \(h(0)=0\) and order \(m\ge1\) at
\(0\), then, after identifying the nearby fibers with \(V\), there is an exact sequence of
\(\Lambda\)-vector spaces
\[
 0\longrightarrow V^T\longrightarrow V^{\oplus m}
 \longrightarrow \cH^0\bigl(\phi_h(Q)[-1]\bigr)_0
 \longrightarrow0.
\]
In particular, \(\dim_\Lambda \cH^0(\phi_h(Q)[-1])_0
=m\dim_\Lambda V-\dim_\Lambda V^T\).
\end{lemma}
\begin{proof}
On a smooth curve, the intermediate extension is
\(Q\simeq(\cH^0(Rj_*L))[1]\), where \(\cH^0(Rj_*L)\) is the ordinary direct image of \(L\).
Indeed, the construction of intermediate extension across a point gives
\(Q\simeq\tau^{\le-1}Rj_*(L[1])\), and \(Rj_*L\) has no negative ordinary cohomology. A section
of \(L\) near the puncture descends from the Kummer cover precisely when its value is fixed by
the cyclic monodromy group. Hence \((\cH^0(Rj_*L))_0=V^T\), and therefore \(Q_0=V^T[1]\).

Write \(h(t)=a t^m u(t)\), where \(a\in k^\times\) and \(u(0)=1\). Since \(k\) has characteristic
zero, the derivative with respect to \(z\) of \(z^m-u(t)\) at \((t,z)=(0,1)\) is \(m\ne0\).
The analytic implicit function theorem therefore gives an analytic function \(z(t)\) satisfying
\(z(t)^m=u(t)\) and \(z(0)=1\) after shrinking \(D\). Choose an \(m\)-th root of \(a\) in \(k\)
and put \(t'=a^{1/m}tz(t)\). Since \(dt'/dt(0)=a^{1/m}\ne0\), the analytic inverse function
theorem makes \(t'\) a coordinate on a smaller disc, and \(h=(t')^m\). This applies to every
\(m\ge1\), including when \(p\mid m\).

For a sufficiently small nonzero value \(\varepsilon\), the equation \((t')^m=\varepsilon\) has
\(m\) distinct solutions, all lying in the chosen disc. The geometric Milnor fiber thus consists
of \(m\) points. Its cohomology with coefficients in \(L\) is concentrated in degree zero and,
after choosing fiber identifications, equals \(V^{\oplus m}\). The interpretation of nearby
cycles in \cite[Proposition 2.7]{ZhouMicro} gives
\(\psi_h(j^*Q)_0\simeq(V^{\oplus m})[1]\). Choose the fiber identifications using a Kummer cover
trivializing \(L\). The specialization map sends an invariant section \(v\in V^T\) to its
transported values at these \(m\) points. Its composition with any coordinate projection is the
inclusion \(V^T\hookrightarrow V\), since the monodromy fixes \(v\). It is therefore injective.
The triangle defining vanishing cycles becomes
\[
 V^T[1]\xrightarrow{\mathrm{sp}}V^{\oplus m}[1]
 \longrightarrow\phi_h(Q)_0\longrightarrow V^T[2].
\]
Its ordinary cohomology sequence is
\[
 0\longrightarrow\cH^{-2}(\phi_h(Q))_0
 \longrightarrow V^T\xrightarrow{\mathrm{sp}}V^{\oplus m}
 \longrightarrow\cH^{-1}(\phi_h(Q))_0\longrightarrow0,
\]
and all other cohomology groups of \(\phi_h(Q)_0\) vanish. Injectivity of \(\mathrm{sp}\) gives
\(\cH^{-2}(\phi_h(Q))_0=0\). Since
\(\cH^{-1}(\phi_h(Q))_0=\cH^0(\phi_h(Q)[-1])_0\), this is the stated short exact sequence.
Taking dimensions gives the formula.
\end{proof}

For a smooth locally closed subspace \(S\subset M\), let \(\overline{T_S^*M}\) be the closure of its
conormal bundle in \(T^*M\).

\begin{proposition}\label{prop:curve-supported-ss}
Let \(M/k\) be smooth, quasi-compact, and pure-dimensional of dimension \(n\). If
\(F\in\Db(M,\Lambda)\) has support of dimension at most one, then \(SS(F)\) is an analytic
Lagrangian subset of \(T^*M\). More precisely, it is a finite union of conormal closures of smooth
curves and cotangent fibers at classical points.
\end{proposition}
\begin{proof}
\smallskip
\noindent\emph{Step 1. Reduce to a simple perverse sheaf.}
The perverse cohomology sheaves of \(F\) have finitely many simple constituents, all supported in
dimension at most one. By \cite[Proposition 8.11(2)]{ZhouMicro}, \(SS(F)\) is the union of their
singular supports, so it suffices to treat one such constituent \(P\). If
\(P\) is supported at a point \(x\), then \(SS(P)=T_x^*M\) by \cite[Lemma 8.8(2)]{ZhouMicro}.
Otherwise, let \(C\) be its support, viewed as a reduced irreducible curve, and let \(g:\widetilde
C\to C\hookrightarrow M\) be its finite normalization. The curve \(\widetilde C\) is smooth because
\(k\) has characteristic zero. Choose a dense smooth open \(C^\circ\subset C\) on which \(P\)
is a nonzero local system shifted by \([1]\). Simplicity identifies \(P\) with the intermediate
extension of this restriction. The normalization is an isomorphism over \(C^\circ\), so the
local system pulls back to a local system \(L\) on \(W=g^{-1}(C^\circ)\). Write
\(\widetilde j:W\hookrightarrow\widetilde C\) and put \(Q=\widetilde j_{!*}(L[1])\). Finite
pushforward is perverse \(t\)-exact and commutes with intermediate extension
\cite[Theorem 4.2(5)--(7)]{BH}, giving \(P=g_*Q\). Let \(S\) be the finite set where \(Q[-1]\)
is not locally constant. These points describe the failure of the local system to extend on the
normalization. Put \(R=\{y\in\widetilde C:dg_y=0\}\), the set where the normalization map has
zero differential. It is finite because \(g\) is an immersion over \(C^{\mathrm{sm}}\). We will prove
\begin{equation}\label{eq:curve-supported-ss}
 SS(P)=\overline{T^*_{C^{\mathrm{sm}}}M}
 \,\cup\,\bigcup_{x\in g(S\cup R)}T_x^*M.
\end{equation}

\smallskip
\noindent\emph{Step 2. Bound the singular support.}
On \(\widetilde C\setminus S\), the singular support of \(Q\) is the zero section. The inclusion
for singular support under proper pushforward \cite[Lemma 8.8(2) and Example 8.14(1)]{ZhouMicro}
therefore bounds the possible covectors over \(g(y)\), for \(y\notin S\), by
\(\ker(dg_y^*:T^*_{g(y)}M\to T_y^*\widetilde C)\). Away from \(R\), these kernels form a
vector bundle of rank \(n-1\). Over \(g^{-1}(C^{\mathrm{sm}})\) it is the conormal bundle, and
density shows that all its fibers lie in the conormal closure. At a point of \(R\), the map
\(dg_y^*\) is zero, so the bound permits the whole fiber \(T^*_{g(y)}M\). The points of \(S\)
can contribute only covectors over \(g(S)\). This proves that \(SS(P)\) is contained in the
right-hand side of \eqref{eq:curve-supported-ss}. Conversely, on the dense open \(C^\circ\), the
sheaf \(P\) is the direct image of a nonzero local system shifted by \([1]\) under a closed
immersion into an open subset of \(M\). The formula for a closed immersion gives its conormal
bundle inside \(SS(P)\). Taking the closure gives
\(\overline{T^*_{C^{\mathrm{sm}}}M}\subseteq SS(P)\). It remains to detect the indicated cotangent
fibers.

\smallskip
\noindent\emph{Step 3. Compute the local vanishing cycles.}
Fix a classical point \(x\in C(k)\). The finitely many points of \(g^{-1}(x)\) correspond to the
branches of \(C\) at \(x\). Let \(f\) be a local function on \(M\) with \(f(x)=0\), and suppose
that \(h_y=f\circ g\) is nonconstant at each \(y\in g^{-1}(x)\). Put
\(m_y=\operatorname{ord}_y(h_y)\ge1\). Since
\(\Lambda\) is finite, the restriction of \(L\) to a punctured disc around \(y\) is trivialized by a
finite \'etale frame cover. After shrinking the disc, this cover becomes a disjoint union of Kummer
covers by \cite[proof of Proposition 8.12 and footnote 11]{ZhouMicro}. Thus \(L\) satisfies the
hypothesis on the Kummer cover in Lemma~\ref{lem:disc-vanishing}. Write \(V_y\) for its fiber,
\(r_y=\dim_\Lambda V_y\), and \(T_y\) for its local monodromy. Lemma~\ref{lem:disc-vanishing}
gives \(\dim_\Lambda\cH^0(\phi_{h_y}(Q)[-1])_y=m_yr_y-\dim_\Lambda V_y^{T_y}\).
Vanishing cycles commute with proper pushforward. Since \(g\) is finite, taking the stalk at
\(x\) gives the direct sum over its finite fiber:
\[
 \phi_f(P)_x\simeq
 \bigoplus_{y\in g^{-1}(x)}\phi_{h_y}(Q)_y.
\]
Taking cohomology after the shift \([-1]\) and adding the dimensions yields
\begin{equation}\label{eq:curve-vanishing-rank}
 \dim_\Lambda \cH^0\bigl(\phi_f(P)[-1]\bigr)_x
 =\sum_{y\in g^{-1}(x)}
 \bigl(m_yr_y-\dim_\Lambda V_y^{T_y}\bigr).
\end{equation}
Each summand is at least \((m_y-1)r_y\ge0\), since \(\dim_\Lambda V_y^{T_y}\le r_y\).
Thus a positive contribution from one branch makes \(\phi_f(P)_x\) nonzero.

\smallskip
\noindent\emph{Step 4. Detect the cotangent fibers.}
Suppose \(x\in g(S\cup R)\), and fix \(\xi\in T_x^*M(k)\). Smoothness of \(M\) gives a local
function \(f_0\) with \(f_0(x)=0\) and \(df_0(x)=\xi\). Choose local coordinates centered at
\(x\). For each branch above \(x\), the linear combinations of these coordinates that vanish
identically on that branch form a proper linear subspace of the space of coefficients. Since
there are finitely many branches and \(k\) is infinite, we can choose a linear combination \(z\)
that is nonconstant on every branch. For a fixed branch, at most one scalar \(a\in k\) makes
\(f_0\circ g+a(z\circ g)^2\) identically zero. Choose \(a\) outside these finitely many values
and put \(f=f_0+az^2\). Then \(f\circ g\) is nonconstant on every branch, while \(z(x)=0\)
gives \(f(x)=0\) and \(df(x)=\xi\).

If \(y\in R\), the chain rule gives \(d(f\circ g)_y=df_x\circ dg_y=0\). Since \(f\circ g\)
is nonconstant and vanishes at \(y\), its order satisfies \(m_y\ge2\). The corresponding
summand in \eqref{eq:curve-vanishing-rank} is therefore at least \(r_y>0\). If \(y\in S\),
the Kummer description in Step 3 shows that \(T_y\ne\operatorname{id}\): trivial monodromy
would make \(L\) constant on a sufficiently small punctured disc, with constant intermediate
extension on the whole disc. Hence \(V_y^{T_y}\) is a proper subspace of \(V_y\), and the
corresponding summand is at least \(r_y-\dim_\Lambda V_y^{T_y}>0\). At least one point of
\(g^{-1}(x)\) belongs to \(S\cup R\), so \eqref{eq:curve-vanishing-rank} gives
\(\phi_f(P)_x\ne0\). The definition using test functions gives \((x,\xi)\in SS(P)\).
This holds for every classical covector \(\xi\) over \(x\). Their density in \(T_x^*M\) and
closedness of \(SS(P)\) give \(T_x^*M\subseteq SS(P)\). Together with Step 2, this proves
\eqref{eq:curve-supported-ss}.

\smallskip
\noindent\emph{Step 5. Prove analyticity and the Lagrangian property.}
Put \(J=\operatorname{im}(g^*\Omega_M^1\to\Omega_{\widetilde C}^1)\) and
\(E=\ker(g^*\Omega_M^1\to J)\). Since \(g\) is generically an immersion, \(J\) is a coherent
subsheaf of the line bundle \(\Omega_{\widetilde C}^1\) with rank one. It is torsion-free and
therefore locally free on the smooth curve \(\widetilde C\). Thus the surjection
\(g^*\Omega_M^1\to J\) locally splits, and its kernel \(E\) is a vector bundle of rank \(n-1\).
Over \(\widetilde C\setminus R\), its fibers are the kernels of \(dg_y^*\) used in Step 2;
the bundle \(E\) extends these conormal directions across \(R\).

Let \(H=\overline{T^*_{C^{\mathrm{sm}}}M}\). The total space \(\operatorname{Tot}(E)\) is a
closed subspace of \(\widetilde C\times_M T^*M\). The projection of this fiber product to
\(T^*M\) is finite because it is a base change of \(g\), so the induced map
\(\rho:\operatorname{Tot}(E)\to T^*M\) is finite. Its image is therefore closed analytic.
Over \(g^{-1}(C^{\mathrm{sm}})\), the map \(\rho\) identifies the total space of \(E\) with
\(T^*_{C^{\mathrm{sm}}}M\). This open subset is dense in \(\operatorname{Tot}(E)\), so the
image of \(\rho\), with its reduced structure, is exactly \(H\). The induced map
\(\rho:\operatorname{Tot}(E)\to H\) is finite and birational. Since \(\operatorname{Tot}(E)\)
is smooth and normal, it is the normalization of \(H\). In particular, \(H\) is closed analytic
and has pure dimension \(1+(n-1)=n\).

Let \(\lambda_M\) be the tautological one-form on \(T^*M\), and put
\(\omega_M=d\lambda_M\). In local coordinates \(x_1,\ldots,x_n\) on \(M\), with corresponding
fiber coordinates \(\xi_1,\ldots,\xi_n\), we have \(\lambda_M=\sum_i\xi_i\,dx_i\). The
definition of \(E\) makes the composite \(E\to g^*\Omega_M^1\to\Omega_{\widetilde C}^1\)
zero. Consequently, on \(\operatorname{Tot}(E)\),
\[
 \rho^*\lambda_M=\sum_i\xi_i\,d(g^*x_i)=0,
 \qquad \rho^*\omega_M=d(\rho^*\lambda_M)=0.
\]
The normalization map \(\rho\) is an isomorphism over the smooth locus of \(H\), where the
local rings are normal. Hence \(\omega_M\) vanishes on that locus. Together with
\(\dim H=n\), this proves that \(H\) is Lagrangian. Each cotangent fiber \(T_x^*M\) also has
dimension \(n\), and \(\lambda_M\) restricts to zero there because the base point is fixed, so
it is Lagrangian as well. Formula~\eqref{eq:curve-supported-ss} expresses \(SS(P)\) as a finite
union of analytic Lagrangian components. Taking the finite union over the perverse constituents
of \(F\) proves the proposition.
\end{proof}

\medskip
\noindent\textbf{Use of AI.}
The content of this article arose from discussions between the author, ChatGPT, and Codex. The
author has rewritten the article and takes responsibility for all of its mathematical content.


\begin{thebibliography}{DLLZb}
\bibitem[BBD82]{BBD}
A.~A.~Beilinson, J.~Bernstein, and P.~Deligne,
\emph{Faisceaux pervers}, Ast\'erisque \textbf{100} (1982),
especially \S6.1.2 and Theorems 6.2.5 and 6.2.10.
\url{https://www.math.tau.ac.il/~bernstei/Publication_list/publication_texts/faisceaux_pervers.pdf}.

\bibitem[Ber93]{BerkovichEtale}
V.~G.~Berkovich, \emph{\'Etale cohomology for non-Archimedean analytic spaces},
Publ. Math. Inst. Hautes \'Etudes Sci. \textbf{78} (1993), 5--161.
\url{https://www.numdam.org/item/PMIHES_1993__78__5_0.pdf}.

\bibitem[BH22]{BH}
B.~Bhatt and D.~Hansen,
\emph{The six functors for Zariski-constructible sheaves in rigid geometry},
Compositio Math. \textbf{158} (2022), 437--482.
\href{https://arxiv.org/abs/2101.09759v2}{arXiv:2101.09759v2},
22 September 2021.
\href{https://doi.org/10.1112/S0010437X22007291}{doi:10.1112/S0010437X22007291}.

\bibitem[BS14]{BS}
B.~Bhatt and P.~Scholze,
\emph{The pro-\'etale topology for schemes}, preprint,
version 2, 17 December 2014.
\href{https://arxiv.org/abs/1309.1198v2}{arXiv:1309.1198v2},
Lemma 6.6.8 and Proposition 6.6.11.

\bibitem[Con99]{ConradIrred}
B.~Conrad, \emph{Irreducible components of rigid spaces},
Ann. Inst. Fourier \textbf{49} (1999), 473--541.
\href{https://doi.org/10.5802/aif.1681}{doi:10.5802/aif.1681}.

\bibitem[Con06]{Conrad}
B.~Conrad, \emph{Relative ampleness in rigid geometry},
Ann. Inst. Fourier \textbf{56} (2006), 1049--1126.
\url{https://math.stanford.edu/~conrad/papers/amplepaperfinal.pdf}.

\bibitem[Del68]{Deligne68}
P.~Deligne, \emph{Th\'eor\`eme de Lefschetz et crit\`eres de
d\'eg\'en\'erescence de suites spectrales},
Publ. Math. IH\'ES \textbf{35} (1968), 107--126.
\href{https://doi.org/10.1007/BF02698925}{doi:10.1007/BF02698925}.

\bibitem[Del94]{Deligne94}
P.~Deligne, \emph{D\'ecompositions dans la cat\'egorie d\'eriv\'ee},
in \emph{Motives}, Proc. Sympos. Pure Math. \textbf{55}, Part 1,
American Mathematical Society, 1994, 115--128.

\bibitem[Hua17]{Huan}
Z.~Huan, \emph{Universal finite subgroup of the Tate curve}, preprint.
\href{https://arxiv.org/abs/1708.08637v3}{arXiv:1708.08637v3}, \S2.

\bibitem[Han20]{HansenCover}
D.~Hansen, \emph{Vanishing and comparison theorems in rigid analytic
geometry}, Compos. Math. \textbf{156} (2020), 299--324,
Theorem 1.6 and Corollary 2.15.
\href{https://doi.org/10.1112/S0010437X19007371}{doi:10.1112/S0010437X19007371}.

\bibitem[Han17]{HansenArtin}
D.~Hansen, \emph{Artin vanishing in rigid analytic geometry},
version 1, 24 August 2017, Introduction, p.~2.
\href{https://arxiv.org/abs/1708.07276v1}{arXiv:1708.07276v1}.

\bibitem[DLLZa]{DLLZFound}
H.~Diao, K.-W.~Lan, R.~Liu, and X.~Zhu,
\emph{Logarithmic adic spaces: some foundational results},
version 2, 30 October 2022.
\href{https://arxiv.org/abs/1912.09836v2}{arXiv:1912.09836v2},
Corollary 6.3.4.

\bibitem[DLLZb]{DLLZLog}
H.~Diao, K.-W.~Lan, R.~Liu, and X.~Zhu,
\emph{Logarithmic Riemann--Hilbert correspondences for rigid varieties},
version 4, 30 October 2022.
\href{https://arxiv.org/abs/1803.05786v4}{arXiv:1803.05786v4},
Corollary 3.2.10.

\bibitem[HZ25]{HZ}
D.~Hansen and B.~Zavyalov,
\emph{Arithmetic properties of \(\ell\)-adic \'etale cohomology
and nearby cycles of rigid-analytic spaces},
version 2, 17 July 2025.
\href{https://arxiv.org/abs/2301.01800v2}{arXiv:2301.01800v2}.

\bibitem[KM85]{KM}
N.~M.~Katz and B.~Mazur, \emph{Arithmetic Moduli of Elliptic Curves},
Annals of Mathematics Studies \textbf{108}, Princeton University Press,
1985, \S8.7.

\bibitem[LRZ24]{LRZ}
S.~Li, E.~Reinecke, and B.~Zavyalov,
\emph{Relative Poincar\'e duality in nonarchimedean geometry}, preprint,
version 1, 10 October 2024.
\href{https://arxiv.org/abs/2410.08200v1}{arXiv:2410.08200v1}.

\bibitem[LZ16]{LiuZhu}
R.~Liu and X.~Zhu,
\emph{Rigidity and a Riemann--Hilbert correspondence for \(p\)-adic
local systems}, version 3, 20 July 2016.
\href{https://arxiv.org/abs/1602.06282v3}{arXiv:1602.06282v3},
Theorem 3.9.

\bibitem[LLZ23]{LLZ}
K.-W.~Lan, R.~Liu, and X.~Zhu,
\emph{De Rham comparison and Poincar\'e duality for rigid varieties},
Peking Math. J. \textbf{6} (2023), 143--216;
preprint version 2, 30 October 2022.
\href{https://arxiv.org/abs/1912.13030v2}{arXiv:1912.13030v2},
\S2.4 and Theorems 3.1.10 and 4.4.1.

\bibitem[LZ26]{LuZheng}
Q.~Lu and W.~Zheng,
\emph{Slopes and weights of \(\ell\)-adic cohomology of rigid spaces},
Adv. Math. \textbf{502} (2026), Part B, 111149.
\href{https://arxiv.org/abs/2502.05421v3}{arXiv:2502.05421v3},
version 3, 9 July 2026.
\href{https://doi.org/10.1016/j.aim.2026.111149}{doi:10.1016/j.aim.2026.111149}.

\bibitem[Ram05]{Ramero}
L.~Ramero, \emph{Local monodromy in non-archimedean analytic geometry},
Publ. Math. IH\'ES \textbf{102} (2005), 167--280.
\url{https://www.numdam.org/item/PMIHES_2005__102__167_0/}.

\bibitem[Sch13]{ScholzeHT}
P.~Scholze, \emph{\(p\)-adic Hodge theory for rigid-analytic varieties},
Forum Math. Pi \textbf{1} (2013), e1.
\href{https://arxiv.org/abs/1205.3463v2}{arXiv:1205.3463v2},
Theorem 1.6 and Corollary 1.8.

\bibitem[Stacks]{Stacks}
The Stacks Project Authors, \emph{The Stacks Project}.
\url{https://stacks.math.columbia.edu/}.

\bibitem[Tem12]{Temkin}
M.~Temkin, \emph{Functorial desingularization of quasi-excellent schemes
in characteristic zero: the nonembedded case},
Duke Math. J. \textbf{161} (2012), 2207--2254.
\href{https://arxiv.org/abs/0904.1592v2}{arXiv:0904.1592v2},
Theorems 1.2.1 and 5.2.2.

\bibitem[Tem17]{TemkinEmbedded}
M.~Temkin, \emph{Functorial desingularization over \(\Q\): boundaries
and the embedded case}, version 3, 21 February 2017.
\href{https://arxiv.org/abs/0912.2570v3}{arXiv:0912.2570v3},
Theorems 1.1.11 and 1.1.13(i).

\bibitem[Zho25a]{ZhouVC}
T.~Zhou, \emph{Vanishing cycles for Zariski-constructible sheaves on
rigid analytic varieties}, preprint, 2025.
\href{https://arxiv.org/abs/2504.16365v1}{arXiv:2504.16365v1}.

\bibitem[Zho25b]{ZhouMicro}
T.~Zhou, \emph{A microlocal theory for Zariski-constructible sheaves on
rigid analytic varieties}, preprint, 2025.
\href{https://arxiv.org/abs/2507.18604v1}{arXiv:2507.18604v1}.

\end{thebibliography}
\end{document}